\documentclass[a4paper]{amsart}

\usepackage{lineno,hyperref,amsmath, amssymb,amsthm,amscd}
\usepackage[all]{xy}
\usepackage{graphics}
\usepackage{color}

\newtheorem{theorem}{Theorem}[section] 
\newtheorem{definition}[theorem]{Definition}
\newtheorem{lemma}[theorem]{Lemma}
\newtheorem{remark}[theorem]{Remark}
\newtheorem{proposition}[theorem]{Proposition}

\newtheorem{example}[theorem]{Example}

\newcommand{\C}{\mathbb{C}}
\newcommand{\R}{\mathbb{R}}

\newcommand{\Exp}{{\rm Exp}}
\newcommand{\Log}{{\rm Log}}

\newcommand{\fa}{\mathfrak{a}}

\newcommand{\fc}{\mathfrak{c}}
\newcommand{\fg}{\mathfrak{g}}

\newcommand{\fk}{\mathfrak{k}}

\newcommand{\ft}{\mathfrak{t}}
\newcommand{\fp}{\mathfrak{p}}
\newcommand{\fq}{\mathfrak{q}}
\newcommand{\fu}{\mathfrak{u}}

\newcommand{\fl}{\mathfrak{l}}

\newcommand{\fw}{\mathfrak{w}}

\newcommand{\wH}{\widetilde{H}}

\newcommand{\da}{\dagger}

\begin{document}

\title[Equivariant  realizations of Hermitian symmetric space of noncompact type]{Equivariant  realizations of Hermitian symmetric space of noncompact type}

\author{Takahiro Hashinaga and Toru Kajigaya}

\address[T. Hashinaga]{National Institute of Technology, Kitakyushu College, Kitakyushu, Fukuoka 802-0985, Japan}
\email{hashinaga@kct.ac.jp}

\address[T. Kajigaya]{Department of Mathematics, Faculty of Science, Tokyo University of Science,
1-3 Kagurazaka Shinjuku-ku, Tokyo 162-8601,  Japan}
 \address{National Institute of Advanced Industrial Science and Technology (AIST), MathAM-OIL, Sendai 980-8577, 
Japan
}
\email{kajigaya@rs.tus.ac.jp}

\subjclass[2010]{Primary 53C35; Secondary  53C55}


\date{\today}
\keywords{}

\begin{abstract}
Let $M=G/K$ be a Hermitian symmetric space of noncompact type. We provide a way of constructing  $K$-equivariant embeddings from $M$ to its tangent space $T_oM$ at the origin by using the polarity of the $K$-action.  As an application, we  reconstruct the $K$-equivariant holomorphic embedding so called the Harish-Chandra realization and the $K$-equivariant symplectomorphism constructed by Di Scala-Loi and Roos under appropriate identifications of spaces.   Moreover, we characterize the holomorphic/symplectic embedding of $M$ by means of the polarity of the $K$-action.  Furthermore, we show a special class of totally geodesic submanifolds in $M$ is realized as either linear subspaces or bounded domains of linear subspaces in $T_oM$ by the $K$-equivariant embeddings.  We also construct a $K$-equivariant holomorphic/symplectic embedding of an open dense subset of the compact dual $M^*$ into its tangent space at the origin as a dual of the holomorphic/symplectic embedding of $M$.

\end{abstract} 

\maketitle

\section{Introduction and main results}

A Hermitian symmetric space is a Riemannian symmetric space equipped with a K\"ahler structure such that the  geodesic symmetry at any point is holomorphic.  Throughout this article, we regard a Hermitian symmetric space $M$ as a homogeneous space $G/K$ associated with a Cartan involution on a connected Lie group $G$. In this paper, we consider  realizations of a Hermitian symmetric space of noncompact type (HSSNT for short) in a vector space of the same dimension.

It is known that any HSSNT is a simply-connected, complete Riemannian manifold of nonpositive sectional curvature. Thus, the Cartan-Hadamard theorem shows that any HSSNT $M$ of real dimension $2n$ is diffeomorphic to the Euclidean space $\R^{2n}$ as a manifold. Indeed, the logarithm map $\Log_o$, i.e. the inverse of the Riemannian exponential map $\Exp_o$ at the origin  $o\in M$ gives a diffeomorphism from $M$ onto the tangent space $T_oM$ at $o$.  

It is also well-known that any HSSNT $M$ is realized as a bounded domain $D$ in the complex Euclidean space $\C^n$ as a complex manifold. Namely, there exists a holomorphic diffeomorphism between $M$ and $D$. While this fact was first verified by \'E. Cartan  in his theory of bounded symmetric domain \cite{Cartan},  the {\it Harish-Chandra realization} is widely known as a canonical construction of the holomorphic diffeomorphism \cite{HC}.  See \cite[\S 7 of Ch. VIII]{Hel} for details of the construction  due to Harish-Chandra (See also Appendix for a brief summary of the Harish-Chandra realization).  

 On the other hand, McDuff \cite{Mc} proved that there exists a symplectomorphism from any complete, simply-connected K\"ahler manifold of nonpositive sectional curvature onto the standard symplectic vector space, namely, the global version of Darboux theorem holds for such a K\"ahler manifold.  In particular, any HSSNT is regarded as the standard symplectic vector space $\R^{2n}$ as a symplectic manifold.  It is an interesting problem to ask the explicit description of the symplectomorphism based on  geometry of Hermitian symmetric spaces.  In fact, an explicit formula of the symplectomorphism  for HSSNT was discovered by Di Scala-Loi \cite[Theorem 1.1]{DL}. Note that their formula was independently defined by Roos \cite[Definition VII.4.1]{Roos} in which he showed that the map preserves the volume form. Although the symplectomorphism from $M$ onto $\R^{2n}$ is not unique (see \cite{DLR}), we call the canonical symplectomorphism given in \cite{DL, Roos} the {\it Di Scala-Loi-Roos realization}  (see \cite{DL, DLR} or Appendix for the precise description of the formula). We remark that their formula was given by the Bergman operator for the associated Hermitian positive Jordan triple system, therefore the HSSNT $M$ is regarded as a bounded symmetric domain $D$.  

The aim of the present paper is to describe  the above realizations in a unified framework based on the Lie theory, and discuss common properties of these maps.  Our starting point is that the above realizations are given by $K$-equivariant embeddings, where $K$ is the isotropy subgroup at the origin $o\in M=G/K$ which is in fact a maximal compact subgroup of $G={\rm Isom}(M)_0$.  This fact is easy to see for the map $\Log_o$, and implied by \cite[Corollary 7.17]{Hel}, \cite[Section 3 in Part I]{Wolf} for the Harish-Chandra realization and proved in \cite[Theorem 1.1 (I)]{DL} for the Di Scala-Loi-Roos realization. In this paper, we first investigate a way to construct $K$-equivariant embeddings from $M$ to its tangent space $T_oM$ by using the {\it polarity} of the $K$-action.

 \subsection{Polarity and Polydisk theorem}
We recall two basic facts on HSSNT.
Let $M=G/K$ be an irreducible Hermitian symmetric space of noncompact type. We denote the standard complex structure and the K\"ahler form on $M$ by $J$ and $\omega$, respectively. The Lie algebras of $G$ and $K$ are denoted by $\fg$ and $\fk$, respectively. Then, we have the associated Cartan decomposition $\fg=\fk\oplus \fp$ with respect to the Cartan involution $\theta$.  The subspace $\fp$ is identified with the tangent space $T_oM$ at the origin $o\in M$. The complex structure and the symplectic structure on $\fp\simeq T_oM$ are denoted by $J_o$ and $\omega_o$, respectively. Then, $(\fp, J_o, \omega_o)$ can be identified with the complex Euclidean space $\C^n$.  Also, we denote the exponential map at $o$ as a Riemannian manifold by $\Exp_o: \fp\to M=G/K$. Note that  $\Exp_o$ is a diffeomorphism since $M$ is a symmetric space of noncompact type.

Let $\fa$ be a maximal abelian subspace in $\fp$. The dimension of $\fa$ is  called the {\it rank} of $M$, and we denote the rank by $r$.  We put $A:=\Exp_o\fa$. Then, $A$ is a totally geodesic flat submanifold in $M$.   The following theorem is a well-known fact for Riemannian symmetric spaces (see \cite[Theorem 2.3.15]{BCO}):
 
 \begin{theorem}[Polarity of $K$-action]\label{Thm:polar}
 The $K$-action on $\fp$ via the isotropy representation is a polar action and any maximal abelian subspace $\fa$ in $\fp$ becomes a section, i.e. any ${\rm Ad}(K)$-orbit in $\fp$ intersects $\fa$ orthogonally. In particular, we have
\begin{align*}
{\rm Ad}(K)\fa=\fp \quad {\rm and}\quad K\cdot A=M.
\end{align*}
\end{theorem}
 
Throughout this article, we say a maximal abelian subspace $\fa$ (resp. $A$) in $\fp$ a {\it section} of the $K$-action on $\fp$ (resp. $M$).

The submanifold $A$ has a complexification in the following sense (e.g. see \cite[Exercises B-2 in  \S 7 of ChVIII]{Hel}, \cite[Part I]{Wolf}). We put $\fa^\C:=\fa\oplus J_o\fa$ and $A^{\C}:=\Exp_o\fa^\C$.  
     
\begin{theorem}[Polydisk theorem] \label{thm:polydisk}
The submanifold $A^{\C}$ is a  totally geodesic complex submanifold in $M$, and there exists a canonical splitting 
\[
A^{\C}\ \xrightarrow{\sim}\ A_1^{\C}\times \cdots \times A_{r}^{\C}
\]
 into a direct product of $r$ totally geodesic complex submanifolds  in $M$. Moreover, each totally geodesic complex submanifold $A_i^{\C}$ equipped with the induced K\"ahler structure is holomorphically isometric to a complex hyperbolic plane $\C H^1(-C)$ of constant holomorphic sectional curvature $-C$, and the  identification
 \[
A^{\C}\ \xrightarrow{\sim}\ \underbrace{\C H^1(-C)\times \cdots \times\C H^1(-C)}_{\textup{$r$-times}}
\]
 is a holomorphic isometry, where the product manifold inherits the product K\"ahler structure. 
 \end{theorem}
 
 To be precise, the splitting described in this theorem is obtained by strongly orthogonal roots consisting of long roots in the restricted root system.  See  Section \ref{Prelim} for details and a proof of this theorem  using the restricted root system.
 
 In the following, we identify  $\C H^1(-C)$ with the Hermitian symmetric space $\widehat{G}/\widehat{K}=SU(1,1)/S(U(1)\times U(1))$.  We denote the associated Cartan decomposition by $\widehat{\fg}=\widehat{\fk}\oplus \widehat{\fp}$, where $\widehat{\fg}$ (resp. $\widehat{\fk}$) is the Lie algebra of $\widehat{G}$ (resp. $\widehat{K}$).  The subspace $\widehat{\fp}$ is  naturally identified with $\C$, and we denote the element in $\widehat{\fp}$ by $z\in \C$ via the identification.

 \subsection{Main results}\label{const}
 
 Based on Theorem \ref{Thm:polar},   we consider a natural way to construct a $K$-equivariant map $\Omega: M\to \fp$ as follows: We  take a  map 
 \[
 \Omega_A: A\to \fa
 \]
 from section to section,  and  consider a map
 \[
 \Omega(k\Exp_ov)={\rm Ad}(k)\Omega_{A}(\Exp_ov)
 \]   
 for $k\in K$ and $v\in \fa$. It is easy to see that $\Omega$ is a $K$-equivariant map if $\Omega$ is well-defined.
 We observe that  the map $\Omega$ is well-defined if and only if $\Omega_A$ is equivariant under the action of the Weyl group $\mathcal{W}$ with respect to the restricted root system of $M$ (Lemma \ref{Lemma:welldef}).  Moreover,  by using an explicit formula of the Weyl group associated with HSSNT, we determine the $K$-equivariant map $\Omega$ in an algebraic way (Proposition \ref{Prop:omh}).  As a special case, we obtain the following simple way of construction of $K$-equivariant embeddings satisfying $\Omega(A)\subseteq \fa$ (see Section \ref{Section:equiv} for details):
 \begin{itemize}
\item First, we take a $\widehat{K}$-equivariant embedding  from $\C H^1(-C)=\widehat{G}/\widehat{K}$ into $\widehat{\fp}$ such that $\widehat{A}=\Exp_o\widehat{\fa}$ is mapped into $\widehat{\fa}$, where $\widehat{\fa}$ is a section in $\widehat{\fp}$.  Such an embedding is always obtained by a ``radial" map
 \[
 \Omega_{\eta, 0}: \C H^1(-C)\to \widehat{\fp},\quad \Exp_oz\mapsto \eta(|z|)\frac{z}{|z|}
 \]
 with $ \Omega_{\eta, 0}(o)=0$ for an injective {\it odd} function $\eta: \R\to \R$.
  \item Next, by using the canonical identification $A^\C\simeq \C H^1(-C)\times \cdots \times\C H^1(-C)$, we extend $\Omega_{\eta,0}$ to an embedding of the polydisk $A^\C$ into its tangent space at the origin $\fa^\C=\fa\oplus J_o\fa$ by the direct product
\[
\Omega_{\eta, A^\C}:=\Omega_{\eta, 0}\times \cdots \times \Omega_{\eta, 0}: A^\C\to \fa^\C.
\] 
Then,  the restricted map $\Omega_{\eta, A}:=\Omega_{\eta, A^\C}|_{A}$ sends $A$ into $\fa$ and  $\Omega_{\eta, A}$ is $\mathcal{W}$-equivariant. 
\item Extending the map $\Omega_{\eta, A}$ (or equivalently, $\Omega_{\eta, A^\C}$) by the $K$-action,  we finally obtain a well-defined $K$-equivariant embedding as follows:
\[
\Omega_{\eta}: M\to \fp,\quad k\Exp_ov\mapsto {\rm Ad}(k)\Omega_A(\Exp_ov)
\]
for $k\in K$ and $v\in \fa$. By definition, it holds that $\Omega_{\eta}(A)\subseteq \fa$.

\end{itemize}

We call the resulting embedding $\Omega_{\eta}$ a {\it strongly diagonal realization} of $M$ associated with $\eta$ (Definition \ref{Def:diag}).  We remark that a map  corresponding to $\Omega_{\eta}$ has been considered in \cite[Section 1.6]{DLR}, \cite{LM} and \cite[Section 3.18]{Loos}  by using the Jordan triple system. Our description above is regarded as a geometric interpretation of it, although our terminology used in the present paper differs from theirs.

For any strongly diagonal realization $\Omega_{\eta}: M\to \fp$,  it immediately  follows from definition  that the image $\Omega_{\eta}(M)$ is  either $\fp$ or a $K$-invariant bounded domain $D_{\eta}$ in $\fp$, and the section $A$ is mapped onto either $\fa$ or a cube $\square_{\eta, \fa}$ in $\fa$ such that 
\[
\Omega_{\eta}(M)=D_{\eta}={\rm Ad}(K)(\square_{\eta,\fa}).
\]
  This situation gives a generalization of the Harish-Chandra realization (see \cite[ Corollary 7.17]{Hel}).   
  
The simplest example of the strongly diagonal realization is the map $\Log_o$. Indeed, $\Log_o: M\to \fp$ is recovered by taking $\Omega_{\eta,0}=\Log_o^{\C H^1}: \C H^1(-C)\to \widehat{\fp}$ (or equivalently, $\eta(x)=x$), that is, the Riemannian logarithm map of $\C H^1(-C)$.  Moreover, we obtain a canonical holomorphic/symplectic realization of $M$ as a strongly diagonal realization:  If $\Omega_{\eta}: M\to \fp$ is holomorphic (or symplectic), it is necessary to assume that so is $\Omega_{\eta, 0}: \C H^1(-C)\to \widehat{\fp}$ since $\Omega_{\eta}|_{A_i^\C}$ is identified with $\Omega_{\eta,0}$. By solving an ordinary differential equation,  we see that the holomorphic $\Omega_{\eta, 0}$ is uniquely determined up to constant multiple, and the associated odd function $\eta$ is given by $\eta(x)=\tanh x$. Similarly, the symplectic  $\Omega_{\eta, 0}$ is uniquely determined up to sign, and  the associated odd function is given by $\eta(x)=\sinh x$.  Conversely,  if $\Omega_{\eta, 0}$ is holomorphic/symplectic, then $\Omega_{\eta, 0}$ is extended to a $K$-equivariant holomorphic/symplectic embedding  of $M$:

\begin{theorem}[Theorem \ref{mainthm1b}]\label{mainthm1}
Let $M$ be an irreducible Hermitian symmetric space of noncompact type equipped with the standard K\"ahler structure $(J, \omega)$. Then, the map $\Psi:=\Omega_{\tanh}: (M, J) \to (D_{\tanh}, J_o)$  is a $K$-equivariant holomorphic diffeomorphism, and
 the map $\Phi:=\Omega_{\sinh}: (M,\omega)\to (\fp,\omega_o)$  is a $K$-equivariant symplectomorphism.
  
 Moreover, $\Psi$ is the unique (up to constant multiple) $K$-equivariant holomorphic embedding from $(M, J)$ to $(\fp, J_o)$ such that the section $A=\Exp_o\fa$ of $M$ is mapped into the section $\fa$ of $\fp$,  and  $\Phi$ is the unique (up to sign) $K$-equivariant symplectic embedding from $(M,\omega)$ to $(\fp, \omega_o)$ such that  $A$ is mapped into $\fa$.
\end{theorem}

We give a proof of this theorem in Section \ref{Proof:mainthm1}. Our proof of the first part of Theorem \ref{mainthm1} provides alternative proofs of the results due to  \'E. Cartan, Harish-Chandra and Di Scala-Loi. In fact, it turns out that  the holomorphic diffeomorphism  $\Psi$ and the symplectomorphism $\Phi$ coincide with the Harish-Chandra realization and  the  Di Scala-Loi-Roos realization, respectively, under appropriate identifications of spaces. We clarify the relation to the known results in Appendix.

On the second part of Theorem \ref{mainthm1}, we remark that  the uniqueness may not hold in general if we drop the assumption that the section $A$ is mapped into the section $\fa$.  In fact,  in \cite{DLR}, Di Scala-Loi-Roos constructed  symplectomorphisms from a bounded symmetric domain onto the standard symplectic vector space satisfying a special property so called the {symplectic duality} (see \cite{DL, DLR} for details). Their symplectomorphisms are actually $K$-equivariant, and hence, the uniqueness of $K$-equivariant symplectomorphism from $M$ onto $\R^{2n}$ does not hold in general.  However,  our result shows that the holomorphic embedding $\Psi$ and the symplectomorphism $\Phi$ are characterized  by means of the polarity of the $K$-action, namely, $\Psi$ (resp. $\Phi$) is the unique $K$-equivariant holomorphic (resp. symplectic) embedding such that the section $A$ is mapped into the section $\fa$.  Regarding $\Psi$ as the Harish-Chandra realization and $\Phi$ as the Di Scala-Loi-Roos realization, this gives a characterization of the canonical holomorphic/symplectic realization.

 As another application, we show the above construction of $K$-equivariant maps can be  applied to the compact dual $M^*=G^*/K$ of $M$ with a slight modification. Namely, analogous to the strongly diagonal realization of $M$, we obtain a $K$-equivariant map from a $K$-invariant neighborhood of the origin $o^*\in M^*$ contained in $(M^*)^o=M^*\setminus {\rm Cut}_{o^*}(M^*)$ into its tangent space $\fp^*\simeq T_{o^*}M^*$, where ${\rm Cut}_{o^*}(M^*)$ is the cut locus of $M^*$ at the origin $o^*$.  In particular, we define a notion of {\it dual map} $\Omega_{\eta^*}^*$ of the strongly diagonal realization $\Omega_{\eta}$ of $M$ when $\eta: \R\to\R$ is a real analytic function  (Definition \ref{Def:dual}). More precisely, by taking a holomorphic function $\eta^\C$ so that $\eta^\C(x)=\eta(x)$ for $x\in \R\subset \C$, we define the {\it dual function} $\eta^*$ of $\eta$ by $\eta^*(x):=-\sqrt{-1}\eta^\C(\sqrt{-1}x)$. The dual function $\eta^*$ is also a real odd function, and we may assume that $\eta^*$ is defined on an interval $(-R^*, R^*)$ for some $R^*\in (0,\pi/2]$. Then, the dual  map $\Omega_{\eta^*}^*$ associated with $\eta^*$ is given by a $K$-equivariant map 
 \[
 \Omega_{\eta^*}^*: U_{\eta^*}^*\to \fp^*,
 \]
 where $U_{\eta^*}^*$ is a $K$-invariant connected open subset of $(M^*)^o$ around $o^*\in M^*$ depending on the domain of $\eta^*$ (see subsection \ref{subsec:dual} for the precise definition). Similar to the strongly diagonal realization of $M$, we see the image $\Omega_{\eta^*}^*(U^*)$ is either $\fp^*$ or a bounded domain $D_{\eta^*}^*$ in $\fp^*$.

For example, the dual function of the identity map $id$ is given by itself and the dual map of $\Omega_{id}={\rm Log}_o: M\to \fp$ coincides with the logarithm map ${\rm Log}_{o^*}^*: (M^*)^o\to \fp^*$ at $o^*\in M^*$. On the other hand, the dual functions of $\eta(x)=\tanh x$ and  $\sinh x$ are given by $\eta^*(x)=\tan x$ and $\sin x$, respectively. Moreover, we see the following:

\begin{theorem}[Theorem \ref{mainthm2b}]\label{mainthm2}
Let $M$ be an irreducible Hermitian symmetric space of noncompact type, and $M^*$  its compact dual.  We put $(M^*)^o:=M^*\setminus {\rm Cut}_{o^*}(M^*)$. Then, the dual map $\Psi^*=\Omega_{\tan}^*$ of $\Psi$ is a $K$-equivariant holomorphic diffeomorphism from $(M^*)^o$ onto $\fp^*$, and the dual map $\Phi^*=\Omega_{\sin}^*$ of $\Phi$ is a $K$-equivariant symplectomorphism from $(M^*)^o$ onto a bounded domain $D_{\sin}^*$ in $\fp^*$.
\end{theorem}

We give some remarks: 
\begin{itemize}
\item[(1)] As an immediate consequence of Theorems \ref{mainthm1} and \ref{mainthm2},   we  obtain a $K$-equivariant holomorphic embedding $h$ from $M$ into $M^*$ and a symplectic embedding $s$ from $(M^*)^o$ into $M$ by composition maps  (Remark \ref{Rem:holsymp})
\[
h:=(\Psi^*)^{-1}\circ \Psi: M\to M^*, \quad s:=\Phi^{-1}\circ \Phi^*: (M^*)^o\to M,
\]
where we used the canonical identification $\sqrt{-1}:\fp\simeq \fp^*$.
It turns out that $h$ coincides with the well-known {\it Borel embedding}, up to ``rotation" on $(M^*)^o$ (see Appendix \ref{A1} for details). Thus, this provides another description of the canonical holomorphic embedding of $M$ into $M^*$, and the symplectic embedding $s$ may be regarded as a symplectic counterpart of it.

\item[(2)] The map $\Omega_{\eta}:M\to \fp$ can be regarded as a global chart of $M$, and the dual map $\Omega_{\eta^*}:U_{\eta^*}^*\to \fp^*$ can be regarded as a local chart of $M^*$ around the origin. 
The induced K\"ahler structure on an open dense subset in $D_{\eta}=\Omega_{\eta}(M)\subset \fp$ and $D_{\eta^*}^*=\Omega^*_{\eta^*}((M^*)^o)\subset \fp^*$ are explicitly expressed  by using the restricted root system.  See Propositions \ref{Prop:JO} and \ref{Prop:JOdual}.

\end{itemize}

We mention that the fact that $\Phi$ and $\Phi^*$ are {\it dual} symplectomorphisms of each other is actually equivalent to a remarkable property of the Di Scala-Loi-Roos realization so called the {\it symplectic duality} \cite[Theorem 1.1 (S)]{DL}. More precisely, in \cite{DL}, Di Scala-Loi proved that the Di Scala-Loi-Roos realization $\widetilde{\Phi}: D\to \fp$ is  a simultaneous symplectomorphism between different symplectic structures, i.e. it holds that $\widetilde{\Phi}^*\omega_o=\omega_{hyp}$ and $\widetilde{\Phi}^*\omega_{FS}=\omega_{o}$, where $\omega_{hyp}$ is the induced K\"ahler form on the bounded domain $D$ from $M$, and $\omega_{FS}$ is the induced K\"ahler form on $\fp^*\simeq \fp$ from the Fubini-Study metric on $M^*$ in a canonical way (see \cite{DL} or Remark \ref{Rem:dual} for details).  Because of this property, the Di Scala-Loi-Roos realization $\widetilde{\Phi}$ is also referred as the {\it symplectic duality map} (cf. \cite{LM}). See \cite{DL, DLR, LM} for further discussion on the symplectic duality.  We see that, by regarding $\Phi$ as $\widetilde{\Phi}$, the symplectic duality follows from the fact that $\Phi$ and $\Phi^*$ are dual symplectomorphisms of each other, and moreover, a holomorphic version of \cite[Theorem 1.1 (S)]{DL} also follows from the fact that $\Psi$ and $\Psi^*$ are {dual} holomorphic diffeomorphisms of each other. See Remark \ref{Rem:dual}.

We also remark  that another interesting example of the pair of strongly diagonal realization and its dual map has been given by Loi-Mossa \cite{LM} by using  the Jordan triple system. We adapt their formula given in \cite{LM} to our formulation in Appendix.

In Section \ref{Section:property}, we consider realizations of totally geodesic submanifolds in $M$ by strongly diagonal realizations.  In \cite{Ciriza}, Ciriza proved that the symplectomorphism constructed by McDuff sends any  totally geodesic  {\it complex} submanifold through the origin to a complex linear subspace in $\C^n$.  Di Scala-Loi also proved that the Di Scala-Loi-Roos realization satisfies the same property \cite[Theorem 1.1 (H)]{DL}.  
Moreover, they showed that  the Di Scala-Loi-Roos realization  $\widetilde{\Phi}$ satisfies $\widetilde{\Phi}|_N=\widetilde{\Phi}_N$ for any  totally geodesic complex submanifold $N$ in $M$ through the origin, where $\widetilde{\Phi}_N$ is the symplectomorphism with respect to $N$ in the sense of \cite{DL}.  In the present paper, we generalize Di Scala-Loi's result to any strongly diagonal realization, and furthermore, we show a wider class of totally geodesic submanifolds is realized as  linear subspaces in $\fp$ (or  bounded domains of linear subspaces in $D_{\eta}$)  by strongly diagonal realizations.

 The results in Section \ref{Section:property} are summarized as follows. 
 First, we introduce a simple property for abelian subspaces in $\fp$. We say an abelian subspace $\fa'$ in $\fp$ has a {\it complexification as Lie triple system} (complexification as LTS for short)  in $\fp$ if $\fa'\oplus J_o\fa'$ becomes a complex Lie triple system in $\fp$, namely, the submanifold $(A')^\C:=\Exp_o(\fa'\oplus J_o\fa')$ is a complex totally geodesic submanifold in $M$. It turns out that $\fa'$ has a complexification as LTS if and only if $(A')^\C$ splits into a product of $\C H^1$ as a K\"ahler manifold (Proposition \ref{keyprop1}), and hence, this gives a generalization of the polydisk $A^\C$.
 
 Next, we consider a complete totally geodesic submanifold $N$ through the origin in $M$ whose maximal abelian subspace $\fa_N$ of $\fp_N$ has a complexification as LTS in $\fp$, where we regard $N$ as a Riemannian symmetric space $N=G_N/K_N$ and we denote the associated  Cartan decomposition by $\fg_N=\fk_N\oplus \fp_N$.   If $\fa_N$ has a complexification as LTS in $\fp$,  we obtain a $K_N$-equivariant embedding $\Omega_{\eta, N}: N\to \fp_N$ analogous to the strongly diagonal realization for any  injective odd function $\eta: \R\to \R$.  It also follows that the image $\Omega_{\eta, N}(N)$ coincides with either $\fp_N$ or a bounded domain $D_{\eta,N}$ of $\fp_N$ and $A_N=\Exp_o\fa_N$ is mapped onto  either $\fa_N$ or a cube $\square_{\eta, \fa_N}$ in $\fa_N$, where the pair $(D_{\eta, N}, \square_{\eta, \fa_N})$ is canonically defined by $\eta$, $N$ and $A_N$ for any totally geodesic submanifold $N$ through the origin (see Subsection \ref{subsec:reali}).
 
  The main result in Section 5 is the following, which is a generalization of \cite[Theorem 1.1 (H)]{DL}.  
  
 \begin{theorem}[Theorem \ref{mainthm3b}]\label{mainthm3}
Let $M=G/K$ be an irreducible  Hermitian symmetric space of noncompact type and $N$  a complete totally geodesic submanifold  in $M$ through the origin.  If $\fa_N$ has a complexification as LTS in $\fp$,  it holds that $\Omega_{\eta}|_N=\Omega_{\eta, N}$ for any injective odd function $\eta:\R\to \R$. In particular,  we have the following commutative diagram:
\[
 \begin{CD}
     (M, A) @>{\Omega_{\eta}\times \Omega_{\eta}|_{A}}>{\textup{diffeo.}}> (D_{\eta}, \square_{\eta,\fa})\\
  @AAA    @AAA \\
    (N, A_N)   @>{\Omega_{\eta,N}\times \Omega_{\eta,N}|_{A_N}}>{\textup{diffeo.}}> (D_{\eta, N}, \square_{\eta,\fa_N})
  \end{CD}
 \]

Conversely, if $(N, A_N)$ is mapped onto $(D_{\eta, N}, \square_{\eta, \fa_N})$ by $\Omega_{\eta}\times \Omega_{\eta}$  for any injective odd function $\eta: \R\to \R$, then $\fa_N$ has a complexification as LTS in $\fp$.
\end{theorem}

 We show that some special classes of totally geodesic submanifolds in HSSNT such as totally geodesic submanifolds of maximal rank,  totally geodesic complex submanifolds and real forms (i.e. totally geodesic Lagrangian submanifolds) provide examples of totally geodesic submanifold whose maximal abelian subspace has a complexification as LTS in $\fp$ (Example \ref{ex:CLTS}).  Thus, as a consequence of Theorem \ref{mainthm3}, these submanifolds are always realized as either linear subspaces or bounded domains of these in $\fp$ for any strongly diagonal realization $\Omega_{\eta}$.

\subsection{Organization}
The paper is organized as follows.  In Section \ref{Prelim}, we give a preliminary on the restricted root system associated with HSSNT, and prove some basic facts on the K\"ahler structure of HSSNT. In particular, we prove the polydisk theorem by using the restricted root system.  In Section \ref{Section:equiv}, we consider a way of constructing $K$-equivariant embeddings from $M$ to $\fp$. We introduce the strongly diagonal map of $M$ and its analogous map for the compact dual $M^*$. Moreover, we define the notion of duality of the maps.  In Section \ref{Section:holsymp},  we  give proofs of Theorems \ref{mainthm1} and \ref{mainthm2}. In Section \ref{Section:property}, we consider realizations of totally geodesic submanifolds in HSSNT by strongly diagonal realizations, and give a proof of Theorem \ref{mainthm3}. In Appendix, we confirm the holomorphic (resp. symplectic) embedding constructed in Theorem \ref{mainthm1} coincides with the Harish-Chandra (resp. Di Scala-Loi-Ross) realization under appropriate identifications of spaces. Moreover, we mention a relation to the example given by Loi-Mossa \cite{LM}.

\section{Preliminaries on restricted root system}\label{Prelim}
In this section, we recall and prove some basic facts on Hermitian symmetric spaces of non-compact type by using the restricted root system.  See \cite{BCO, FT, Hel} for the theory of symmetric spaces and facts on the restricted root system. 

\subsection{Basic facts on restricted root system}\label{Prelim:root}

Let $M=G/K$ be an irreducible Riemannian symmetric space of non-compact type, where $G={\rm Isom}(M)_0$ is the connected component of ${\rm Isom}(M)$ containing the identity map and $K$ is the stabilizer subgroup at $o\in M$. It is known that $G$ is a real simple Lie group and $K$ is a maximal compact subgroup of $G$. We denote the Lie algebras of $G$ and $K$ by $\fg$ and $\fk$, respectively. Let $\fg=\fk\oplus\fp$ be  a Cartan decomposition associated with a Cartan involution $\theta$ on $\fg$, i.e. $\fk$ (resp. $\fp$) is the $+1$ (resp. $-1$)-eigenspace of $\fg$.  We denote the Killing-Cartan form on $\fg$ by $B(\cdot, \cdot)$. Then, the bilinear form $\langle X, Y\rangle:=-B(X,\theta Y)$ for $X, Y\in \fg$ defines a positive definite inner product on $\fg$. Note that the decomposition $\fg=\fk\oplus\fp$ is orthogonal with respect to $\langle, \rangle$.  
The subspace $\fp$ is  identified with the tangent space $T_oM$ at the origin $o\in G/K$ via the natural projection $\pi: G\to G/K$, and the restriction of $\langle, \rangle$ onto $\fp\simeq T_oM$ is extended to a $G$-invariant Riemannian metric $g$ on $M$. It is known that $(M, g)$ is a simply-connected, complete Riemannian manifold of non-positive sectional curvature.

Let $\fa$ be a maximal abelian subspace in $\fp$. Then,  we obtain the following direct sum decomposition of $\fg$ so called the {\it restricted root decomposition} with respect to $\fa$:
\begin{align*}
\fg=\fg_0\oplus \bigoplus_{\alpha\in \Sigma}\fg_{\alpha},
\end{align*}
where $\alpha: \fa\to \R$ is an element in the dual space of $\fa$ which we call a {\it (restricted) root}, $\fg_{\alpha}=\{X\in \fg: {\rm ad}(H)X=\alpha(H)X\ \forall H\in \fa\}$ and we denote the set of non-zero roots by $\Sigma$. The dimension $m_{\alpha}:={\rm dim}_{\R}\fg_{\alpha}$ is called the {\it multiplicity} of the root $\alpha\in \Sigma$. Note that we have $\fg_0=\fk_0\oplus\fa$, where $\fk_0=\{X\in \fk: [X, H]=0\ \forall H\in \fa\}$ is the centralizer of $\fa$ in $\fk$ (cf. \cite[Section 13.1]{BCO}).  Moreover, we have $[\fg_{\alpha}, \fg_{\beta}]\subset \fg_{\alpha+\beta}$ for any $\alpha, \beta\in \Sigma$.
 
 For any $\alpha\in \Sigma$, it is easy to see that $\theta\fg_{\alpha}=\fg_{-\alpha}$ and $\theta|_{\fg_{\alpha}}: \fg_{\alpha}\to \fg_{-\alpha}$ is an isomorphism. Thus, we put
$
X_{-\alpha}:=\theta X_{\alpha}\in \fg_{-\alpha}
$
for $X_{\alpha}\in \fg_{\alpha}$. Moreover, we set
\[
X_{\alpha}^{\fk}:=\frac{1}{2}(X_{\alpha}+X_{-\alpha})\in \fk,\quad {\rm and}\quad X_{\alpha}^{\fp}:=\frac{1}{2}(X_{\alpha}- X_{-\alpha})\in \fp
\]
so that $X_{\alpha}=X_{\alpha}^\fk+X_{\alpha}^\fp$. Letting 
\[
\fk_{\alpha}:=\fk\cap(\fg_{\alpha}\oplus\fg_{-\alpha})\quad  {\rm and}\quad \fp_{\alpha}:=\fp\cap(\fg_{\alpha}\oplus\fg_{-\alpha}),
\]
we have $\fk_{\alpha}=\fk_{-\alpha}$, $\fp_{\alpha}=\fp_{-\alpha}$ for any $\alpha\in \Sigma$, and it follows that $\fk_{\alpha}\oplus \fp_{\alpha}=\fg_{\alpha}\oplus\fg_{-\alpha}$. In particular, we have orthogonal decompositions
\begin{align*}
\fk=\fk_0\oplus\bigoplus_{\alpha\in \Sigma_+}\fk_{\alpha},\quad \fp=\fa\oplus\bigoplus_{\alpha\in \Sigma_+}\fp_{\alpha},
\end{align*}
where $\Sigma_+$ is the set of positive roots with respect to an ordering on $\fa^*$ which is determined by simple roots of $\Sigma$. Since $X_{\alpha}\mapsto X_{\alpha}^{\fk}$ (resp. $X_{\alpha}\mapsto X_{\alpha}^{\fp}$) gives rise to an isomorphism between $\fg_{\alpha}$ and $\fk_{\alpha}$ (resp. $\fp_{\alpha}$),  we often denote an element  in $\fk_{\alpha}$ (resp. $\fp_{\alpha}$) by $X_{\alpha}^{\fk}$ (resp. $X_{\alpha}^{\fp}$)  by using the element $X_{\alpha}\in \fg_{\alpha}$.

Putting $\fp_0=\fa$, one verifies that 
\begin{align}\label{brakp}
[\fk_{\alpha}, \fk_{\beta}]\subset \fk_{\alpha+\beta}\oplus \fk_{\alpha-\beta}, \quad [\fk_{\alpha}, \fp_{\beta}]\subset \fp_{\alpha+\beta}\oplus \fp_{\alpha-\beta},\quad [\fp_{\alpha}, \fp_{\beta}]\subset \fk_{\alpha+\beta}\oplus \fk_{\alpha-\beta}
\end{align}
for any $\alpha, \beta\in \Sigma\cup \{0\}$.
Moreover, it holds that
\begin{align}\label{feq2}
[H, X_{\alpha}^{\fk}]=\alpha(H)X_{\alpha}^{\fp},\quad [H, X_{\alpha}^{\fp}]=\alpha(H)X_{\alpha}^{\fk}
\end{align}
for any $\alpha\in \Sigma$, $H\in \fa$ and $X_{\alpha}\in \fg_{\alpha}$.
Note that, we have $\|X_{\alpha}^{\fk}\|^2=\|X_{\alpha}^{\fp}\|^2$. Indeed, by taking an element $H\in \fa$ so that $\alpha(H)\neq 0$,  \eqref{feq2} implies
\begin{align*}
\|X_{\alpha}^{\fk}\|^2&=-B(X_{\alpha}^{\fk}, \theta X_{\alpha}^{\fk})=-\frac{1}{\alpha(H)}B([H, X_{\alpha}^{\fp}], X_{\alpha}^{\fk})\\
&=\frac{1}{\alpha(H)}B(X_{\alpha}^{\fp}, [H, X_{\alpha}^{\fk}])=B(X_{\alpha}^{\fp}, X_{\alpha}^{\fp})=\|X_{\alpha}^{\fp}\|^2.\nonumber
\end{align*}

Since $B$ is non-degenerate on $\fa\times \fa$, we can define the dual vector $H_{\alpha}\in \fa$ of $\alpha\in \Sigma$ with respect to $B|_{\fa}$, i.e. $H_{\alpha}$ is defined by $\alpha=B(H_{\alpha}, \cdot)|_{\fa}$.  We call $H_{\alpha}$  the {\it root vector} of $\alpha\in \Sigma$. Note that, we have
\begin{align}\label{feq3}
H_{\alpha}=\frac{[X_{\alpha}, X_{-\alpha}]}{B(X_{\alpha}, X_{-\alpha})}=-\frac{[X_{\alpha}, X_{-\alpha}]}{\|X_{\alpha}\|^2}
\end{align}
for any $X_{\alpha}\in \fg_{\alpha}\setminus \{0\}$ and $\alpha\in \Sigma$.
Indeed,  
\[
B([{X}_{\alpha}, {X}_{-\alpha}], H)=-B(X_\alpha, [H, X_{-\alpha}])=\alpha(H)B(X_\alpha, X_{-\alpha})=B(B(X_\alpha, X_{-\alpha})H_{\alpha}, H)
\]
for any $H\in \fa$, and hence, we have $[X_\alpha, X_{-\alpha}]=B(X_\alpha, X_{-\alpha})H_{\alpha}=-\|X_{\alpha}\|^2H_{\alpha}$.

For the maximal abelian subspace $\fa$ in $\fp$, we set  $N_{\fa}(K):=\{k\in K; {\rm Ad}(k)\fa\subset\fa\}$ and $C_{\fa}(K):=\{k\in K; {\rm Ad}(k)v=v,\ \forall v\in \fa\}$.  Then, the {\it Weyl group} $\mathcal{W}$ is defined by the factor group $\mathcal{W}:=N_{\fa}(K)/C_{\fa}(K)$ (see \cite[Section 2 in Ch. VII]{Hel} for details of the Weyl group). 
It is well-known that $\mathcal{W}$ is generated by the reflection $\rho_{\alpha}: \fa\to \fa$ for $\alpha\in \Sigma$, that is, the reflection with respect to the hyperplane $\Pi_{\alpha}:=\{H\in \fa; \alpha(H)=0\}$.  
More precisely, the reflection is expressed by
\begin{align}\label{reflect}
\rho_{\alpha}(H)=H-2\frac{\alpha(H)}{\alpha (H_{\alpha})}H_{\alpha}\quad {\rm for}\ H\in \fa.
\end{align}

\subsection{Hermitian symmetric space}
A Riemannian symmetric space $M$ is said to be a {\it Hermitian symmetric space}  if $M$ admits a K\"ahler structure such that the geodesic symmetry at any point becomes a holomorphic isometry. In the following, we suppose $M$ is a Hermitian symmetric space of noncompact type (HSSNT for short).

We use the following characterization of Hermitian symmetric spaces. One verifies this  by the classification result, e.g. see \cite[Section 13.1]{BCO}. A classification-free proof for symmetric spaces of {\it compact} type is given in \cite{NT}.

\begin{proposition}\label{Prop: Hermitian}
An irreducible Riemannian symmetric space of noncompact type is a Hermitian symmetric space if and only if {\rm (i)} the restricted root system $\Sigma$  is either type $C$ or  $BC$ and {\rm (ii)} the multiplicity $m_{\alpha}$ of the highest root $\alpha$ is equal to  $1$.
\end{proposition}

More precisely, if $M$ is an irreducible Hermitian symmetric space,  we may assume that there exists an orthogonal basis $\{e_i\}_{i=1}^r$ of the dual space of $\fa$ such that $\Sigma$ is expressed by either
\begin{align*}
\Sigma_{C_r}&=\{\pm(e_i+e_j),\ \pm(e_i-e_j):\ 1\leq i<j\leq r\}\cup\{\pm2e_i: 1\leq i\leq r\},\ {\rm or}\\
\Sigma_{BC_r}&=\{\pm(e_i+e_j),\ \pm(e_i-e_j):\ 1\leq i<j\leq r\}\cup\{\pm e_i,\ \pm2e_i: 1\leq i\leq r\},
\end{align*}
where  $r={\rm dim}_{\R} \fa$, and $r$ is called the {\it rank} of $M$.
We exhibit the  list of irreducible Hermitian symmetric spaces of noncompact type and the type of associated restricted root system in Table 1 (see also \cite[Section 13.1]{BCO}).

\begin{table}[h]
  \begin{tabular}{|c|c|c|}\hline
    $M=G/K$ & rank & type of $\Sigma$ \\\hline
     $SU(p,q)/S(U(p) \times U(q)) \quad (p < q)$ & $p$ & $BC_p$\\\hline
    $SU(p,p)/S(U(p) \times U(p))$ & $p$  & $C_p$ \\\hline
    $SO_0(2,q)/SO(2) \times SO(q) \quad (2 \leq q)$ &2  & $C_2(\cong B_2)$\\\hline
    $SO^*(2n)/U(n)$ ($n$:odd) & $(n-1)/2 $ & $BC_{(n-1)/2}$ \\\hline
    $SO^*(2n)/U(n)$ ($n$:even) & $n/2$  & $C_{n/2}$ \\\hline
    $Sp(n, \R)/U(n)$ & $n$ &  $C_n$ \\\hline
    $E_6^{-14}/T \cdot Spin(10)$ & 2 & $BC_2$ \\\hline
    $E_7^{-25}/T \cdot E_6$ &3  & $C_3$ \\\hline
   \end{tabular}
  \label{list}
  \caption{Irreducible Hermitian symmetric spaces of non-compact type.}
  \end{table}

 In order to simplify our argument, we  use the explicit description of the restricted root system $\Sigma$. We fix simple roots of $\Sigma$  by 
$\{e_1-e_{2},\ldots, e_{r-1}-e_{r}, 2e_r\}$ if $\Sigma=\Sigma_{C_r}$ and $\{e_1-e_{2},\ldots, e_{r-1}-e_{r}, e_r\}$ if $\Sigma=\Sigma_{BC_r}$. Then, the highest root is given by $2e_1$ for both cases, and the set of positive roots $\Sigma_+$ consists of the following four disjoint subsets: 
\[
\Lambda:=\{e_i+e_j\}_{1\leq i<j\leq r}, \quad \overline{\Lambda}:=\{e_i-e_j\}_{1\leq i<j\leq r},\quad  E:=\{e_i\}_{1\leq i\leq r},\quad\Gamma:=\{2e_i\}_{1\leq i\leq r}.
\]
 Note that  $E=\{\phi\}$ if $\Sigma=\Sigma_{C_r}$. We use the following notations: 
\begin{itemize}
\item For the root $\lambda=e_i+e_j\in \Lambda$ , we set $\overline{\lambda}:=e_i-e_j\in \overline{\Lambda}$. We write the element in $E$  by $\epsilon$.
\item We set $\gamma_i=2e_i\in \Gamma$ for $i=1,\ldots, r$. In the following, we often abbreviate $\gamma_i$ simply by $i$, e.g. $\fg_{i}=\fg_{\gamma_i}$, $H_i=H_{\gamma_i}$ etc.
\end{itemize}

The subset $\Gamma\subset \Sigma_+$ plays a fundamental role in our study.  It is easy to see that $\Gamma$ is a subset  consisting of  {\it strongly orthogonal roots}, i.e. it holds that $\gamma\pm\gamma'\not\in \Sigma$ for any two elements $\gamma, \gamma'\in \Gamma$. Moreover, we have the following:

\begin{lemma}\label{Lemma:orth}
The set of root vectors $\{H_{i}\}_{i=1}^r$ corresponding to $\Gamma$ is an orthogonal basis of $\fa$. Moreover, $\|H_{1}\|^2=\cdots =\|H_{r}\|^2=C$ for some positive constant $C$.
\end{lemma}
\begin{proof}
Recall that $H_i$ is parallel to $[X_i, X_{-i}]$ for any non-zero element $X_i\in \fg_i$ by  \eqref{feq3}. For any distinct $i,j\in \{1,\ldots, r\}$, we see
\begin{align*}
B([X_i, X_{-i}], [X_{j}, X_{-j}])&=-B(X_{-i}, [X_i, [X_j, X_{-j}]])\\
&=B(X_{-i}, [X_j,[X_{-j}, X_i]])+B(X_{-i}, [X_{-j},[X_{i}, X_j]])=0
\end{align*} 
since $[X_i, X_{\pm j}]\in \fg_{\gamma_i\pm \gamma_j}=\{0\}$ by the strong orthogonality of roots. This proves $\{H_i\}_{i=1}^r$ consists of orthogonal vectors in $\fa$. Since $r={\rm dim}_{\R} \fa$,  $\{H_{i}\}_{i=1}^r$ is  an orthogonal basis of $\fa$. 
On the other hand,  since $H_i$ corresponds to $2e_i$ in $\Sigma=\Sigma_{C_r}$ or $\Sigma_{BC_r}$, we have $\|H_i\|^2=\|H_j\|^2$ for any $i,j=1,\ldots, r$.
\end{proof}

Recall that the Weyl group $\mathcal{W}$ is generated by reflection $\rho_{\alpha}:\fa\to \fa$ for $\alpha\in \Sigma$. By using the formula \eqref{reflect},  it is easy to see that
\begin{align}\label{reflect2}
&\rho_{e_j+e_k}(H_i)=
\begin{cases}
- H_k & {\rm if}\ i=j\\
- H_j & {\rm if}\ i=k\\
H_i & {\rm if}\ i\neq j, k
\end{cases},
\quad
\rho_{e_j-e_k}(H_i)=
\begin{cases}
 H_k & {\rm if}\ i=j\\
 H_j & {\rm if}\ i=k\\
H_i & {\rm if}\ i\neq j, k
\end{cases},
\\
&\rho_{e_j}(H_i)=
\rho_{2e_j}(H_i)=
\begin{cases}
-H_j & {\rm if}\ i=j\\
H_i & {\rm if}\ i\neq j 
\end{cases}.
\notag
\end{align}
In particular, $\{H_i\}_{i=1}^r$  is $\mathcal{W}$-invariant up to sign of each vector.

\subsection{The complex structure}\label{strcplx}
Let $M=G/K$ be an HSSNT. 
It is known that the subspace $\fk$ has a real $1$-dimensional center $\fc(\fk)=\{Z\in \fk:\ [X,Z]=0\ \forall X\in \fk\}$. An element $\zeta\in \fc(\fk)$ defines the canonical complex structure $J_o:={\rm ad(\zeta)|_{\fp}}$ on $\fp$, and $J_o$ is extended to the standard complex structure $J$ on $M$. Then, by  putting $\omega(\cdot, \cdot):=g(J\cdot, \cdot)$,  $(M, \omega, J)$ becomes a K\"ahler-Einstein manifold of negative Ricci curvature.

We shall prove some facts on the complex structure $J_o$. First, we note the following fact.

\begin{lemma}\label{Lemma:dim}
For each $\gamma_i\in \Gamma$, ${\rm dim}_{\R}\fg_{i}=1$. 
\end{lemma}
\begin{proof}
Recall that  the highest root with respect to the simple roots described above is given by $2e_1$, and hence, ${\rm dim}_{\R}\fg_{1}=1$ by Proposition \ref{Prop: Hermitian}.  Take an element $k_i\in K$ such that $\rho_{e_1-e_i}={\rm Ad}(k_i)|_{\fa}$. Then, it is easy to see that ${\rm Ad}(k_i)|_{\fg_1}: \fg_1\to \fg_i$ is an isomorphism. Thus, ${\rm dim}_{\R}\fg_{i}=1$ for any $i$.
\end{proof}

Then, we see the following.

\begin{lemma}\label{Lemma:cplx}
For each $i=1,\ldots, r$, there exists a unique vector $Z_i\in \fg_i$ such that the element $\zeta\in \mathfrak{c(k)}$  is expressed by 
\begin{align}\label{cplx}
\zeta=Z_0-\sum_{i=1}^r Z_i^{\fk}
\end{align}
for some $Z_0\in\fk_0$. Moreover, we have 
\begin{align}\label{norm}
\|Z_i^{\fk}\|^2=\|Z_i^{\fp}\|^2=\frac{1}{C},\ \textup{or equivalently}\ \|Z_i\|^2=\frac{2}{C}.
\end{align}
\end{lemma}

\begin{proof}
\eqref{cplx} follows from a similar argument in \cite[Proposition~4.3]{FT} for Hermitian symmetric spaces of compact type.
The uniqueness of the expression is due to Lemma \ref{Lemma:dim}. 
Combining  \eqref{cplx} with \eqref{feq2} and Lemma \ref{Lemma:orth}, we have
\begin{align}\label{jh}
J_oH_i=\Big[Z_0-\sum_{j=1}^r Z_j^{\fk}, H_i\Big]=\sum_{j=1}^r \gamma_j(H_i)Z_j^{\fp}=\gamma_i(H_i)Z_i^{\fp}=CZ_i^{\fp}.
\end{align}
This implies $\|Z_i^{\fp}\|^2=1/C$, where $C=\|H_i\|^2$. Since $\|Z_i^{\fk}\|^2=\|Z_i^{\fp}\|^2$, this shows \eqref{norm}.
\end{proof}

 We set $\fa_i:=\R H_i$ for $i=1,\ldots, r$. 

\begin{lemma}\label{Lemma:J}
We have 
\begin{enumerate}
\item $J_o\fp_{i}=\fa_i$ for any $\gamma_i\in \Gamma$.
\item  $J_o\fp_{\lambda}=\fp_{\overline{\lambda}}$ for any $\lambda\in \Lambda$.
\item $J_o\fp_{\epsilon}=\fp_{\epsilon}$ for any $\epsilon\in E$.
\end{enumerate}
\end{lemma}
\begin{proof}
(i) follows from \eqref{jh} since ${\rm dim}_{\R} \fa_i={\rm dim}_{\R} \fp_i=1$.

Fix arbitrary $\lambda\in \Lambda$. We may assume $\lambda=e_j+e_k$ for $1 \leq j < k \leq r$  in $\Sigma_{C_r}$ or $\Sigma_{BC_r}$.  Combining Lemma \ref{Lemma:cplx} with the fact that $\zeta\in \mathfrak{c(k)}$, we see
\begin{align*}
0&=[\zeta, X_{\lambda}^{\fk}]=[Z_0-\sum_{i=1}^r Z_{2e_i}^{\fk}, X_{e_j+e_k}^{\fk}]=[Z_0, X_{e_j+e_k}^{\fk}]
-[Z_{2e_j}^{\fk}+Z_{2e_k}^{\fk}, X_{e_j+e_k}^{\fk}]
\end{align*}
for any $X_{\lambda}^\fk\in \fk_{\lambda}$. By \eqref{brakp}, the first term of the right hand-side belongs to $\fk_{e_j+e_k}$ although the second term belongs to $\fk_{e_j-e_k}$.
Thus,   we obtain $[Z_0, X_{e_j+e_k}^{\fk}]=0$ since $\fk_{e_j+e_k}\cap\fk_{e_j-e_k}=\{0\}$. This implies $[Z_0, X_{e_j+e_k}^{\fp}]=0$, and hence, we see
\[
J_oX_{\lambda}^\fp=[Z_0-\sum_{i=1}^r Z_{2e_i}^{\fk}, X_{e_j+e_k}^{\fp}]=-[Z_{2e_j}^{\fk}+Z_{2e_k}^{\fk}, X_{e_j+e_k}^{\fp}]\in \fp_{e_j-e_k}=\fp_{\overline{\lambda}}
\]
for any $X_{\lambda}^\fp\in \fp_{\lambda}$. Namely, we have $J_o\fp_{\lambda}\subset \fp_{\overline{\lambda}}$. By a similar argument, we also have $J_o\fp_{\overline{\lambda}}\subset \fp_{{\lambda}}$. Thus, $J_o\fp_{\lambda}= \fp_{\overline{\lambda}}$. This proves (ii).

Finally, we show (iii). For any $\epsilon=e_j\in E$ with $1\leq j\leq r$, we easily see
\begin{align*}
J_oX_{\epsilon}^\fp=[Z_0-\sum_{i=1}^r Z_{2e_i}^{\fk}, X_{e_j}^\fp]=
\frac{1}{2}[Z_0-Z_{2e_j}^{\fk}, X_{e_j}^\fp] \in \fp_{e_{j}}=\fp_{\epsilon}.
\end{align*}
This implies (iii).
\end{proof}

We put 
\begin{align}\label{Def:decac}
\fa^\C:=\bigoplus_{i=1}^r\fa_i^\C,\quad {\rm where}\ \fa_i^{\C}:=\fa_i\oplus \fp_i=\fa_i\oplus J_o\fa_i.
\end{align}
Then, $\fa^\C$ is a complex subspace in $\fp$ and \eqref{Def:decac} is an orthogonal decomposition of $\fa^\C$.

\begin{lemma}\label{Lemma:Jp}
With respect to the orthogonal decomposition \eqref{Def:decac}, we have 
\[
J_o|_{\fa^\C}={\rm ad}(-Z_1^{\fk})|_{\fa_1^\C}\oplus \cdots \oplus {\rm ad}(-Z_r^{\fk})|_{\fa_r^\C}.
\]
 \end{lemma}
 \begin{proof}
The equation \eqref{jh} shows that 
$
J_oH_i={\rm ad}(-Z_i^{\fk})H_i.
$
On the other hand,  we have
\begin{align}\label{ikp}
[Z_i^\fk, Z_i^\fp]=\frac{1}{4}[Z_i+Z_{-i}, Z_i-Z_{-i}]=-\frac{1}{2}[Z_i,Z_{-i}]=\frac{1}{C}H_i
\end{align}
by  \eqref{feq3} and \eqref{norm}.  Therefore, ${\rm ad}(-Z_i^{\fk})J_oH_i=[-Z_i^{\fk}, CZ_i^{\fp}]=-H_i=J_o(J_oH_i)$. Thus, $J_o|_{\fa_i^\C}={\rm ad}(-Z_i^{\fk})|_{\fa_i^\C}$ for each $i$. Moreover, we have ${\rm ad}(-Z_i^\fk)|_{\fa_j^\C}=0$ if $i\neq j$ by the strong orthogonality of roots. Thus, we obtain the lemma.
\end{proof}

\subsection{Polydisk theorem}\label{subsec:polydisk}
Let $Z_i$ ($i=1,\ldots, r$) be the basis of $\fg_i$ satisfying \eqref{cplx} and \eqref{norm} in Lemma \ref{Lemma:cplx}. 
It follows that 
$\fk_i=\R Z_i^{\fk}$ and $\fp_i=\R Z_i^{\fp}$.
Also, we recall that
$
\fa_i:=\R H_i,
$
where $H_i$ is the root vector of $\gamma_i$. 
 We put
\begin{align*}
\fu_i:=\fk_i\oplus\fa_i\oplus\fp_i\subset \fg\quad {\rm and}\quad \fa_i^{\C}:=\fa_i\oplus\fp_i\subset \fp
\end{align*}
for $i=1,\ldots, r$ so that 
\[
\fa\oplus\bigoplus_{\gamma_i\in \Gamma}(\fg_{\gamma_i}\oplus\fg_{-\gamma_i})=\bigoplus_{i=1}^r \fu_i.
\]

\begin{lemma}\label{brau}
Suppose $i,j\in \{1,\ldots, r\}$. 
\begin{enumerate}
\item 
$
[\fk_i, \fk_j]=[\fa_i, \fa_j]=[\fp_i, \fp_j]=\{0\}
$
for any $i,j$.
\item $
[\fk_i, \fp_j]=[ \fa_i, \fk_j]=[\fa_i, \fp_j]=\{0\}
$ if $i\neq j$.
\item $[\fk_i, \fp_i]=\fa_i$, $[\fa_i, \fk_i]=\fp_i$ and $[\fa_i, \fp_i]=\fk_i$. More precisely, we have
\begin{align}\label{brau1}
[Z_i^{\fk}, Z_i^{\fp}]=\frac{1}{C}H_i,\quad [H_i, Z_i^{\fk}]=CZ_i^{\fp}, \quad [H_i, Z_i^{\fp}]=CZ_i^{\fk}. 
\end{align}
\end{enumerate}
In particular, $[\fu_i, \fu_j]=\{0\}$ if $i\neq j$ and $\fu_i$ is a Lie subalgebra of $\fg$ for each $i$.
\end{lemma}
\begin{proof}
First, we assume $i\neq j$. Then,  the strong orthogonality of roots implies $[\fk_i, \fk_j]=[\fp_i, \fp_j]=[\fk_i, \fp_j]=\{0\}$. Also, $[\fa_i,\fa_j]=\{0\}$ since $\fa$ is abelian.  On the other hand, \eqref{feq2} shows 
 $[H_i, X_j^{\fk}]=\gamma_j(H_i)X_j^{\fp}=0$ and $[H_i, X_j^{\fp}]=\gamma_j(H_i)X_j^{\fk}=0$
since $\gamma_j(H_i)=\langle H_j, H_i \rangle=0$ by Lemma \ref{Lemma:orth}. Thus, $[ \fa_i, \fk_j]=[\fa_i, \fp_j]=\{0\}$, and this proves (i) and (ii) for the case when $i\neq j$.

If $i=j$, we have $[\fk_i, \fk_i]=[\fp_i, \fp_i]=[\fa_i,\fa_i]=\{0\}$ since ${\rm dim}_{\R}\fk_i={\rm dim}_{\R}\fp_i={\rm dim}_{\R}\fa_i=1$. On the other hand,  we have \eqref{brau1} by \eqref{feq2} and \eqref{ikp}.  This implies (iii).
\end{proof}

Let us consider the Lie subalgebra $\fu_i$. 
Since we have $J_oH_i=CZ_{i}^{\fp}$ (see \eqref{jh}), it follows that 
$
[H_i, J_oH_i]=C^2Z_i^{\fk},
$
and hence, we have
\[
\fu_i=\R[H_i, J_oH_i]\oplus \R H_i\oplus \R J_oH_i\quad {\rm and}\quad \fa_i^{\C}=\R H_i\oplus \R J_oH_i.
\]

For the sake of convenience, we put
\[
\wH_i:=\frac{2}{C}H_i=\frac{2}{\langle \gamma_i, \gamma_i\rangle} H_i
\]
 for $i=1,\ldots, r$, and we say $\wH_i$ the {\it normalized root vector}. Note that $\|\wH_i\|^2=4/C$. Moreover, by \eqref{jh} and \eqref{brau1}, we have
\begin{align}\label{brawh}
J_o\wH_i=2Z_i^{\fp},\quad [Z_i^{\fk}, Z_i^{\fp}]=\frac{1}{2}\wH_i,\quad [\wH_i, Z_i^{\fk}]=2Z_i^{\fp}, \quad [\wH_i, Z_i^{\fp}]=2Z_i^{\fk}.
\end{align}

\begin{lemma}\label{Lemma:iso}
For each $i=1,\ldots, r$, $\fu_i$  is isomorphic to $\mathfrak{su}(1,1)$ as a Lie algebra via the linear map $f_i: \fu_i\to \mathfrak{su}(1,1)$ defined by
\begin{align}\label{Def:fl}
\wH_i\mapsto 
\left[
\begin{array}{cc}
0 & 1\\
 1 & 0
\end{array}
\right], \quad
J_o\wH_i\mapsto 
\left[
\begin{array}{cc}
0 & -\sqrt{-1}\\
\sqrt{-1} & 0
\end{array}
\right],\quad
[\wH_i, J_o\wH_i]\mapsto 
2
\left[
\begin{array}{cc}
\sqrt{-1}& 0\\
 0 & -\sqrt{-1}
\end{array}
\right].\quad
\end{align}
\end{lemma}

\begin{proof}
By \eqref{brawh},   we see
\begin{align}\label{brah}
[[\wH_i, J_o\wH_i], J_o\wH_i]=4\wH_i,\quad [[\wH_i, J_o\wH_i], \wH_i]=-4J_o\wH_i.
\end{align}
Thus, $\fu_i$ is a Lie subalgebra of $\fg$. Then, it is easy to verify that the linear map $f_i$ is a Lie algebra isomorphism.
\end{proof}

\begin{lemma}\label{Lemma:ch1}
For each $i=1,\ldots, r$,  $A_i^\C:=\Exp_o\fa_i^{\C}$ is a totally geodesic complex submanifold in $M$  which is holomorphically isometric to the complex hyperbolic space $\C H^1(-C)$ of constant holomorphic sectional curvature $-C$.
\end{lemma}

\begin{proof}
Let $U_i$ be a connected closed subgroup of $G$  with Lie algebra $\fu_i=[\fa_i^{\C}, \fa_i^{\C}]\oplus \fa_i^{\C}$. Then, we have $U_i\cdot o=A_i^\C$ and $T_oA_i^\C=\fa_i^\C$. Since the subspace $\fa_i^{\C}$ is a complex subspace of $\fp$ (see Lemma \ref{Lemma:J}) and  $J$ is left invariant, $A_i^\C$ is a complex submanifold in $M$.
Moreover, the equation \eqref{brah} shows that $\fa_i^{\C}$ is a Lie triple system, i.e. $[[\fa_i^\C, \fa_i^\C],\fa_i^\C]\subseteq \fa_i^\C$, and hence, $U_i\cdot o=A_i^\C$ is a complete connected totally geodesic submanifold (see \cite[Proposition 11.1.2]{BCO}).  

We shall show $A^\C$ is holomorphic isometric to the product of $\C H^1(-C)$. 
We identify  $\C H^1(-C)$ with the Hermitian symmetric space $\widehat{G}/\widehat{K}=SU(1,1)/S(U(1)\times U(1))$, and we denote the associated Cartan decomposition by $\widehat{\fg}=\widehat{\fk}\oplus \widehat{\fp}$.   More precisely, 
\begin{align*}
& \widehat{\mathfrak{k}}=\Big{\{}
\frac{1}{2}
\left[
\begin{array}{cc}
-\sqrt{-1}\theta& 0\\
 0 & \sqrt{-1}\theta
\end{array}
\right]:\ \theta\in \R
\Big{\}},
\quad
\widehat{\mathfrak{p}}=\Big{\{}
\left[
\begin{array}{cc}
0 & \overline{z}\\
 z & 0
\end{array}
\right]: z\in \C\Big{\}},
 \end{align*}
and the canonical inner product, the complex structure and the symplectic form on $\widehat{\fp}$ are defined by $\langle X, Y\rangle_{\widehat{\fp}}=(2/C){\rm tr}(XY)$, $\widehat{J}_o:={\rm ad}\widehat{\zeta}|_{\widehat{\fp}}$ where $\widehat{\zeta}=(1/2){\rm diag}(-\sqrt{-1},\sqrt{-1})$ and $\widehat{\omega}_o(\cdot, \cdot):=\langle \widehat{J}_o\cdot, \cdot\rangle_{\widehat{\fp}}$, respectively.

By Lemma \ref{Lemma:iso},  we see the restriction
$
f_i|_{\fa_i^{\C}}
$ yields a linear isomorphism between $\fa_i^\C$ and $\widehat{\fp}$.  Moreover, we see
\begin{align*}
\widehat{J}_o (f_i)_* \wH_i={\rm ad}(\widehat{\zeta})
\left[
\begin{array}{cc}
0 & 1\\
 1 & 0
\end{array}
\right]
=
\left[
\begin{array}{cc}
0 & -\sqrt{-1}\\
 \sqrt{-1} & 0
\end{array}
\right]
=(f_i)_*J_o\wH_i,
\end{align*}
and 
\begin{align*}
\|(f_i)_*\wH_i\|^2_{\widehat{\fp}}=\frac{2}{C}{\rm tr}\left[
\begin{array}{cc}
0 & 1\\
 1 & 0
\end{array}
\right]^2=\frac{4}{C}=\|\wH_i\|^2.
\end{align*}
Namely, $f_i|_{\fa_i^{\C}}$ preserves the complex structure and the inner product.

Then, one easily checks that  the map $F_i: A_i^\C\to \C H^1(-C)$ defined by
\begin{align*}
F_i:={\rm Exp}_o^{\C H^1}\circ f_i\circ ({\rm Exp}_o|_{\fa_i^{\C}})^{-1}
\end{align*}
gives  a holomorphic isometry, where $\Exp_o^{\C H^1}: \widehat{\fp}\to \C H^1(-C)$ is the Riemannian exponential map of $\C H^1(-C)$ at the origin. 
\end{proof}

Recall that we put
$
\fa^{\C}=\fa\oplus J_o\fa=\bigoplus_{i=1}^r\fa_i^{\C}.
$
Then, the submanifold $A^\C:=\Exp_o\fa^\C$ is also a totally geodesic complex submanifold in $M$. 
Moreover, we have a canonical splitting into a direct product of $r$ totally geodesic complex submanifolds:
\begin{align*}
\widetilde{F}&:  A^\C\to A_1^\C\times \cdots \times  A_r^\C,\quad \Exp_o(w_1+\cdots +w_r)\mapsto (\Exp_ow_1,\ldots, \Exp_o w_r),
\end{align*}
 where $w_i \in \fa_i^{\C}$.  Since $J_o|_{\fa^\C}$ and  $\langle,\rangle|_{\fa^\C}$ split into  $J_o|_{\fa_1^\C}\oplus \cdots \oplus J_o|_{\fa_r^\C}$ (see Lemma \ref{Lemma:Jp}) and $\langle,\rangle|_{\fa_1^\C}\oplus\cdots\oplus\langle,\rangle|_{\fa_r^\C} $, respectively,  $\widetilde{F}$ gives rise to a holomorphic isometry.
 
 We note that,  by Lemma \ref{Lemma:Jp}, we have $J_o|_{\fa_i^{\C}}={\rm ad}(-Z_i^{\fk})|_{\fa_i^{\C}}$. Thus, by putting
 \[
 k_i(\theta):=\exp(-\theta Z_i^{\fk})\in K,
 \]
 we see
 \begin{align*}
{\rm Ad}(k_i(\theta))|_{\fa_i^{\C}}=\sum_{n=0}^{\infty}\frac{\theta^n}{n!}{\rm ad}(-Z_i^{\fk})^n\Big|_{\fa_i^{\C}}=e^{\sqrt{-1}\theta}
  \end{align*}
under the identifications $\fa_i^\C\simeq \widehat{\fp}\simeq \C$.  Therefore, the abelian subgroup $T_i:=\{k_i(\theta): \theta\in \R\}\subset K$ acts on $A_i^\C$ and $\fa_i^{\C}$ transitively, and acts on $A_j^\C$ and $\fa_{j}^{\C}$ trivially if $i\neq j$.  Obviously,  the $T_i$-action on $A_i^\C$ (resp. $\fa_i^\C$) is equivariant to the $\widehat{K}$-action on $\C H^1(-C)$ (resp. $\widehat{\fp}$) under the identification $A_i^\C\simeq \C H^1(-C)$.
Furthermore, since $[\fk_i,\fk_j]=0$,  the subgroup $T^r:=T_1\cdots T_r$ is also abelian and acts on $\fa^{\C}$ transitively. More precisely, we have
\[
{\rm Ad}(k_1(\theta_1)\cdots k_r(\theta_r))(w_1+\cdots +w_r)=e^{\sqrt{-1}\theta_1}w_1+\cdots +e^{\sqrt{-1}\theta_r}w_r
\]
for $w_i\in \fa_i^{\C}\simeq \C$. In particular,  $\widetilde{F}$ is a $T^r$-equivariant map.

Summarizing the arguments,  we proved the following fact which we mentioned in Theorem \ref{thm:polydisk}.

 \begin{proposition} \label{key1}
Let $M=G/K$ be a HSSNT of rank $r$, and $\fa$ be any maximal abelian subspace in $\fp$. Then,  the submanifold $A^{\C}:=\Exp_o\fa^\C$ splits into a direct product of $r$ totally geodesic $\C H^1(-C)$ in $M$. Moreover, the map
\begin{align*}
&F: A^{\C}\to \C H^1(-C)\times \cdots \times \C H^1(-C),\\
& \Exp_o(w_1+\cdots +w_r)\mapsto (\Exp_o^{\C H^1}f_1(w_1),\ldots, \Exp_o^{\C H^1} f_r(w_r)), \nonumber
\end{align*}
is a $T^r$-equivariant holomorphic isometry, where $w_i\in \fa_i^{\C}$, $f_i$ is defined by \eqref{Def:fl}, ${\rm Exp}_o^{\C H^1}: \widehat{\fp}\to \C H^1(-C)$ is the exponential map of $\C H^1(-C)$. 
\end{proposition}

\begin{remark}\label{Rem: product}
{\rm  Proposition \ref{key1} was described in, for example, \cite[Exercises B-2 in  \S 7 of ChVIII]{Hel} \cite[Part I]{Wolf} in a different manner, that is, by using the strongly orthogonal roots associated with the root system with respect to the Cartan subalgebra $\ft_{\C}$ in $\fg_\C$ for  a maximal abelian subspace $\ft$ in $\fk$,  where $\ft_\C$ and $\fg_\C$  are complexifications of $\ft$ and $\fg$, respectively.  Although our description above is given by using the {\it restricted  root system}, we see that Proposition \ref{key1} coincides with the classical polydisk theorem. See Appendix \ref{A1} for the correspondence. }
\end{remark}

\section{$K$-equivariant realizations}\label{Section:equiv}
Let $M=G/K$ be a Riemannian symmetric space of noncompact type, where $M$ is not necessarily  a Hermitian symmetric space at first. We denote the Cartan decomposition by $\fg=\fk\oplus \fp$. 
 We say a map $\Omega: M\to \fp$   {\it $K$-equivariant } if $\Omega$ satisfies $\Omega(k\cdot x)={\rm Ad}(k)\cdot \Omega(x)$ for any $k\in K$ and $x\in M$.  A typical example of $K$-equivariant map is given by the inverse of the exponential map 
$
\Log_o: M\to \fp.
$
We note that, by the polarity of the $K$-action (Theorem \ref{Thm:polar}), any element $x\in M$ is represented as $x=k\Exp_ov$ for some $k\in K$ and $v\in \fa$, and hence, $\Log_o$ is expressed by
  \[
\Log_o(k\Exp_o v)= {\rm Ad}(k)v.
 \]
  In particular,  $\Log_o$ sends the section $A=\Exp_o\fa$ into the section $\fa$. 
 In this section, we generalize the situation of the map $\Log_o$, and provide a way to construct a $K$-equivariant map $\Omega: M\to \fp$ satisfying $\Omega(A)\subseteq \fa$.

\subsection{A natural construction}\label{Subsec:equiv1}

We fix a  maximal abelian subspace $\fa$ of $\fp$, and put $A:=\Exp_o\fa$. Recall that we have ${\rm Ad}(K)\fa=\fp$ and $K\cdot A=M$ by Theorem \ref{Thm:polar}. 
We consider a slightly general situation.

Let $U\subseteq M$ be a $K$-invariant connected open  neighborhood of $o\in M$,  possibly $U=M$. We set 
\[
U_A:=U\cap A, \quad \fu:=\Log_oU\subseteq \fp\quad  {\rm and}\quad  \fu_{\fa}:=\Log_oU_A\subseteq \fa.
\]
 Then, $U_A$ is an open subset of $A$ around $o$ such that $K\cdot U_A=U$, and  $\fu_\fa$ is an open subset of $\fa$ around $0$ such that ${\rm Ad}(K)\fu_\fa=\fu$.
We take a $C^0$-map  
\[
{\Omega}_{U_A}: U_A\to \fa
\]
and put
\begin{align}\label{Def:Omega}
\Omega: U\to \fp,\quad\Omega(k\Exp_ov):={\rm Ad}(k)\circ {\Omega}_{U_A}(\Exp_ov).
\end{align}
for $k\in K$ and $v\in \fu_\fa$.  By definition, $\Omega$ is a $K$-equivariant $C^0$-map satisfying $\Omega(U_A)\subseteq \fa$ if $\Omega$ is well-defined, i.e. $\Omega$ is independent of the expression of $x=k\Exp_ov\in M$.  Conversely, any $K$-equivariant map $\Omega: U\to \fp$ such that $\Omega(U_A)\subseteq \fa$ is obtained in this way.

Nevertheless, $\Omega$ does not become well-defined without extra condition for $\Omega_{U_A}$. Indeed, we see the following Lemma: Recall that we denote the Weyl group associated with the restricted root system $\Sigma$ by $\mathcal{W}$. The Weyl group naturally acts on $\fa$ and $A$, and both $\fu_\fa$ and $U_A$ are $\mathcal{W}$-invariant subsets.

\begin{lemma}\label{Lemma:welldef}
Let $M$ be a Riemannian symmetric space of noncompact type. Then, the map $\Omega$ defined by \eqref{Def:Omega}  is well-defined if and only if the map $\Omega_{U_A}: U_A\to \fa$ is $\mathcal{W}$-equivariant.  Moreover, if the $\mathcal{W}$-equivariant map $\Omega_{U_A}$ is an embedding, then $\Omega$ is also an embedding.
\end{lemma}

To prove this, we need to divide our argument depending on the type of the element in $\fa$. We say $v$ is a {\it principal (resp. singular) element} if the $K$-orbit ${\rm Ad}(K)v$ through $v$ is a principal (resp. singular) orbit. Note that the ${\rm Ad}(K)$-action on $\fp$ does not admit any exceptional orbit. We denote the set of principal elements in $\fa$ by $\fa^{pri}$. It is known that $\fa^{pri}=\{v\in\fa: \alpha(v)\neq 0\ \textup{for any}\ \alpha\in \Sigma\}$ (see \cite[Subsection 2.3.2]{BCO}). In particular,  $\fa^{pri}$ is an open dense subset in $\fa$.

For the sake of convenience, we put 
\begin{align*}
{\bf h}_{\fu_{\fa}}&:\fu_\fa\to\fa,\quad {\bf h}_{\fu_\fa}:=\Omega_{U_A}\circ \Exp_o\quad {\rm and}\\
{\bf h}&: \fu\to \fp,\quad {\bf h}({\rm Ad}(k)v):={\rm Ad}(k)\circ{\bf h}_{\fu_{\fa}}(v)
\end{align*}
for $k\in K$ and $v\in \fu_\fa$ so that $\Omega={\bf h}\circ \Log_o$. Then, $\Omega$ is well-defined if and only if so is ${\bf h}$.

\begin{proof}[Proof of Lemma \ref{Lemma:welldef}]
Suppose $\Omega$  is well-defined.  Then, $\Omega$ is a $K$-equivariant map. Since any element $\rho \in \mathcal{W}$ is represented by some  $k_{\rho}\in K$, it is easy to see that $\Omega_{U_A}=\Omega|_{U_A}$ is $\mathcal{W}$-equivariant. 

Conversely, suppose that the map $\Omega_{U_A}:U_{A}\to \fa$ is $\mathcal{W}$-equivariant, or equivalently, the map ${\bf h}_{\fu_\fa}: \fu_\fa\to \fa$ is $\mathcal{W}$-equivariant.  We shall show the map ${\bf h}$ is well-defined. Namely, for any $X\in \fu$,  ${\bf h}(X)$ is independent of the choice of $k\in K$ and $v\in \fu_\fa$ such that $X={\rm Ad}(k)v$. 

Suppose  $X$ is represented by $X={\rm Ad}(k)v={\rm Ad}(k')v'$ for some $k, k'\in K$ and $v, v'\in \fu_\fa$.  Since $v'={\rm Ad}((k')^{-1}k)v$, we see $v'$ is an element in $\fa\cap {\rm Ad}(K)v$. Then, it follows that there exists an element $\rho\in\mathcal{W}$  such that $v'=\rho v$ (see \cite[Proposition 2.2 in Ch.VII]{Hel}). We may assume $\rho$ is represented by ${\rm Ad}(k_{\rho})$ for some $k_{\rho}\in K$. Then, we have ${\rm Ad}(k_{\rho})v={\rm Ad}((k')^{-1}k)v$,  and hence, $k_{\rho}^{-1}(k')^{-1}k$ is an element in the isotropy subgroup $ K_v:=\{k\in K: {\rm Ad}(k)v=v\}$ at $v$.

First, we assume $v$ is a principal element. Then,  the slice representation at $v$ is trivial \cite[Theorem 2.1.3]{BCO}, namely, for any   $k_0\in K_v$, ${\rm Ad}(k_0)$ acts on $T_v^{\perp}({\rm Ad}(K)v)$ trivially.  Since the orbit ${\rm Ad}(K)\cdot v$ intersects to $\fa$ orthogonally,  we see ${\rm Ad}(k_0)|_{\fa}=id$ for any $k_0\in K_v$.  Thus, we obtain 
$
{\rm Ad}(k)w={\rm Ad}(k')\circ \rho w
$
for any $w\in \fa$  since $k_{\rho}^{-1}(k')^{-1}k\in K_v$.
Therefore,
\begin{align*}
{\bf h}({\rm Ad}(k)v)&={\rm Ad}(k){\bf h}_{\fu_\fa}(v)={\rm Ad}(k')\circ \rho{\bf h}_{\fu_\fa}(v)\\
&={\rm Ad}(k'){\bf h}_{\fu_\fa}(\rho v)={\rm Ad}(k'){\bf h}_{\fu_\fa}(v')={\bf h}({\rm Ad}(k')v').
\end{align*}
Here, in the third equality, we used the assumption that ${\bf h}_{\fu_\fa}$ is $\mathcal{W}$-equivariant. This proves that  ${\bf h}$ (and hence, $\Omega$) is well-defined if $v$ is a principal element.

Next,  we assume $v$ is a singular element.   Since $\fa^{pri}=\{v\in\fa: \alpha(v)\neq 0\ \textup{for any}\ \alpha\in \Sigma\}$, we may take a continuous curve $v(t)\in \fu_\fa$ for $t\in (-\epsilon, \epsilon)$ so that $v(0)=v$ and $v(t)\in \fa^{pri}$ if $t\neq 0$.   Also, we define another continuous curve $v'(t):={\rm Ad}((k')^{-1}k)v(t)$ so that $v'(0)=v'$ and ${\rm Ad}(k)v(t)={\rm Ad}(k')v'(t)$ for any $t$.  Since $v(t)\in \fa^{pri}$ if $t\neq 0$, we have that ${\bf h}({\rm Ad}(k)v(t))={\bf h}({\rm Ad}(k')v'(t))$ for $t\neq 0$. Because the functions of both hand side are continuous, this implies ${\bf h}({\rm Ad}(k)v(0))={\bf h}({\rm Ad}(k')v'(0))$. Thus, ${\bf h}$ is also well-defined for any singular element, and we complete the proof of well-definedness of $\Omega$.

Finally, we suppose $\Omega_{U_A}: U_A\to\fa$, or equivalently, ${\bf h}_{\fu_\fa}:\fu_\fa\to\fa$ is a $\mathcal{W}$-equivariant embedding. Setting $\square:={\bf h}_{\fu_\fa}(\fu_\fa)$ and $D:={\bf h}(\fu)$, we have $D={\rm Ad}(K)(\square)$. We put
\[
{\bf h}^{*}: D\to \fp,\quad {\bf h}^{*}({\rm Ad}(k)w)={\rm Ad}(k)\circ {\bf h}_{\fu_\fa}^{-1}(w)
\]
for $k\in K$ and $w\in \square$.
By a similar argument given above,  it turns out that ${\bf h}^{*}$ is a well-defined $C^0$-map, and we have ${\bf h}^*={\bf h}^{-1}$. Therefore, ${\bf h}$ is an embedding, and hence, so is $\Omega$.
\end{proof}

\subsection{The case of irreducible HSSNT} 
Now, we assume that $M=G/K$ is an irreducible HSSNT of rank $r$. We use the same notations given in Section \ref{Prelim}.

Let $\Omega_{U_A}: U_A\to {\fa}$ be a $C^0$-map.   By using the orthogonal basis $\{\wH_i\}_{i=1}^r$ of $\fa$ given in Lemma \ref{Lemma:orth}, we may assume $\Omega_{U_A}$ is expressed by
\begin{align}\label{Def:oma}
\Omega_{U_A}\Big(\Exp_o\Big(\sum_{i=1}^r x_i\wH_i\Big)\Big)=\sum_{i=1}^r h_i(x_1,\ldots, x_r)\wH_i
\end{align}
for some continuous functions $h_i: \fu_\fa\to\R$, where we identify $\fu_{\fa}$ with an open subset in $\R^r$ by $\sum_{i=1}^r x_i\wH_i\mapsto (x_1,\ldots, x_r)$.   Recall that the restricted root system $\Sigma$ associated with HSSNT $M$ is type $C_r$ or $BC_r$ (Proposition \ref{Prop: Hermitian}), and the Weyl group $\mathcal{W}$ is precisely described by \eqref{reflect2}.  This implies the following restrictions to be the map $\Omega_{U_A}$ $\mathcal{W}$-equivariant. We denote $p_{jk}: \R^r\to\R^r$ as the permutation of the $j$-th variable $x_j$ and the $k$-th variable $x_k$.

\begin{lemma}\label{Lemma:rest}
Suppose $M$ is an irreducible HSSNT.  Then, the map $\Omega_{U_A}$ is $\mathcal{W}$-equivariant if and only if there exists a single function $h: \fu_{\fa}\to \R$ such that 
\begin{enumerate}
\item $h_i=h\circ p_{1i}$ for any $i=1,\ldots, r$,  and
\item $h$ satisfies the following properties:
\begin{itemize}
\item[\rm (a)] $h$ is an odd function with respect to the first variable $x_1$. 
\item[\rm (b)] $h$ is an even function with respect to the variable $x_i$ for $i>1$ and symmetric with respect to $x_i$ and $x_j$ for $i,j>1$, i.e. $h=h\circ p_{ij}$ for any $i,j>1$.
\end{itemize}
\end{enumerate}
\end{lemma}
\begin{proof}
We put ${\bf h}_{\fu_\fa}:=\Omega_{U_A}\circ \Exp_o|_{\fu_\fa}$. Then $\Omega_{U_A}$ is $\mathcal{W}$-equivariant if and only if so is ${\bf h}_{{\fu_\fa}}$. Since $\mathcal{W}$ is generated by the reflection $\rho_{\alpha}$ for $\alpha\in \Sigma$, ${\bf h}_{{\fu_\fa}}$ is $\mathcal{W}$-equivariant  if and only if ${\bf h}_{{\fu_\fa}}\circ \rho_{\alpha}=\rho_{\alpha}\circ {\bf h}_{{\fu_\fa}}$ for any $\alpha\in \Sigma_+$. Here, we have the following equivalent conditions due to the explicit description of the reflection \eqref{reflect2} and the expression \eqref{Def:oma} of $\Omega_{U_A}$:
\begin{itemize}
\item The identity ${\bf h}_{{\fu_\fa}}\circ \rho_{2e_i}=\rho_{2e_i}\circ {\bf h}_{{\fu_\fa}}$ (or equivalently, ${\bf h}_{{\fu_\fa}}\circ \rho_{e_i}=\rho_{e_i}\circ {\bf h}_{{\fu_\fa}}$) is equivalent to
\begin{align}
\label{h11} h_i(x_1,\ldots,-x_i,\ldots, x_r)&=-h_i(x_1,\ldots,x_i,\ldots, x_r)\quad {\rm and}\\
\label{h12} h_j(x_1,\ldots, -x_i,\ldots, x_r)&=h_j(x_1,\ldots, x_i,\ldots, x_r)\quad \textup{if $j\neq i$}.
\end{align}

\item The identity ${\bf h}_{{\fu_\fa}}\circ \rho_{e_i-e_j}=\rho_{e_i-e_j}\circ {\bf h}_{{\fu_\fa}}$ ($1\leq i<j\leq r$)   is equivalent to
\begin{align}
\label{h21} &h_i\circ p_{ij}=h_j,\quad h_j\circ p_{ij}=h_i\quad {\rm and}\\
\label{h22} &h_k\circ p_{ij}=h_k\quad \textup{if $k\neq i,j$}.
\end{align}
\end{itemize}
Remark that the two equations of \eqref{h21} are equivalent to each other.  Moreover,  if both ${\bf h}_{{\fu_\fa}}\circ \rho_{2e_i}=\rho_{2e_i}\circ {\bf h}_{{\fu_\fa}}$ and ${\bf h}_{{\fu_\fa}}\circ \rho_{e_i-e_j}=\rho_{e_i-e_j}\circ {\bf h}_{{\fu_\fa}}$ are satisfied for $1\leq i<j\leq r$, it follows that ${\bf h}_{{\fu_\fa}}\circ \rho_{e_i+e_j}=\rho_{e_i+e_j}\circ {\bf h}_{{\fu_\fa}}$ since $\rho_{e_i+e_j}=\rho_{2e_i}\circ \rho_{2e_j}\circ \rho_{e_i-e_j}$.

Now, we suppose that ${\bf h}_{{\fu_\fa}}$ is $\mathcal{W}$-equivariant.  Then, by putting $i=1$ and $h:=h_1$, \eqref{h11} and \eqref{h21} show that (i) and (ii)--(a) are satisfied.  Moreover, \eqref{h12} shows that $h$ is an even function with respect to the variable $x_i$ for $i>1$, and by putting $k=1$ in \eqref{h22}, we obtain $h=h\circ p_{ij}$ for any $i,j>1$. Namely, the condition (ii)--(b) is satisfied.

Next, we consider  the converse. Suppose that there exists a function $h: \fu_{\fa}\to \R$ satisfying (i) and (ii), and $\Omega_{U_A}$ is given by \eqref{Def:oma}. We shall show the equations \eqref{h11}--\eqref{h22} are satisfied.
The equations \eqref{h11} and \eqref{h12} are easy to check. We show that \eqref{h21} and \eqref{h22} hold for any $i,j$. 

First, we assume $i=1$. Then \eqref{h21} is obvious by (i), and for any $j,k>1$, we have
\begin{align*}
h_k\circ p_{1j}=(h\circ p_{1k})\circ p_{1j}&=(h\circ p_{1k}\circ p_{1j}\circ p_{1k})\circ p_{1k}=(h\circ p_{jk})\circ p_{1k}
=h\circ p_{1k}=h_k,
\end{align*}
where we used $ p_{1k}\circ p_{1j}\circ p_{1k}=p_{jk}$ and (ii)-(b) $h\circ p_{jk}=h$.
Thus \eqref{h22} holds if $i=1$.

Next, we suppose $i>1$. Then, for any $j>i>1$, we have
\begin{align*}
h_i\circ p_{ij}&=(h\circ p_{1i})\circ p_{ij}=(h\circ p_{1i}\circ p_{ij}\circ p_{1j})\circ p_{1j}=(h\circ p_{ij})\circ p_{1j}=h\circ p_{1j}=h_j,
\end{align*}
where we used $p_{1i}\circ p_{ij}\circ p_{1j}=p_{ij}$ and (ii)-(b). On the other hand, for any $k\neq i, j$, we see
\begin{align*}
h_k\circ p_{ij}=(h\circ p_{1k})\circ p_{ij}&=(h\circ p_{ij})\circ p_{1k}=h\circ p_{1k}=h_k
\end{align*}
by (ii)-(b). Therefore, \eqref{h21} and \eqref{h22} are satisfied. This proves that ${\bf h}_{{\fu_\fa}}$ is $\mathcal{W}$-equivariant.
\end{proof}

Combining this  with Lemma \ref{Lemma:welldef}, we obtain the following:

\begin{proposition}\label{Prop:omh}
Let $M$ be an irreducible HSSNT, and $U=K\cdot U_A$ be a $K$-invariant open subset in $M$ containing $o$, where $U_A=U\cap A$. Then,  any function $h:\fu_{\fa}\to \R$ satisfying the properties {\rm (a)} and {\rm (b)} described in Lemma \ref{Lemma:rest} yields a well-defined $K$-equivariant map $\Omega_{h}: U\to\fp$ such that $\Omega_{h}(U_A)\subseteq \fa$,  which is defined by
\begin{align}\label{Def:omh}
&\Omega_{h}(k\Exp_o v):={\rm Ad}(k)\circ \Omega_{h,U_A}(\Exp_ov),\quad {\rm where}\\
&\Omega_{h,U_A}\Big(\Exp_o\Big(\sum_{i=1}^r x_i\wH_i\Big)\Big)=\sum_{i=1}^r h\circ p_{1i}(x_1,\ldots, x_r)\wH_i. \nonumber
\end{align}
for $k\in K$ and $v=\sum_{i=1}^r x_i\wH_i\in \fu_\fa$.  Conversely,  any $K$-equivariant map $\Omega: U\to \fp$ satisfying $\Omega(U_A)\subseteq \fa$ is obtained in this way.  
\end{proposition}

The following lemma shows that  $\Omega_{h}$ depends only on the function $h$. 

\begin{lemma}\label{Lemma:indep}
The definition of $\Omega_{h}$ is independent of the choice of maximal abelian subspace $\fa$ of $\fp$.
\end{lemma}
\begin{proof}
Take two maximal abelian subspaces $\fa^1$ and $\fa^2$ of $\fp$, and set $A^l:=\Exp_o\fa^l$ for each $l=1,2$. We put $\fu_{\fa^l}:=\Log_o(U_{A^l})$ for $U_{A^l}:=U\cap A^l$.  We  consider the map ${\bf h}_{h}^l: \fu\to \fp$ defined by ${\bf h}_{h}^l({\rm Ad}(k)v^l)={\rm Ad}(k){\bf h}_{h, \fu_{\fa^l}}(v^l)$ for $k\in K$ and $v^l\in \fu_{\fa^l}$, where ${\bf h}_{h, \fu_{\fa^l}}:=\Omega_{h,U_{A^l}}\circ \Exp_o|_{\fu_{\fa^l}}$. We shall show ${\bf h}_{h}^1(x)={\bf h}_{h}^2(x)$ for any $x\in \fu$.

Let $\Sigma^l$ be the associated restricted root system with respect to $\fa^l$, and $\{\wH_i^l\}_{i=1}^r$  the orthogonal basis of $\fa^l$ given in Lemma \ref{Lemma:orth}. Since $\fa^1$ and $\fa^2$ are maximal abelian subspaces in $\fp$, it is known that there exists an element $k_0\in K$ such that ${\rm Ad}(k_0)\fa^1=\fa^2$  (\cite[Theorem 2.3.16]{BCO}). Moreover,  ${\rm Ad}(k_0)$ maps any root vector $H_{\alpha}^1$ for $\alpha \in \Sigma^1$ to some root vector $H_{\beta}^2$ for $\beta\in \Sigma^2$.
In particular,  ${\rm Ad}(k_0)$ maps  $\{\wH_i^1\}_{i=1}^r$ to $\{\wH_i^2\}_{i=1}^r$ bijectively, and hence, we may assume  ${\rm Ad}(k_0)\wH_i^1=\wH_i^2$   for  each $i=1,\ldots, r$.

Then, any $x\in \fu$ is represented by $x={\rm Ad}(k^1)v^1={\rm Ad}(k^1(k_0)^{-1})v_2$ for some $k_1\in K$, $v_1\in \fu_{\fa^1}$ and $v_2\in \fu_{\fa^2}$, where $v^1$ and $v^2$ are related by ${\rm Ad}(k_0)v^1=v^2$. Setting $v^1=\sum_{i=1}^r x_i \widetilde{H}_i^1$, we have $v^2={\rm Ad}(k_0)v^1=\sum_{i=1}^r x_i \widetilde{H}_i^2$, and hence, 
\begin{align*}
{\bf h}_{h}^1(x)&={\bf h}_{h}^1({\rm Ad}(k^1)v^1)={\rm Ad}(k^1)\sum_{i=1}^r h\circ p_{1i}(x_1,\ldots, x_r)\wH_i^1\\
&={\rm Ad}(k^1k_0^{-1})\sum_{i=1}^r h\circ p_{1i}(x_1,\ldots, x_r)\wH_i^2={\bf h}_{h}^2({\rm Ad}(k^1k_0^{-1})v^2)={\bf h}_{h}^2(x).
\end{align*}
This proves the lemma.
 \end{proof}

\begin{example}{\rm
We consider a special case, that is, when $M=\C H^1(-C)=SU(1,1)/S(U(1)\times U(1))=\widehat{G}/\widehat{K}$. Let $\widehat{\fg}=\widehat{\fk}\oplus \widehat{\fp}$ be the associated Cartan decomposition of $\widehat{G}/\widehat{K}$. We  identify $\widehat{\fp}$ with $\C$ by 
$
\left[
\begin{array}{cc}
0 & \overline{z}\\
 z & 0
\end{array}
\right]
\mapsto z,
$
and denote the element in $\widehat{\fp}$ simply by $z\in \C$ via this identification. Note that any real 1-dimensional subspace becomes a maximal abelian subspace $\widehat{\fa}$. We may assume $\widehat{\fa}$ is given by the real axis in $\C$ due to Lemma \ref{Lemma:indep}. The isotropy subgroup $\widehat{K}$ is given by
\[
\widehat{K}=S(U(1)\times U(1))=
\Big{\{}
\widehat{k}_{\theta}:=
\begin{pmatrix}
e^{-i\theta/2} &\\
 & e^{i\theta/2}
\end{pmatrix}
: \theta\in \R
\Big{\}},
\] 
and  $\widehat{K}$ acts on $\C\simeq \widehat{\fp}$ by  $\widehat{k}_{\theta} \cdot z=e^{i\theta}z$. 

In this case, a $\widehat{K}$-invariant connected open subset $\widehat{U}$ in $\C H^1(-C)$ around $o$ is given by the image of $\Exp_o^{\C H^1}$ of an open disk $D_R=\{z\in \widehat{\fp}: |z|<R\}$ with radius $R\in (0,\infty]$, and the function satisfying properties (a) and (b) in Lemma \ref{Lemma:rest} is given by any odd function $\eta:(-R,R)\to\R$. We denote the $\widehat{K}$-equivariant map  given in Proposition \ref{Prop:omh} for this special case by $\Omega_{\eta,0}:\widehat{U}\to \widehat{\fp}$.

Since the expression $z=\widehat{k}_{\theta}x=e^{i\theta} x$ for some $\widehat{k}_{\theta}\in \widehat{K}$ and $x\in \widehat{\fa}$ coincides with the usual polar coordinate on $\C$,  the definition \eqref{Def:omh} shows that $\Omega_{\eta, 0}$ is precisely expressed by
\begin{align}\label{Def:ometa0}
\Omega_{\eta,0}(\Exp_oz)=\eta(|z|)\frac{z}{|z|}
\end{align}
with $\Omega_{\eta, 0}(o)=0$.  Namely,  $\Omega_{\eta,0}$ corresponds to a ``radial" map $z\mapsto \eta(|z|)\frac{z}{|z|}$.   
}
\end{example}

We return to the general case. Recall that the polydisk $A^\C\simeq A_1^\C\times \cdots\times A_r^\C\simeq \C H^1(-C)\times \cdots \times \C H^1(-C)$ lies in $M$ (see Proposition \ref{key1}). By Proposition \ref{Prop:omh}, we obtain the following  necessary geometric condition for the $K$-equivariant map.

\begin{proposition}\label{Prop:restri}
Let $M$ be an irreducible HSSNT, and $\Omega: U\to \fp$ a $K$-equivariant map such that $\Omega(U_A)\subseteq\fa$ for a $K$-invariant open subset $U=K\cdot U_A$ around $o$ with $U_A=U\cap A$. Then, there exists a single odd function $\eta: (-R, R)\to\R$ such that the restricted map $\Omega|_{U\cap A_i^\C}$ coincides with $\Omega_{\eta, 0}$ through the identification $A_i^\C\simeq \C H^1(-C)$ for any $i=1,\ldots, r$.
\end{proposition}
\begin{proof}
By Proposition \ref{Prop:omh}, we suppose $\Omega=\Omega_{h}$ for some function $h: \fu_\fa\to \R$ satisfying (a) and (b) described in Lemma \ref{Lemma:rest}. Then, we have
\begin{align*}
\Omega(\Exp_o(x_i\wH_i))=\sum_{j=1}^r h\circ p_{1j}(0,\ldots,0, x_i,0,\ldots, 0)\wH_j= h(x_i,0, \ldots, 0)\wH_i
\end{align*}
for each $i=1,\ldots, r$ due to the property (a) in Lemma \ref{Lemma:rest}. In particular, we have   $\Omega(U\cap A_i)\subseteq \fa_i$ for any $i=1,\ldots, r$.  
Recall that the abelian subgroup $K_i$ acts on $A_i^\C$ (resp. $\fa_i^\C$) transitively (see Subsection \ref{subsec:polydisk}) and $K_i\cdot A_i=A_i^\C$ (resp. ${\rm Ad}(K_i)\cdot \fa_i=\fa_i^\C$).  Since $U$ is $K$-invariant, we obtain $\Omega(U\cap A_i^\C)\subseteq\fa_i^\C$,  and $\Omega|_{U\cap A_i^\C}$ is a $K_i$-equivariant map such that $\Omega|_{U\cap A_i^\C}(U\cap A_i)\subseteq \fa_i$. Thus, by using the equivariant identification $A_i^\C\simeq \C H^1(-C)$, we see that $\Omega|_{U\cap A_i^\C}$ coincides with $\Omega_{\eta_i,0}$ for some odd function $\eta_i$, and $\eta_i(x)=\eta(x):=h(x,0,\ldots, 0)$ for any $i=1,\ldots, r$. This proves the proposition.
\end{proof}

\subsection{Strongly diagonal realization}\label{subsec:SDR}
Proposition \ref{Prop:restri} brings us the following  simple way of construction of a special type of $K$-equivariant map $\Omega: U\to \fp$ satisfying $\Omega(U_A)\subseteq\fa$  from an open neighborhood $U$ around $o\in M$.

First, we take arbitrary continuous {\it odd} function $\eta:(-R, R)\to\R$ for $R\in (0, \infty]$, and we put $\widehat{U}_{\eta}:=\Exp_o^{\C H^1}(D_R)\subset \C H^1(-C)$, where $D_R:=\{z\in \widehat{\fp}: |z|<R\}$.  Then, we obtain the associated $\widehat{K}$-equivariant  map $\Omega_{\eta,0}: \widehat{U}_{\eta}\to \widehat{\fp}$ defined by a ``radial" map \eqref{Def:ometa0}, that is, a $\widehat{K}$-equivariant map  such that $\widehat{U}_{\eta, \widehat{A}}=\widehat{U}_{\eta}\cap \widehat{A}$ is mapped into $\widehat{\fa}$.

Next, we extend $\Omega_{\eta, 0}$ to a $T^r$-equivariant map from an open subset of the polydisk $A^\C\simeq \C H^1(-C)\times \cdots \times \C H^1(-C)$ into $\fa^\C$ as follows:  By using the canonical identification $\C H^1(-C)\simeq A_i^\C$, we identify $\widehat{U}_{\eta}$ with  an open subset  $U_{\eta, i}\subseteq A_i^\C$  for each $i=1,\ldots, r$, and put $U_{\eta, A^\C}:=U_{\eta, 1}\times \cdots \times U_{\eta, r}$.  Then, $U_{\eta, A^\C}$ is identified with a $T^r$-invariant open subset of $A^\C$ around $o$. According to Proposition \ref{Prop:restri}, we define a $T^r$-equivariant map on $U_{\eta, A^\C}$  by the direct product
\begin{align}\label{diag}
\Omega_{\eta, U_{\eta, A^\C}}:=\Omega_{\eta,0}\times\cdots\times\Omega_{\eta,0}: U_{\eta, A^\C}\to \fa^\C.
\end{align}
Note that, by putting $U_{\eta, A}=U_{\eta, A^\C}\cap A$, the restricted map $\Omega_{\eta, U_{\eta, A}}:=\Omega_{\eta, U_{\eta, A^\C}}|_{U_{\eta, A}}$  is expressed by  
\begin{align*}
\Omega_{\eta, U_{\eta, A}}\Big(\Exp_o\Big(\sum_{i=1}^r x_i\wH_i\Big)\Big)=\sum_{i=1}^r \eta(x_i)\wH_i.
\end{align*}
Then, it follows that $\Omega_{\eta, U_{\eta, A}}(U_{\eta, A})\subseteq \fa$ and $\Omega_{\eta, U_{\eta, A}}$ is $\mathcal{W}$-equivariant  by Lemma \ref{Lemma:rest}.

Therefore, by Lemma \ref{Lemma:welldef}, we finally obtain a $K$-equivariant map $\Omega_{\eta}:U_{\eta}\to \fp$ on a $K$-invariant connected open subset $U_{\eta}:=K\cdot U_{\eta, A^\C}=K\cdot U_{\eta, A}$ of $M$ around $o$, which is defined by
\[
\Omega_{\eta}(k\Exp_ov):={\rm Ad}(k)\Omega_{\eta, U_{\eta, A}}(\Exp_ov)={\rm Ad}(k)\sum_{i=1}^r \eta(x_i)\wH_i
\]
for $k\in K$ and $v=\sum_{i=1}^r x_i\wH_i\in \fu_{\eta, \fa}=\Log_o(U_{\eta, A})$.  

By Lemma \ref{Lemma:indep}, the definition of $\Omega_{\eta}$ depends only on $\eta$. Note that, if $\eta$ is defined on $\R=(-\infty, \infty)$, then $\widehat{U}_{\eta}=\C H^1(-C)$, and hence, $U_{\eta}=M$. Moreover, if $\eta$ is an injective function, then $\Omega_{\eta, 0}$ is an embedding, and so is $\Omega_{\eta}$.

\begin{definition}\label{Def:diag}
We call  the $K$-equivariant map $\Omega_{\eta}:U_{\eta}\to \fp$ a {\rm strongly diagonal map}  associated with  an 
 odd function $\eta:(-R,R)\to \R$. If $U_{\eta}=M$ and $\Omega_{\eta}$ is an embedding (i.e. if $\eta$ is an injective odd function on $\R$),  we say $\Omega_{\eta}: M\to \fp$ a {\rm strongly diagonal realization of $M$}.
\end{definition}

\begin{remark}
{\rm
A  map corresponding to $\Omega_{\eta}$ has been considered in \cite[Section 1.6]{DLR}, \cite{LM} and \cite[Section 3.18]{Loos} by using the Jordan triple system associated with $M$. }
\end{remark}

By the above arguments,  the strongly diagonal realization of $M$ is characterized by the following three properties: {\rm (i)}  $\Omega: M\to \fp$ is a $K$-equivariant embedding, {\rm (ii)} $\Omega(A)\subseteq \fa$ and {\rm (iii)} $\Omega|_{A^\C}$ is diagonal in the sense of \eqref{diag} with respect to the canonical splitting $A^\C\simeq A_1^\C\times \cdots \times A_r^\C$ given in Proposition \ref{key1}.

By definition, we see that
\begin{align*}
\Omega_{\eta}(U_{\eta})=D_{\eta}:={\rm Ad}(K)(\square_{\eta,\fa}),\ {\rm where}\ \square_{\eta, \fa}:=\Omega_{\eta}(U_{\eta, A})=\Big{\{}\sum_{i=1}^{r} x_i \wH_i\in \fa:\ |x_i|<s_{\eta}\Big{\}},
\end{align*}
where we put $s_{\eta}:={\rm sup}\{\eta(x): x\in (-R, R)\}$. Note that $D_{\eta}=\fp$ if and only if $\eta:(-R, R)\to \R$ is a surjective odd function.  Otherwise, $D_{\eta}$ is a $K$-invariant bounded domain in $\fp$.  This description is a generalization of the Harish-Chandra realization of $M$ (see \cite[Corollary 7.17 in Ch. VIII]{Hel}).

We mention the differentiability of the map $\Omega_{\eta}$.

\begin{lemma}
If the  odd function $\eta$ is $C^{1}$, then $\Omega_{\eta}$ is also $C^{1}$.
\end{lemma}

See also \cite[Section 1.6]{DLR} on the smoothness of the corresponding map in the context of Jordan triple systems.  This lemma also follows from  similar arguments, and thus, we left the proof to the reader.

If $\Omega_{\eta}$ is an embedding, the inverse map is expressed by
\[
\Omega_{\eta}^{-1}\Big({\rm Ad}(k)\sum_{i=1}^r y_i\wH_i\Big)=k\Exp_o\Big(\sum_{i=1}^r \eta^{-1}(y_i)\wH_i\Big)
\]
for $k\in K$ and $y_i\in \eta((-R, R))$. If furthermore $\eta$ is $C^1$, $\Omega_{\eta}^{-1}$ is also a $C^1$ map, and hence, $\Omega_{\eta}: U\to D_{\eta}$ is a $C^1$-diffeomorphism.

\subsection{Dual map}\label{subsec:dual}
The above construction of $K$-equivariant maps can be applied to the compact dual $M^*$ of $M$ with a slight modification.  

It is known that any HSSNT $M=G/K$ is associated with a Hermitian symmetric space of compact type $M^*=G^*/K$ which we call the {\it compact dual} of $M$, where $G^*$ is a connected compact semisimple Lie group.  We denote the Lie algebra of $G^*$ by $\fg^*$.  Let $\fg=\fk\oplus \fp$ be the Cartan decomposition associated with $M=G/K$, and $\fg_\C=\fg\oplus \sqrt{-1}\fg$ the complexification of $\fg$. Then, $\fg^*$ is given by  $\fg^*=\fk \oplus \fp^*$, where $\fp^*:=\sqrt{-1} \fp$ which is identified with the tangent space $T_{o^*}M^*$ at the origin $o^*\in M^*$.  In the following, we often denote the origin $o^*$ by $o$ if there is no confusion. For example, we denote the  the exponential map of $M^*$ at the origin by $\Exp_o^*$.

We define  the bilinear form on $\fg^*$ by $\langle \cdot, \cdot \rangle^* :=-B^*(\cdot , \cdot)$, where $B^*$ is the Killing-Cartan form on $\fg^*$. Since $B^*$ is negative definite on the compact semisimple Lie algebra $\fg^*$, $\langle, \rangle^*$ defines a positive definite inner product on $\fg^*$. Moreover, the restriction $\langle, \rangle^*|_{\fp^*}$ is extended to the $G^*$-invariant Riemannian metric on $M^*$. Note that the linear isomorphism  $\fg \rightarrow \fg^*: X^{\fk} + X^{\fp} \mapsto X^{\fk} + \sqrt{-1} X^{\fp}$ is an isometry.

The complex structure $J_o^*$ on $\fp^*$ is given by $J_o^*={\rm ad}(\zeta)|_{\fp^*}$, where $\zeta\in \mathfrak{c(k)}$ is the same element  given in Subsection \ref{Prelim:root}. We put $\omega_o^*:=\langle J_o^*\cdot,\cdot \rangle^*$. Then, $(J_o^*, \omega_o^*)$ is extended to the canonical K\"ahler structure $(J^*, \omega^*)$ on $M^*$ such that $(M^*, J^*, \omega^*)$ becomes a compact K\"ahler-Einstein manifold of positive Ricci curvature.

We often identify $\fp$ and $\fp^*$ by the linear isomorphism $\sqrt{-1}: \fp\to \fp^*$, and  denote the element in $\fp^*$ by $X^*=\sqrt{-1}X$ for $X\in \fp$.  A maximal abelian subspace in $\fp^*$ is given by $\fa^*:=\sqrt{-1}\fa$, and then, $\{\wH_i^*=\sqrt{-1}\wH_i\}_{i=1}^r$ becomes an orthogonal basis of $\fa^*$.  Moreover, we put $\fa_i^*:=\R \wH_i^*$, $A_i^*:=\exp(\fa_i^*)\cdot o^*$ and  $A^*:=\exp(\fa^*)\cdot o^*$.  Note that we have a natural splitting $A^*\simeq A_1^*\times \cdots \times A_r^*$ as a Riemannian manifold, and each $A_i^*$ is isometric to $S^1(1/\sqrt{C})=\{z\in \C: |z|=1/\sqrt{C}\}$, where $C=\|\wH_i^*\|^2=\|\wH_i\|^2$, and therefore, $A^*$ is isometric to an $r$-dimensional flat torus. 

Analogous to the Polydisk Theorem (Proposition \ref{key1}), we see that $A^*$ has a complexification, namely, $(A^*)^\C:=\exp(\fa^*\oplus J_o^*\fa^*)\cdot o^*$ is a totally geodesic complex submanifold in $M^*$, and there exists a canonical splitting 
\[
F^*: (A^*)^\C \simeq (A_1^*)^\C\times \cdots \times (A_r^*)^\C,
\]
as a K\"ahler manifold, where $(A_i^*)^\C=\exp(\fa_i^*\oplus J_o^*\fa_i^*)\cdot o^*$. Moreover, $(A_i^*)^\C$ is holomorphic isometric to the complex projective space $\C P^1(C)$ of constant holomorphic sectional curvature $C$, and $F^*$ becomes  a holomorphic isometry.

We denote the cut locus of a Riemannian manifold $X$ at a point $x\in X$ by ${\rm Cut}_x(X)$.  It is known that we have ${\rm Cut}_{o}(M^*)=K\cdot {\rm Cut}_{o}(A^*)$ (\cite{Sakai}, see also Lemma 6 in \cite{Tasaki}) and 
\[
{\rm Cut}_{o}(A^*)=\bigcup_{i=1}^r A_1^*\times \cdots \times {\rm Cut}_{o_i}(A_i^*)\times \cdots \times A_r^*,
\]
where we put $o=(o_1,\ldots, o_r)$ via the identification $A^*\simeq A_1^*\times\cdots\times A_r^*$. Since $A_i^*$ is isometric to $S^1(1/\sqrt{C})$, the injective radius at $o_i\in A_i^*$ is given by $\pi/\sqrt{C}$. Note that any element in $\fa_i^*$ is expressed by $x\wH_i^*$, and $\|x\wH_i^*\|<\pi/\sqrt{C}$ is equivalent to $|x|<\pi/2$ since $\|\wH_i^*\|^2=4/C$.
Thus, we put
\[
\square_{\fa^*}^*:=\Big\{\sum_{i=1}^r x_i\wH_i^*: |x_i|<\frac{\pi}{2}\ \textup{for each $i=1,\ldots, r$}\Big\},
\]
then $\Exp_{o}^{A^*}|_{\square_{\fa^*}^*}: \square_{\fa^*}^*\to A^*\setminus {\rm Cut}_{o}(A^*)$ is a diffeomorphism, where $\Exp_{o}^{A^*}$ is the exponential map of $A^*$ at ${o}$. Note that we have $\Exp_{o}^*|_{A^*}=\Exp_{o}^{A^*}$ since  $A^*$ is totally geodesic.

We put $(\fp^*)^o:=K\cdot \square_{\fa^*}^*$ and $(M^*)^o:=M^*\setminus {\rm Cut}_{o}(M^*)$. It is easy to see $(M^*)^o=K\cdot (A^*\setminus {\rm Cut}_{o}(A^*))$ since $M^*=K\cdot A^*$ and ${\rm Cut}_o(M^*)=K\cdot {\rm Cut}_o(A^*)$. Moreover, we see
\[
\Exp_o^*|_{(\fp^*)^o}: (\fp^*)^o\to (M^*)^o
\]
is a diffeomorphism.

Now, we are ready to define a $K$-equivariant map for $M^*$.  Let us take  a continuous odd function $\eta^*: (-R^*, R^*)\to \R$  defined on the interval $(-R^*,R^*)$ satisfying $0<R^*\leq \pi/2$. We put 
\begin{align*}
\square_{R^*,\fa^*}^*&:=\Big\{\sum_{i=1}^r x_i\wH_i^*: |x_i|<R^*\ \textup{for each $i=1,\ldots, r$}\Big\},\\
U^*_{\eta^*,A^*}&:=\Exp_o^*(\square_{R^*,\fa^*}^*)\quad {\rm and}\quad U_{\eta^*}^*:=K\cdot U^*_{\eta^*,A^*}.
\end{align*}
Then, $U^*_{\eta^*}$ is a $K$-invariant connected open subset in $(M^*)^o$ containing $o\in M^*$.   Note that $U^*_{\eta^*}=(M^*)^o$ if and only if $R^*=\pi/2$.

We put
\begin{align*}
\Omega_{\eta^*}^*: U^*_{\eta^*}\to \fp^*,\quad 
\Omega_{\eta^*}^*\Big(k\cdot \Exp_o^*\Big(\sum_{i=1}^r x_i\wH_i^*\Big)\Big):={\rm Ad}(k)\cdot \sum_{i=1}^r \eta^*(x_i)\wH_i^*
\end{align*}
for $k\in K$ and $\sum_{i=1}^r x_i\wH_i^*\in \square_{R^*,\fa^*}^*$. Then, by a similar argument given in Lemma \ref{Lemma:welldef}, $\Omega_{\eta^*}^*$ is a well-defined $K$-equivariant map. We call $\Omega_{\eta^*}^*$ a {\it strongly diagonal map} associated with $\eta^*$.  By definition, 
\begin{align*}
\Omega^*_{\eta^*}(U^*_{\eta^*})=D^*_{\eta^*}:={\rm Ad}(K)(\square^*_{\eta^*,\fa^*}),\ {\rm where}\ \square^*_{\eta^*, \fa^*}:=\Omega_{\eta^*}^*(U_{\eta^*, A^*})=\Big{\{}\sum_{i=1}^{r} x_i \wH_i^*\in \fa:\ |x_i|<s_{\eta^*}\Big{\}},
\end{align*}
where $s_{\eta^*}:={\rm sup}\{\eta^*(x): x\in (-R^*, R^*)\}$. 

We shall consider a special situation. 
Let $\Omega_{\eta}: U_{\eta}\to \fp$ be a strongly diagonal map of $M$ associated with an  odd function $\eta:(-R,R)\to \R$. Suppose that $\eta$ is a real analytic function on $(-R, R)$ such that  $\eta$ is expressed by
\begin{align*}
\eta(x)=\sum_{k=0}^\infty a_kx^{2k+1}
\end{align*}
if $|x|<\epsilon$ for a sufficiently small $\epsilon>0$. We define a ``conjugate" odd function  by
\begin{align*}
\eta^*(x):=\sum_{k=0}^\infty a_k(-1)^kx^{2k+1},
\end{align*}
for $|x|<\epsilon$, or equivalently, $\eta^*(x):=-\sqrt{-1}\eta^\C(\sqrt{-1}x)$, where $\eta^\C(z)=\sum_{k=0}^\infty a_kz^{2k+1}$  for $z\in \C$ with $|z|<\epsilon$. We denote the unique analytic continuation of $\eta^*$ by the same $\eta^*$.  We put $\widetilde{R}^*:={\rm sup}\{R>0: \textup{$\eta^*$ is defined on $(-R,R)$}\}$, and define $R^*:={\rm min}\{\widetilde{R}^*, \pi/2\}$. Then, $\eta^*$ is also an odd function on $(-R^*, R^*)$, and we say the function $\eta^*: (-R^*,R^*)\to \R$ the {\it dual function} of $\eta: (-R,R)\to \R$.

\begin{definition}\label{Def:dual}
Suppose $\eta: (-R, R)\to \R$ is a real analytic function. For the dual function $\eta^*$ of $\eta$, we call the associated $K$-equivariant map $\Omega_{\eta^*}^*: U^*_{\eta^*}\to \fp^*$  the {\rm dual map} of $\Omega_{\eta}:U_{\eta}\to \fp$.
\end{definition}

\subsection{Examples}
We give some examples of  strongly diagonal realizations of $M$ and their dual maps.

{\rm 
(1) If $\eta$ is the linear odd  function $\eta(x)=x$, then $\Omega_{\eta, 0}=\Log_o^{\C H^1}$, i.e. the inverse of $\Exp_o^{\C H^1}: \widehat{\fp}\to \C H^1(-C)$, and the associated strongly diagonal realization of $M$ coincides with $\Log_o: M\to \fp$. The dual function of $\eta$ is  $\eta^*(x)=\eta(x)=x$, and the dual map of $\Omega_{\eta}$ is given by $\Omega_{\eta^*}^*=\Log_{o^*}^*: M^*\setminus {\rm Cut}_{o^*}(M^*)\to \fp^*$. 

(2) It is well-known that the complex hyperbolic plane $\C H^1(-C)$ is biholomorphic to an open unit disk
$
D:=\{z\in \C: |z|<1\}
$
equipped with the standard complex structure $J_0$  on $\C$.  More precisely, the holomorphic diffeomorphism is given by
\begin{align*}
{\Psi}_0: (\C H^1(-C), \widehat{J}) \to (D, J_0),\quad [1:z]\mapsto z,
\end{align*}
where we used the homogeneous coordinate on $\C H^1$.
On the other hand,  it is known that $\C H^1(-C)$ is symplectomorphic to $\C$, and a canonical symplectomorphism is given by
\begin{align*}
{\Phi}_0: (\C H^1(-C), \widehat{\omega})\to (\C, \omega_C),\quad [1:z]\mapsto \frac{z}{\sqrt{1-|z|^2}},
\end{align*}
where $\omega_C:=\frac{2\sqrt{-1}}{C}dz\wedge d\overline{z}$, that is, a constant multiple of the standard symplectic form on $\C$ (see Lemma \ref{Lemma:unich1} in the next section  for a proof).

Regarding $\C H^1(-C)$ as a Hermitian symmetric space $\widehat{G}/\widehat{K}=SU(1,1)/S(U(1)\times U(1))$, the map $\Psi_0$ and $\Phi_0$ are regarded as embeddings from $\widehat{G}/\widehat{K}$ to $\widehat{\fp}\simeq \C$.
More precisely, since $\Exp_oz\in \widehat{G}/\widehat{K}$ is identified with $\Big[1: \frac{\tanh |z|}{|z|}z\Big]\in \C H^1(-C)$ for any $z\in \widehat{\fp}\setminus \{0\}$, we have
\begin{align*}
\Psi_0(\Exp_oz)=
\tanh |z|\cdot\frac{z}{|z|},\quad 
\Phi_0(\Exp_oz)=
 \sinh |z|\cdot\frac{z}{|z|},
\end{align*}
and $\Psi_0(o)=\Phi_0(o)=0$ at the origin $o\in \widehat{G}/\widehat{K}$. In particular, the holomorphic embedding $\Psi_0$ is given by $\Psi_0=\Omega_{\tanh,0}$, and the symplectomorphism $\Phi_0$ is given by $\Phi_0=\Omega_{\sinh,0}$.

The associated strongly diagonal realizations $\Psi:=\Omega_{\tanh}$ and $\Phi:=\Omega_{\sinh}$ of $M$ are given by
\begin{align}
\label{hdiff2}
 \Psi(k\Exp_ov)={\rm Ad}(k)\sum_{i=1}^r (\tanh x_i) \wH_i,\quad
\Phi(k\Exp_ov)={\rm Ad}(k)\sum_{i=1}^r(\sinh x_i)\wH_i
\end{align}
for  $k\in K$ and $v=\sum_{i=1}^r x_i \wH_i \in \fa$, respectively.
We will prove in Section \ref{Proof:mainthm1} that $\Psi: M\to \fp$ is a holomorphic embedding onto a bounded domain $D\subset \fp$ and $\Phi: M\to \fp$ is a symplectomorphism.  In fact,  the map $\Psi$ (resp. $\Phi$) coincides with the Harish-Chandra (resp. Di Scala-Loi-Roos) realization given in \cite{HC} (resp. in \cite{DL, DLR}) under appropriate identifications of spaces. We confirm this in Appendix \ref{A1}--\ref{A2}.

The dual functions of $\tanh x$ and $\sinh x$ are given by $\tan x$  and $\sin x$ for $x\in (-\pi/2, \pi/2)$, respectively, and the dual maps of $\Psi=\Omega_{\tanh}$ and $\Phi=\Omega_{\sinh}$ are expressed by 
\begin{align*}
 \Psi^*(k\Exp_{o^*}^*v^*)={\rm Ad}(k)\sum_{i=1}^r (\tan x_i) \wH_i^*,\quad
\Phi^*(k\Exp_{o^*}^*v^*)={\rm Ad}(k)\sum_{i=1}^r(\sin x_i)\wH_i^*,
\end{align*}
for $k\in K$ and $v^*=\sum_{i=1}^r x_i\wH_i^*\in \square_{\fa^*}^*$, respectively. We will show in the next section that the dual map $\Psi^*=\Omega_{\tan}^*$ is also a holomorphic embedding from $(M^*)^o=M^*\setminus {\rm Cut}_{o^*}(M^*)$ onto $\fp^*$, and  $\Phi^*=\Omega_{\sin}^*$ is a symplectomorphism from $(M^*)^o$ onto a bounded domain in $\fp^*$.

(3) In \cite{LM}, Loi-Mossa introduced the notion of {\it diastatic exponential map} of a K\"ahler manifold $X$,  that is, a smooth map $f$ from a neighborhood $W$ of the origin $0\in T_xX$ into $X$ satisfying $(df_x)_0={\rm id}_{T_xM}$ and  $D_x\circ f_x=\|\cdot\|^2_x$, where $D_x:M\to \R$ is Calabi's  diastasis function at $x$ (see \cite{LM} for details). We will denote the diastatic exponential  map by ${\rm DE}_x$ in the present paper.

Regarding an HSSNT $M$ as a bounded domain $D$ via the the Harish-Chandra realization, Loi-Mossa explicitly constructed the {diastatic exponential map} ${\rm DE}_o: T_oD\to D$ at the origin $o\in D$ as a global diffeomorphism from $T_oD$ onto $D$.   It turns out that, under the identification $D\simeq M$, the inverse map ${\rm DL}_o:={\rm DE}_o^{-1}: D\to T_oD$ is given by a strongly diagonal realization $\Omega_{\eta}: M\to \fp$ associated with some injective odd function $\eta$  such that $\Omega_{\eta,0}: \C H^1(-C)\to \widehat{\fp}$ is regarded as the diastatic exponential map for $\C H^1(-C)$.  Moreover, the dual map $\Omega_{\eta^*}^*$ of the map $\Omega_{\eta}$ coincides with the diastatic exponential map for the compact dual $M^*$ constructed in \cite{LM}. We give the details in Appendix \ref{A3}.
}

\section{Holomorphic/symplectic realizations}\label{Section:holsymp}
Let $M$ be an irreducible HSSNT.  In this section, we show a canonical holomorphic/symplectic realization of $M$ in $\fp$ is obtained by a strongly diagonal realization. Throughout this section, we suppose $\Omega_{\eta}:U_{\eta}\to \fp$ is a $C^1$ strongly diagonal embedding of an open neighborhood $U_{\eta}$ of $o\in M$ associated with an injective odd function $\eta$.

\subsection{Complex/symplectic structure} \label{Subsec:cs}

In this subsection, we give a formula of the expression of the K\"ahler structure $(\omega, J)$ of $M$ on $D_{\eta}=\Omega_{\eta}(U_{\eta})\subset \fp$.  We show the formula has a better expression on an open dense subset of the  section due to the strong diagonality.  

Let $w$ be an element in $D_{\eta}$.  Then, by Zorn's Lemma, $w$ is contained in a maximal abelian subspace $\fa$ of $\fp$. We fix $\fa$ containing $w$ and denote the associated restricted root system by $\Sigma$.
According to Lemma \ref{Lemma:J}, we put
\[
\fq_{\gamma_i}:=\fp_i\oplus \fa_i=\fp_i\oplus J_o\fp_i(=\fa_i^\C),\quad
\fq_{\lambda}:=\fp_{\lambda}\oplus \fp_{\overline{\lambda}}=\fp_{\lambda}\oplus J_o\fp_{\lambda},\quad
\fq_{\epsilon}:=\fp_{\epsilon}
\]
for $\gamma_i\in \Gamma$, $\lambda\in \Lambda$ and $\epsilon\in E$. Then, we obtain the following orthogonal decomposition of $\fp$ into $J_o$-complex subspaces:
\begin{align}\label{Def:decp}
\fp=\bigoplus_{\alpha\in \Gamma\sqcup \Lambda\sqcup E} \fq_{\alpha}
\end{align}
Note that, ${\rm dim}_{\C}\fq_{\gamma_i}=m_{\gamma_i}=1$, ${\rm dim}_{\C}\fq_{\lambda}=m_{\lambda}$ and ${\rm dim}_{\C}\fq_{\epsilon}=m_{\epsilon}/2$, where $m_{\alpha}$ is the multiplicity of $\alpha\in \Sigma_+$.

We shall derive formulas of the complex structure and the symplectic structure on $D_{\eta}\cap\fa^{pri}$, where $\fa^{pri}$ is the set of principal elements in $\fa$ of the $K$-action on $\fp$.  Note that $D_{\eta}\cap\fa^{pri}$ becomes an open dense subset in $D_{\eta}$.

\begin{proposition}\label{Prop:JO}
Let $\Omega_{\eta}: U_{\eta}\to \fp$ be a $C^1$ strongly diagonal embedding associated with an injective odd function $\eta$.
Then, the induced K\"ahler structure $(J, \omega)$ at $w\in D_{\eta}\cap \fa^{pri}$ via the map $\Omega_{\eta}$ is expressed as follows: 
With respect to the decomposition \eqref{Def:decp}, $J_w$ and $\omega_w$ are decomposed into direct sums
\begin{align}
\label{Expr:JO}
J_w=\bigoplus_{\alpha\in \Gamma\sqcup \Lambda\sqcup E} J_{\eta, \alpha}(w),\quad \omega_w=\bigoplus_{\alpha\in \Gamma\sqcup \Lambda \sqcup E}\omega_{\eta, \alpha}(w),
\end{align}
and each $J_{\eta, \alpha}(w)$ and $\omega_{\eta, \alpha}(w)$ are given by
\begin{align*}
J_{\eta, \alpha}(w)&=
\begin{cases}
G_{\eta,\alpha}(w)J_o|_{\fp_{\alpha}} & \textup{on}\ \fp_{\alpha}\\
-G_{\eta,\alpha}(w)^{-1}J_o|_{J_o\fp_{\alpha}}& \textup{on}\ J_o\fp_{\alpha}
\end{cases}
\quad {\rm if}\ \alpha\in \Gamma\sqcup \Lambda,\quad 
J_{\eta, \alpha}(w)=J_o|_{\fq_{\alpha}}\quad {\rm if}\ \alpha\in E,\\
\omega_{\eta, \alpha}(w)&=F_{\eta, \alpha}(w)\omega_o|_{\fq_{\alpha}}
\end{align*}
 for some functions $G_{\eta,\alpha}: D_{\eta}\cap \fa^{pri}\to \R$ and $F_{\eta, \alpha}:D_{\eta}\cap \fa^{pri}\to \R$ depending on $\eta$ and $\alpha$, respectively.
\end{proposition}

The functions  $G_{\eta, \alpha}$ and $F_{\eta,\alpha}$  will be  determined explicitly in the proof of this proposition.

To prove Proposition \ref{Prop:JO}, we compute two linear isomorphisms $({\exp v})_*:  T_oM\simeq \fp\to T_xM$ and $(\Omega_{\eta})_*: T_xM\to T_{\Omega_{\eta}(x)}D_{\eta}$ for any principal element $x=\exp v\cdot o\in A$.

Let $x=\exp v\cdot o$ be a principal element for the $K$-action contained in $A$, that is, an element in $A$ such that $v\in \fa^{pri}=\{v\in\fa: \alpha(v)\neq 0\ \textup{for any}\ \alpha\in \Sigma\}$. Then, $T_xM$ is decomposed into $T_xM=T_x^{\perp}(K\cdot x)\oplus T_x(K\cdot x)=T_xA\oplus T_x(K\cdot x)$. 
We denote the fundamental vector field of the $G$-action on $M$ at $x\in M$ by $X_x^\da$ for $X\in \fg$, i.e. $X_x^\da:=(d/dt)|_{t=0} \exp tX\cdot x$. Then, we have
$T_xA={\{}H^\da_x: H\in \fa{\}}$ and $T_x(K\cdot x)={\{}X^\da_x: X\in \fk{\}}$.

Recall that $\fk$ is decomposed into $\fk=\fk_0\oplus \bigoplus_{\alpha\in \Sigma_+}\fk_{\alpha}$, and we have
\begin{align}\label{bra}
[v, X_0]=0, \quad [v, X_{\alpha}^{\fk}]=\alpha(v)X_{\alpha}^{\fp}, \quad  [v, X_{\alpha}^{\fp}]=\alpha(v)X_{\alpha}^{\fk}
\end{align}
for any $\alpha\in \Sigma_+$ by \eqref{feq2}.

\begin{lemma}\label{Lemma:expush}
Suppose $x=\exp v\cdot o$ is a principal element of the $K$-action contained in $A$.  Then we have
\begin{align}\label{expush1}
(\exp v)_*\wH_i=(\wH_i)_x^\da,\quad 
(\exp v)_*X_{\alpha}^{\fp}=-\dfrac{1}{\sinh(\alpha(v))}(X_{\alpha}^{\fk})_x^\da
\end{align}
for any $i=1,\ldots, r$, $X_{\alpha}^{\fp}\in \fp_{\alpha}$ and  $\alpha\in \Sigma_+$.
\end{lemma}

\begin{proof}
The first equation of \eqref{expush1} is obvious since $H,v\in \fa$ and $\fa$ is abelian. 
To prove the second equation, we compute $(\exp v^{-1})_*{X}^\da_x$ for $X\in \fk$. We have
\begin{align*}
(\exp v^{-1})_*{X}^\da_x&=\frac{d}{dt}\Big{|}_{t=0}\exp v^{-1}\exp tX\exp v\cdot o=\frac{d}{dt}\Big{|}_{t=0} \exp(t{\rm Ad}(\exp v^{-1})X)\cdot o=({\rm Ad}(\exp(-v))X)^{\top_{\fp}}\nonumber
\end{align*}
for any $X\in \fk$. Putting $X=X_{\alpha}^\fk$, we see

\begin{align*}
{\rm Ad}(\exp(-v))X_\alpha^\fk=\sum_{k=0}^{\infty}\frac{(-1)^k}{k!}{\rm ad}(v)^kX_{\alpha}^\fk=\cosh(\alpha(v))X_{\alpha}^\fk-\sinh(\alpha(v))X_{\alpha}^\fp
\end{align*}
by  \eqref{bra}.
Therefore, we obtain
$
(\exp v^{-1})_*{(X_{\alpha}^{\fk})}^\da_x=-\sinh(\alpha(v))X_{\alpha}^{\fp}.
$
This proves the second formula of \eqref{expush1}.
\end{proof}

\begin{lemma}\label{Lem:Ompush}
Suppose $x=\exp v\cdot o$ is a principal element of the $K$-action contained in  $A\cap U_{\eta}$. We put $w:=\Omega_{\eta}(x)\in \fp$. Then, we have 
\begin{align}
\label{Ompush1}
(\Omega_{\eta})_*(\wH_i)_x^\da
=\eta'\Big(\dfrac{\gamma_i(v)}{2}\Big)\wH_i, \quad
(\Omega_{\eta})_*(X_{\alpha}^{\fk})_x^\da=-\alpha(w)X_{\alpha}^\fp
\end{align}
for any $i=1,\ldots, r$, $X_{\alpha}^{\fk} \in \fk_{\alpha}$ and $\alpha\in \Sigma_+$.
\end{lemma}

\begin{proof}
We put $v=\sum_{i=1}^r x_i\wH_i$. Since $\wH_i, v\in \fa$ and $\fa$ is abelian, we see
\begin{align*}
(\Omega_{\eta})_*(\wH_i)_x^\da&=\frac{d}{dt}\Big|_{t=0} \Omega_{\eta}(\exp t\wH_i\exp v\cdot o)=\frac{d}{dt}\Big|_{t=0} \Omega_{\eta}(\exp (t\wH_i+v)\cdot o)\\
&=\frac{d}{dt}\Big|_{t=0} \eta(t+x_i)\wH_i+({\rm const.})=\eta'(x_i) \wH_i.
\end{align*}
Since $x_i=\langle v, \wH_i\rangle/\|\wH_i\|^2=\gamma_i(v)/2$, we obtain the first equation of \eqref{Ompush1}.

On the other hand,  by definition of $\Omega_{\eta}$ and \eqref{bra}, we have
\begin{align*}
({\Omega_{\eta}})_*(X_{\alpha}^\fk)^\da_x&=\frac{d}{dt}\Big{|}_{t=0}\Omega_{\eta}(\exp tX_{\alpha}^\fk\exp v\cdot o)=\frac{d}{dt}\Big{|}_{t=0}\sum_{i=1}^r\eta(x_i){\rm Ad}(\exp tX_{\alpha}^\fk)\wH_i\\
&=\sum_{i=1}^r  \eta(x_i) [X_{\alpha}^\fk, \wH_i]=\Big(-\sum_{i=1}^r \eta(x_i)\alpha(\wH_i)\Big)X_{\alpha}^{\fp}=-\alpha(w)X_{\alpha}^\fp
\nonumber
\end{align*}
since $w=\sum_{i=1}^r \eta(x_i)\wH_i$. This proves the formula.
 \end{proof}

Now, we give a proof of Proposition \ref{Prop:JO}.

\begin{proof}[Proof of Proposition \ref{Prop:JO}]
Let $w\in D_{\eta}$ be an arbitrary principal element of the $K$-action on $D_{\eta}$. As described earlier, we may assume $w\in \fa$ for some maximal abelian subspace $\fa$ in $\fp$.  We put $x=\exp v \cdot o=\Omega_{\eta}^{-1}(w)\in A$, where $v$ is an element in $\fa$ which is uniquely determined by $w$.  Note that $w$ is a principal element if and only if so is $x$ for the $K$-action on $M$ since $\Omega_{\eta}$ is $K$-equivariant.

First, we consider the expression of the induced complex structure $J_w$ at $w$.  Suppose $\alpha\in \Sigma_+\setminus \Gamma$. Then, by Lemmas \ref{Lemma:expush} and \ref{Lem:Ompush}, we have
\begin{align}\label{Omex1}
(\Omega_{\eta})_*\circ (\exp v)_* X_{\alpha}^{\fp}=\frac{\alpha(w)}{\sinh(\alpha(v))}X_{\alpha}^{\fp}
\end{align}
for any $X_{\alpha}^{\fp}\in \fp_{\alpha}$.
On the other hand, Lemma \ref{Lemma:J} shows that we have $J_oX_{\alpha}^{\fp}\in  \fp_{\overline{\alpha}}$, where we put $\overline{\epsilon}=\epsilon$. Therefore, we have 
\begin{align}\label{Omex2}
(\Omega_{\eta})_*\circ (\exp v)_* (J_oX_{\alpha}^{\fp})=\frac{\overline{\alpha}(w)}{\sinh(\overline{\alpha}(v))}J_oX_{\alpha}^{\fp}.
\end{align}
The induced complex structure at $T_wD_{\eta}$ is defined by $J_w:=(\Omega_{\eta})_*\circ J_x\circ (\Omega_{\eta}^{-1})_*$.  Since $J_x=(\exp v)_*\circ J_o\circ (\exp v)_*^{-1}$,  it follows from  \eqref{Omex1} and \eqref{Omex2} that
\begin{align*}
J_wX_{\alpha}^{\fp}=\frac{\overline{\alpha}(w)}{\alpha(w)}\cdot\frac{\sinh\alpha(v)}{\sinh\overline{\alpha}(v)}J_oX_{\alpha}^{\fp}
\end{align*} 
for any $X_{\alpha}^{\fp}\in \fp_{\alpha}$ and $\alpha\in \Sigma_+\setminus \Gamma$.  

Similar computations show that 
\begin{align*}
J_wX_{\gamma_i}^{\fp}=\eta'\Big(\dfrac{\gamma_i(v)}{2}\Big)\cdot \dfrac{\sinh (\gamma_i(v))}{\gamma_i(w)}J_oX_{\gamma_i}^{\fp},\quad J_w\wH_i=\eta'\Big(\dfrac{\gamma_i(v)}{2}\Big)^{-1}\cdot  \dfrac{\gamma_i(w)}{\sinh (\gamma_i(v))}J_o\wH_i.
\end{align*}
Note that \eqref{Ompush1} shows that $\eta'(\gamma_i(v)/2)\neq 0$ since $\Omega_{\eta}$ is an embedding.
Moreover, we have
\[
\gamma_i(w)=\gamma_i\Big(\sum_{i=1}^r \eta(x_i)\wH_i\Big)=2\eta(x_i)=2\eta\Big(\frac{\gamma_i(v)}{2}\Big).
\]
Thus, if we put
\begin{align}\label{Def:G}
G_{\eta,\alpha}(w)=
\begin{cases}
\eta'\Big(\dfrac{\gamma_i(v)}{2}\Big)\eta\Big(\dfrac{\gamma_i(v)}{2}\Big)^{-1} \dfrac{\sinh\gamma_i(v)}{2}& {\rm if}\ \alpha=\gamma_i\in \Gamma \bigskip\\
\dfrac{\overline{\alpha}(w)}{\alpha(w)}\cdot \dfrac{\sinh\alpha(v)}{\sinh\overline{\alpha}(v)}& {\rm if}\ \alpha\in \Lambda\sqcup E
\end{cases},
\end{align}
then we obtain the first formula of \eqref{Expr:JO}.  Note that if $\alpha=\epsilon$,  we have $G_{\eta, \epsilon}=1$ and hence, $J_w|_{\fp_{\epsilon}}=J_o|_{\fp_{\epsilon}}$.

Next, we consider the induced symplectic structure $\omega_w:=(\Omega_{\eta}^{-1})^{*}\omega_x=(\Omega_{\eta}^{-1})^{*}\circ (\exp v^{-1})^{*}\omega_o$. Note that $\omega_o$ is decomposed into $\omega_o=\bigoplus_{\alpha\in \Gamma\sqcup \Lambda\sqcup E}\omega_o|_{\fq_{\alpha}}$. Thus,  Lemmas \ref{Lemma:expush} and \ref{Lem:Ompush} imply  $\omega_w$ is also decomposed into $\omega_w=\bigoplus_{\alpha\in \Gamma\sqcup \Lambda\sqcup E}\omega_w|_{\fq_{\alpha}}$. 
Moreover, for any $\alpha\in \Sigma_+\setminus \Gamma$, we have 
\begin{align*}
 (\exp v^{-1})_*\circ (\Omega_{\eta}^{-1})_*X_{\alpha}^{\fp}=\frac{\sinh\alpha(v)}{\alpha(w)}X_{\alpha}^{\fp}.
\end{align*}
Therefore, we see
\[
 \omega_w(X_{\alpha}^{\fp}, Y_{\overline{\alpha}}^{\fp})=\frac{\sinh\alpha(v)\cdot\sinh\overline{\alpha}(v)}{\alpha(w)\cdot\overline{\alpha}(w)}\omega_o(X_{\alpha}^{\fp}, Y_{\overline{\alpha}}^{\fp})
\]
for any $X_{\alpha}^{\fp}\in \fp_{\alpha}$ and $Y_{\overline{\alpha}}^{\fp}\in \fp_{\overline{\alpha}}$ with $\alpha\in \Sigma_+\setminus\Gamma$.  Similar computations show that 
\[
 \omega_w(X_{\gamma_i}^{\fp}, \wH_i)=\eta'\Big(\dfrac{\gamma_i(v)}{2}\Big)^{-1}\cdot \frac{\sinh\gamma_i(v)}{\gamma_i(w)}\omega_o(X_{\gamma_i}^{\fp}, \wH_i).
\]
Therefore, by putting
\begin{align}\label{Def:F}
F_{\eta,\alpha}(w)=
\begin{cases}
\eta'\Big(\dfrac{\gamma_i(v)}{2}\Big)^{-1}\eta\Big(\dfrac{\gamma_i(v)}{2}\Big)^{-1} \dfrac{\sinh\gamma_i(v)}{2}& {\rm if}\ \alpha=\gamma_i\in \Gamma \bigskip\\ 
\dfrac{\sinh\alpha(v)\cdot\sinh\overline{\alpha}(v)}{\alpha(w)\cdot\overline{\alpha}(w)}& {\rm if}\ \alpha\in \Lambda\sqcup E
\end{cases},
\end{align}
we obtain the second formula of  \eqref{Expr:JO}.
\end{proof}

\subsection{Proof of Theorem \ref{mainthm1}} \label{Proof:mainthm1}
In this subsection, we  give a proof of Theorem \ref{mainthm1} (see also Theorem \ref{mainthm1b} below).  We begin with the following lemma.

\begin{lemma}\label{Lemma:uni}
Let $\Omega: M\to \fp$ be a $K$-equivariant $C^1$ embedding of $M$ such that $\Omega(A)\subseteq \fa$. If furthermore, $\Omega$ preserves the holomorphic structure or the symplectic structure, then $\Omega$ is a strongly diagonal realization of $M$.
\end{lemma}

\begin{proof}
By Proposition \ref{Prop:omh}, we may assume $\Omega$ is given by
\begin{align*}
&\Omega\Big(k\cdot \Exp_o\Big(\sum_{i=1}^r x_i\wH_i\Big)\Big)={\rm Ad}(k)\sum_{i=1}^r h\circ p_{1i}(x_1,\ldots, x_r)\wH_i
\end{align*}
for some differentiable function $h:\R^r\to \R$ satisfying (a) and (b) in Lemma \ref{Lemma:rest}. Then, for any  principal element $x=\exp v\cdot o \in A$ such that $v=\sum_{i=1}^r x_i\wH_i\in \fa$, a similar computation given in Lemma \ref{Lem:Ompush} shows that
\begin{align}\label{ompushg1}
\Omega_*(\wH_1)_x^\da&=\frac{d}{dt}\Big|_{t=0} \Omega(\exp(t\wH_1)\exp v\cdot o)=\frac{d}{dt}\Big|_{t=0} \Omega(\Exp_o(t\wH_1+v))\\
&=\frac{d}{dt}\Big|_{t=0} \sum_{i=1}^r h\circ p_{1i}(t+x_1,x_2,\ldots, x_r)\wH_i
=\sum_{i=1}^r \frac{\partial h}{\partial x_i}\circ p_{1i}(x_1,\ldots, x_r)\wH_i.\nonumber
\end{align}
Moreover, for the basis $Z_j$ of $\fg_j$ given in Lemma \ref{Lemma:cplx}, we have
\begin{align}\label{ompushg2}
\Omega_*(Z_j^{\fk})_x^\da&=\frac{d}{dt}\Big|_{t=0} \Omega(\exp(tZ_j^{\fk})\exp v\cdot o)=\dfrac{d}{dt}\Big|_{t=0} {\rm Ad}(\exp (tZ_j^{\fk}))\Omega(\Exp_ov)\\
&=\sum_{i=1}^r h\circ p_{1i}(x_1,\ldots, x_r)[Z_j^{\fk}, \wH_i]=-2h\circ p_{1j}(x_1,\ldots, x_r)Z_j^{\fp}=-h\circ p_{1j}(x_1,\ldots, x_r)J_o\wH_j\nonumber
\end{align}
for any $j=1,\ldots, r$. In particular, $\Omega_*(Z_j^{\fk})_x^\da\in \fp_j=\R J_o\wH_j$.

  If $\Omega: (M,J)\to (\fp, J_o)$ is holomorphic,  we have $\Omega_*J_x(\wH_1)_x^\da=J_o\Omega_*(\wH_1)_x^\da$.  Here,  we have that $\Omega_*J_x(\wH_1)_x^\da\in \fp_1=\R J_o\wH_1$. Indeed,
\begin{align*}
\Omega_*J_x(\wH_1)_x^\da&=\Omega_*(\exp v)_*J_o(\exp v)_*^{-1}(\wH_1)_x^\da=\Omega_*(\exp v)_*J_o\wH_1=\Omega_*(\exp v)_*(2Z_1^{\fp}),
\end{align*}
and by Lemma \ref{Lemma:expush}, we have $(\exp v)_*Z_1^\fp\in \fk_1=\R Z_1^\fk$ and hence, $ \Omega_*(\exp v)_*(2Z_1^{\fp})\in \fp_1=\R J_o\wH_1$ by \eqref{ompushg2}.  Therefore, the equation $\Omega_*J_x(\wH_1)_x^\da=J_o\Omega_*(\wH_1)_x^\da$ and \eqref{ompushg1} show that
\begin{align*}
\frac{\partial h}{\partial x_j}\circ p_{1j}(x_1,\ldots, x_r)=0
\end{align*}
for any $j>1$ and any $v=\sum_{i=1}^r x_i\wH_i\in \fa^{pri}$. This implies  $\partial h/\partial x_j=0$ on $\fa$ for any $j>1$ since $p_{1j}$  maps $\fa^{pri}$ to $\fa^{pri}$ bijectively and $\fa^{pri}$ is an open dense subset  in $\fa\simeq \R^r$.  Thus, $h$ is independent of the variable $x_j$ for any $j>1$, and this implies that $\Omega$ is strongly diagonal.

Next,  we suppose $\Omega: (M,\omega)\to (\fp, \omega_o)$ is symplectic. Since  
$
\omega_x((\wH_1)_x^\da, (Z_j^{\fk})_x^\da)=0
$
for any $j\neq 1$, the equations \eqref{ompushg1} and \eqref{ompushg2} show that
\[
0=\Omega^*\omega_o((\wH_1)_x^\da, (Z_j^{\fk})_x^\da)=-\frac{\partial h}{\partial x_j}\circ p_{1j}(x_1,\ldots, x_r)\cdot h\circ p_{1j}(x_1,\ldots, x_r)\cdot C.
\]
Since $\Omega$ is an embedding, $h$ is not identically $0$ on any open subset in $\fa$, and hence, if $\Omega^*\omega_o=\omega$, then $\partial h/\partial  x_j=0$ holds on $\fa^{pri}$ for any $j>1$.  This implies $\Omega$ is strongly diagonal.
\end{proof}

Recall that any strongly diagonal realization $\Omega=\Omega_{\eta}: M\to\fp$ is obtained by an extension of the $\widehat{K}$-equivariant embedding $\Omega_{\eta, 0}: \C H^1(-C)\to \widehat{\fp}$ of $\C H^1(-C)\simeq \widehat{G}/\widehat{K}$ associated with an injective odd function $\eta$ defined on $\R$. If $\Omega_{\eta}$ is holomorphic (resp. symplectic), it is necessary to assume that so is $\Omega_{\eta,0}$ since $\Omega|_{A_i^\C}$ is identified with $\Omega_{\eta, 0}$.   Here, the holomorphic/symplectic $\Omega_{\eta, 0}$ is uniquely determined as follows, up to appropriate constant multiple:

\begin{lemma}\label{Lemma:unich1}
Suppose the map $\Omega_{\eta, 0}$ is a $C^1$-embedding.
\begin{enumerate}
\item $\Omega_{\eta,0}$ is holomorphic  if and only if $\eta(x)=\tanh x$ up to non-zero constant multiple.
\item $\Omega_{\eta,0}$ is symplectic  if and only if $\eta(x)=\sinh x$ up to sign.
\end{enumerate}
\end{lemma}

\begin{proof}
We apply Proposition \ref{Prop:JO} by assuming  $\Omega_{\eta}=\Omega_{\eta, 0}$. If $\Omega_{\eta, 0}$ is holomorphic (resp. symplectic),  then we have $J_w=J_o$ (resp. $\omega_w=\omega_o$) for $w\in D_{\eta}\cap \fa^{pri}$ in Proposition \ref{Prop:JO},  and hence, \eqref{Def:G} (resp. \eqref{Def:F}) implies that $\eta$ satisfies the following ordinary differential equation on the interval $(0, \infty)$:
\[
\eta'(x)\eta(x)^{-1}\frac{\sinh 2x}{2}=1\quad \Big(\textup{resp.}\ \eta'(x)^{-1}\eta(x)^{-1}\frac{\sinh 2x}{2}=1\Big).
\]
The solution with $\eta(0)=0$ is given by $\eta(x)=c\tanh x$ (resp. $\eta(x)=\pm\sinh x$). Conversely, if $\eta(x)=c\tanh x$ (resp. $\eta(x)=\pm\sinh x$), it is easy to see that $J_w=J_o$ (resp. $\omega_w=\omega_o$) for any $w\in D_{\eta}\cap \fa^{pri}$. Since $(J,\omega)$ and  $(J_o,\omega_o)$ are $K$-invariant, we see $\Omega_{\eta, 0}$ is holomorphic (resp. symplectic) on an open dense subset in $\C H^1(-C)$, and the property is extended to whole $M=\C H^1(-C)$.
\end{proof}

Conversely, if $\Omega_{\eta, 0}$ is holomorphic/symplectic, then so is $\Omega_{\eta}$:

\begin{theorem}\label{mainthm1b}
Let $M=G/K$ be an irreducible Hermitian symmetric space of noncompact type. 
\begin{enumerate}
\item The map $\Psi=\Omega_{\tanh}: (M, J)\to (\fp, J_o)$ is a $K$-equivariant holomorphic embedding.  Moreover, $\Psi$ is the unique (up to constant multiple) $K$-equivariant holomorphic embedding such that $A$ is mapped into $\fa$.

\item The  map $\Phi=\Omega_{\sinh}: (M, \omega)\to (\fp,\omega_o)$ is a $K$-equivariant symplectomorphism. Moreover, $\Phi$ is the unique (up to sign) $K$-equivariant symplectic embedding such that $A$ is mapped into $\fa$.
\end{enumerate}
\end{theorem}

\begin{proof}
By Proposition  \ref{Prop:JO}, it is sufficient to show that $G_{\tanh, \alpha}(w)=1$ (resp. $F_{\sinh, \alpha}(w)=1$) for any $\alpha\in \Lambda\sqcup \Gamma\sqcup E$ and any principal element $w\in D_{\eta}$, which leads $J_w=J_o$ (resp. $\omega_w=\omega_o$), since the set of principal elements is open dense in $\fa$ and $J$, $J_o$, $\omega$ and $\omega_o$ are $K$-invariant.

By Lemma \ref{Lemma:unich1}, we have $G_{\tanh, \gamma}(w)=1$ and $F_{\sinh, \gamma}(w)=1$  for any $\gamma\in \Gamma$. Also,  it is easy to see from \eqref{Def:G} and \eqref{Def:F} that  $G_{\tanh, \epsilon}(w)=1$ and $F_{\sinh, \epsilon}(w)=1$  for any $\epsilon \in E$.

On the other hand, for an element $\lambda=e_i+e_j\in \Lambda$, by putting $v=\sum_{i=1}^r x_i\wH_i$ so that $w=\sum_{i=1}^r \eta(x_i)\wH_i$, it follows from \eqref{Def:G} and \eqref{Def:F} that
\begin{align*}
G_{\tanh, \lambda}(w)=\frac{\tanh x_i-\tanh x_j}{\tanh x_i+\tanh x_j}\cdot \frac{\sinh (x_i+x_j)}{\sinh (x_i-x_j)},\quad 
F_{\sinh, \lambda}(w)=\frac{\sinh (x_i+x_j)\cdot \sinh (x_i-x_j)}{(\sinh x_i+\sinh x_j)(\sinh x_i-\sinh x_j)}.
\end{align*}
The right hand sides of both equations are equal to $1$ since the following trigonometric identities hold for any complex numbers $z, w$ except the singularities of $\tanh$:
\begin{align*}
(\tanh z-\tanh w)\sinh(z+w)&=(\tanh z+\tanh w)\sinh (z-w),\\
\sinh(z+w)\sinh(z-w)&=(\sinh z+\sinh w)(\sinh z-\sinh w).
\end{align*}
This proves that $\Psi$ is holomorphic, and $\Phi$ is symplectic.

The uniqueness of each map  follows from Lemmas \ref{Lemma:uni} and \ref{Lemma:unich1}.
\end{proof}

\begin{remark}
{\rm It  also follows from \eqref{Def:G}, \eqref{Def:F} that we have
\begin{align*}
F_{\tanh, \alpha}(w)&=\cosh^2\Big(\frac{\alpha(v)+\overline{\alpha}(v)}{2}\Big)\cosh^2\Big(\frac{\alpha(v)-\overline{\alpha}(v)}{2}\Big),\\
G_{\sinh, \alpha}(w)&=\cosh^2\Big(\frac{\alpha(v)}{2}\Big)\cosh^{-2}\Big(\frac{\overline{\alpha}(v)}{2}\Big)
\end{align*}
for $\alpha\in \Gamma\sqcup\Lambda\sqcup E$.  Here, we define $\overline{\alpha}:=\alpha$ if $\alpha\in E$ and $\overline{\gamma}:=0$ if $\gamma\in \Gamma$.  

}
\end{remark}

\subsection{Proof of Theorem \ref{mainthm2}} \label{Proof:mainthm2}
 Let $M^*=G^*/K$ be the compact dual of an irreducible HSSNT $M=G/K$.  We show the dual map of the holomorphic (resp. symplectic) embedding $\Psi$ (resp. $\Phi$) is also holomorphic (resp. symplectic).
  We use the same notations given in Section \ref{Prelim} and Subsection \ref{subsec:dual}.

Recall that $\fg^*=\fk\oplus \fp^*$, where $\fp^*=\sqrt{-1}\fp$.  We identify $\fp$ and $\fp^*$ by $\sqrt{-1}: \fp\to \fp^*$, and 
we denote the element in $\fp^*$ by $X^*=\sqrt{-1}X$ for $X \in \fp$.  Let $\fa^*:=\sqrt{-1} \fa$ be a maximal abelian subspace in $\fp^*$, where  $\fa$ is a maximal abelian subspace in $\fp$. Then, the dual spaces of $\fa$ and $\fa^*$ are identified by the map $\alpha \mapsto \alpha^*:=-\sqrt{-1}\alpha$. Note that, then we have
$\alpha(H)=\alpha^*(H^*)$ for any $\alpha$ and $H \in \fa$.

For an element $\alpha^*$ in the dual space of $\fa^*$, we put 
$\fg_{\alpha^*}^*:=\{Y\in \fg^* : ({\rm ad}(H^*))^2 Y = -(\alpha^*(H^*))^2 Y\ \forall H^*
\in \fa^* \}$. 
Note that $\fg_{\alpha^*}^*=\fg_{-\alpha^*}^*$.
We call $\alpha^*$ a {\it (restricted) root} if $\fg_{\alpha^*}^* \neq \{0\}$, and denote the set of nonzero roots by $\Sigma^*$.   
We define an ordering on $\fa^*$ by the ordering on $\fa$ via the isomorphism $\sqrt{-1}:\fa\simeq \fa^*$, and we denote the set of positive roots by $\Sigma^*_+$.

It turns out that  $\fg_{\alpha^*}^*$ is linear isomorphic to $\fg_{\alpha}$ for any $\alpha$ by the restriction of the linear isomorphism $\fg\to \fg^*$, $X^\fk+X^\fp\to X^\fk+\sqrt{-1}X^\fp$.
Hence the multiplicity $m_{\alpha^*}^*:=\rm{dim}_{\mathbb{R}} \fg_{\alpha^*}^*$ of $\alpha^*\in \Sigma_{+}^*$ is equal to the multiplicity $m_{\alpha}$ of $\alpha\in \Sigma$ and the type of the restricted root system $\Sigma^*$ is the same as the type of $\Sigma$. 
We put
$\fk_{\alpha^*}:=\fk \cap \fg_{\alpha^*}^*(=\fk_{\alpha})$ and  $\fp_{\alpha^*}^*:=\fp^* \cap \fg_{\alpha^*}^*$ for each $\alpha^*\in \Sigma^*_+$.
Then, we obtain  decompositions
\begin{align*}
\fk=\fk_0 \oplus \bigoplus_{\alpha^* \in \Sigma^*_+} \fk_{\alpha^*},\quad  \fp^*=\fa^* \oplus \bigoplus_{\alpha^* \in \Sigma^*_+} \fp_{\alpha^*}^*,
\end{align*} 
and 
$[\fk_{\alpha^*}, \fk_{\beta^*}] \subset \fk_{\alpha^*+\beta^*} \oplus \fk_{\alpha^*-\beta^*}, \
[\fk_{\alpha^*}, \fp_{\beta^*}^*] \subset \fp_{\alpha^*+\beta^*}^* \oplus \fp_{\alpha^*-\beta^*}^*, \
[\fp_{\alpha^*}^*, \fp_{\beta^*}^*] \subset \fk_{\alpha^*+\beta^*} \oplus \fk_{\alpha^*-\beta^*}$
for any $\alpha^*, \beta^* \in \Sigma^*_+$. 

Any element $X_{\alpha^*}\in \fg_{\alpha^*}^*$ is decomposed into 
$ X_{\alpha^*}= X_{\alpha^*}^{\fk}+ X_{\alpha^*}^{\fp^*}$ for $ X_{\alpha^*}^{\fk}\in \fk_{\alpha^*}$ and $ X_{\alpha^*}^{\fp^*}\in \fp_{\alpha^*}^*$.
Since $X_{\alpha^*}^{\fk}=X_{\alpha}^{\fk}$ and $X_{\alpha^*}^{\fp^*}=\sqrt{-1}X_{\alpha}^{\fp}$, where $X_{\alpha}=X_{\alpha}^\fk+X_{\alpha}^\fp\in \fg_{\alpha}$ is the corresponding element of $X_{\alpha^*}$ by the identification $\fg_{\alpha}\simeq \fg_{\alpha^*}^*$,   \eqref{feq2} shows that we have
\begin{align}\label{feq2dual}
[H^*, X_{\alpha^*}^{\fk}]= \alpha^*(H^*) X_{\alpha^*}^{\fp^*}, \quad
[H^*, X_{\alpha^*}^{\fp^*}]= -\alpha^*(H^*) X_{\alpha^*}^{\fk}
\end{align}
for any $H^* \in \fa^*$. We also note that $[\fa^*, \fk_0]=0$.

Recall that $\Sigma_+=\Lambda \sqcup \overline{\Lambda} \sqcup E \sqcup \Gamma$. We put  $\Lambda^*=-\sqrt{-1}\Lambda$, $\overline{\Lambda^*}=-\sqrt{-1}(\overline{\Lambda})$, $E^*=-\sqrt{-1}E$, and $\Gamma^*=-\sqrt{-1}\Gamma$.  Since the element $\zeta \in \fc(\fk)$ given in Lemma  \ref{Lemma:cplx} defines the complex structure $J_o^*:=\rm{ad}(\zeta)\mid_{\fp^*}$ on $\fp^*$, we see the following analogous to Lemma \ref{Lemma:J}:
\begin{enumerate}
\item $J_o^*\fp_{\gamma_i^*}^*=\fa_{\gamma_i^*}^*$ for any $\gamma_i^* \in \Gamma^*$.
\item  $J_o^*\fp_{\lambda^*}^*=\fp_{\overline{\lambda^*}}^*$ for any $\lambda^* \in \Lambda^*$.
\item $J_o^*\fp_{\epsilon^*}^*=\fp_{\epsilon^*}^*$ for any $\epsilon^* \in E^*$.
\end{enumerate}
By putting 
\[
\fq_{\gamma_i^*}^*:=\fp_{\gamma_i^*}^*\oplus \fa_{\gamma_i^*}^*=\fp_{\gamma_i^*}^*\oplus J_o^*\fp_{\gamma_i^*}^*,\quad
\fq_{\lambda^*}^*:=\fp_{\lambda^*}^*\oplus \fp_{\overline{\lambda^*}}^*=\fp_{\lambda^*}^*\oplus J_o^*\fp_{\lambda^*}^*,\quad
\fq_{\epsilon^*}^*:=\fp_{\epsilon^*}^*
\]
for $\gamma_i^* \in \Gamma^*$, $\lambda^* \in \Lambda^*$ and $\epsilon^* \in E^*$, we obtain the following orthogonal decomposition of $\fp^*$ into $J_o^*$-complex subspaces:
\begin{align}\label{Def:decpC}
\fp^*=\bigoplus_{\alpha^* \in \Gamma^* \sqcup \Lambda^* \sqcup E^*} \fq_{\alpha^*}^*.
\end{align}
Moreover, by the similar arguments in Subsection \ref{Subsec:cs}, we obtain the following proposition:

\begin{proposition}\label{Prop:JOdual}
Let  $\Omega^*_{\eta^*}: U_{\eta^*}^*\to \fp^*$ be a strongly diagonal embedding associated with an injective odd function $\eta^*: (-R^*, R^*)\to\R$. Suppose that $\Omega^*_{\eta^*}$ is $C^1$. Then, the induced K\"ahler structure $(J^*, \omega^*)$ at $w^*\in D_{\eta^*}^*\cap (\fa^*)^{pri}$ via the strongly diagonal map $\Omega^*_{\eta^*}$ is expressed as follows: 
With respect to the decomposition \eqref{Def:decpC}, $J^*_{w^*}$ and $\omega^*_{w^*}$ are decomposed into
\begin{align*}
J^*_{w^*}=\bigoplus_{\alpha^* \in \Gamma^* \sqcup \Lambda^* \sqcup E^*} J^*_{\eta^*, \alpha^*}(w^*),\quad \omega^*_{w^*}=\bigoplus_{\alpha^* \in \Gamma^* \sqcup \Lambda^* \sqcup E^*} \omega_{\eta^*,\alpha^*}^*(w^*),
\end{align*}
where each $J_{\eta^*, \alpha^*}^*(w^*)$ and $\omega_{\eta^*,\alpha^*}^*(w^*)$ are given by
\begin{align*}
J_{\eta^*, \alpha^*}^*(w^*)&=
\begin{cases}
G^*_{\eta^*,\alpha^*}(w^*)J_o^*|_{\fp_{\alpha^*}^*} & \textup{on}\ \fp_{\alpha^*}^*\\
-G^*_{\eta^*,\alpha^*}(w^*)^{-1}J_o^*|_{J_o^*\fp_{\alpha^*}^*}& \textup{on}\ J_o^*\fp_{\alpha^*}^*
\end{cases}
\quad {\rm if}\ \alpha^* \in \Gamma^* \sqcup \Lambda^*,\quad 
J_{\eta^*, \alpha^*}^*(w^*)=J_o^*|_{\fq_{\alpha^*}^*}\quad {\rm if}\ \alpha^* \in E^*,\\
\omega_{\eta^*,\alpha^*}^*(w^*)&=F^*_{\eta^*, \alpha^*}(w^*)\omega_o^*|_{\fq_{\alpha^*}^*},
\end{align*}
for functions $G^*_{\eta^*,\alpha^*}: D^*_{\eta^*}\cap (\fa^*)^{pri}\to \R$ and $F^*_{\eta^*, \alpha^*}:D^*_{\eta^*}\cap (\fa^*)^{pri}\to \R$ which are defined by
\begin{align*}
G^*_{\eta^*,\alpha^*}(w^*)&=
\begin{cases}
(\eta^*)'\Big(\dfrac{\gamma_i^*(v^*)}{2}\Big)\eta^*\Big(\dfrac{\gamma_i^*(v^*)}{2}\Big)^{-1} \dfrac{\sin\gamma_i^*(v^*)}{2}& {\rm if}\ \alpha^*=\gamma_i^*\in \Gamma^* \bigskip\\
\dfrac{\overline{\alpha^*}(w^*)}{\alpha^*(w^*)}\cdot \dfrac{\sin\alpha^*(v^*)}{\sin\overline{\alpha^*}(v^*)}& {\rm if}\ \alpha^* \in \Lambda^* \sqcup E^*
\end{cases},\\
F^*_{\eta^*,\alpha^*}(w^*)&=
\begin{cases}
(\eta^*)'\Big(\dfrac{\gamma_i^*(v^*)}{2}\Big)^{-1}\eta^*\Big(\dfrac{\gamma_i^*(v^*)}{2}\Big)^{-1} \dfrac{\sin\gamma_i^*(v^*)}{2}& {\rm if}\ \alpha^*=\gamma_i^* \in \Gamma^* \bigskip\\ 
\dfrac{\sin \alpha^*(v^*)\cdot\sin \overline{\alpha^*}(v^*)}{\alpha^*(w^*) \cdot \overline{\alpha^*}(w^*)}& {\rm if}\ \alpha^* \in \Lambda^* \sqcup E^*
\end{cases}.
\end{align*}
Here $v^* \in \fa^*$ such that $\exp v^* \cdot o^*=(\Omega_{\eta^*}^*)^{-1}(w^*) \in A^*$.
\end{proposition}

Note that the sine function appeared in the expressions of $G^*_{\eta^*,\alpha^*}$ and $F^*_{\eta^*, \alpha^*}$, although the hyperbolic sine function was used in \eqref{Def:G} and \eqref{Def:F}. The difference is caused by the signature in the bracket relations  \eqref{feq2dual}.

Analogous to Theorem \ref{mainthm1b}, we see the following:

\begin{theorem}\label{mainthm2b}
Let $M$ be an irreducible HSSNT and $M^*$ be its compact dual.
Then, the dual map $\Psi^*=\Omega_{\tan}^*: (M^*)^o\to \fp^*$ of $\Psi=\Omega_{\tanh}$ is a $K$-equivariant holomorphic diffeomorphism, and the dual map $\Phi^*=\Omega_{\sin}^*: (M^*)^o\to D^*_{\sin}$ of $\Phi=\Omega_{\sinh}$ is a K-equivariant symplectomorphism.
\end{theorem}

\begin{proof}
It suffices to show that $G_{\tan,\alpha^*}^*(w^*)=1$ and $F^*_{\sin, \alpha^*}(w^*)=1$ for any $\alpha^*\in \Lambda^*\sqcup \Gamma^*\sqcup E^*$ and $w^*\in D_{\eta^*}^*\cap (\fa^*)^{pri}$. This follows from a parallel argument of the proof of Theorem \ref{mainthm1b} combining with Proposition \ref{Prop:JOdual}. Note the dual functions of $\tanh$ and $\sinh$ are defined by $\tan x=-\sqrt{-1}\tanh (\sqrt{-1}x)$ and $\sin x=-\sqrt{-1}\sinh(\sqrt{-1}x)$, respectively.
\end{proof}

\begin{remark}\label{Rem:holsymp}
{\rm
As an immediate consequence of Theorems \ref{mainthm1b} and \ref{mainthm2b}, we see that $M$ is holomorphically embedded into $M^*$ and $(M^*)^o$ is symplectically embedded into $M$ by the following composition maps:
\[
h:=(\Psi^*)^{-1}\circ \Psi: M\to M^*, \quad s:=\Phi^{-1}\circ \Phi^*: (M^*)^o\to M,
\]
where we used the identification $\sqrt{-1}:\fp\simeq \fp^*$ (see the following diagram). 
\[
 \begin{CD}
     D_{\tanh} @>{\sqrt{-1}}>> \fp^*\\
  @A{\Psi}A{\rm hol.}A    @V{(\Psi^*)^{-1}}V{\rm hol.}V \\
     M  @>{h}>> (M^*)^o
  \end{CD}
 \quad\quad\quad\quad
  \begin{CD}
     D_{\sin}^* @>-{\sqrt{-1}}>>  \fp\\
  @A{\Phi^*}A{\rm symp.}A    @V{(\Phi)^{-1}}V{\rm symp.}V \\
     (M^*)^o  @>{s}>> M
  \end{CD}
 \]
We remark that the holomorphic embedding $h$ coincides with the well-known {\it Borel embedding} up to ``rotation" on $(M^*)^o$.  See Proposition \ref{Prop:HC} and Remark \ref{Rem:HC} in Appendix for details.
}
\end{remark}

\begin{remark}\label{Rem:dual}
{\rm
 It also turns out that we have
\begin{align}\label{CDpp}
(\Phi^*)^{-1}\circ \Psi=(\Psi^*)^{-1}\circ \Phi
\end{align}
under  the  identification $\sqrt{-1}:\fp\simeq \fp^*$, i.e. the following commutative diagram holds:
\[
\xymatrix@C=10pt@R=10pt{
&\ \ \fp\ \simeq\ \fp^*\ \ \ar[rd]^-{(\Psi^*)^{-1}}&\\ 
M\ar[ru]^-{\Phi} \ar[rd]^-{\Psi}  & & (M^*)^o\\
&\quad D_{\tanh}\ \simeq\   D^*_{\sin}\quad \ar[ru]^-{(\Phi^*)^{-1}} &
}
\]
Indeed,  we see that
\begin{align*}
(\Phi^*)^{-1}\circ \Psi(k\Exp_ov)=(\Psi^*)^{-1}\circ \Phi(k\Exp_ov)=k\Exp_o^*\Big(\sum_{i=1}^r {\rm gd}(x_i)\wH_i^*\Big)
\end{align*}
for any $k\in K$ and $v=\sum_{i=1}^r x_i\wH_i\in \fa$, where ${\rm gd}(x)={\rm Arcsin}(\tanh x)={\rm Arctan}(\sinh x)$, and the function ${\rm gd}(x)$ is so called the Gudermannian function.

The identity \eqref{CDpp} is related to the main result in \cite{DL}.  Let us identify $M$ with the bounded domain $D_{\tanh}$ by the holomorphic diffeomorphism $\Psi$, and denote the induced metric on $D_{\tanh}$ by $\omega_{\rm hyp}:=(\Psi)^*\omega$, namely, we identify $(M, \omega, J)$ with $(D_{\tanh}, \omega_{\rm hyp}, J_o)$ by $\Psi$. Similarly, we shall identify $((M^*)^o, \omega^*, J^*)$ with $(\fp^*, \omega_{\rm FS}, J_o^*)$ by the dual map $\Psi^*$, where $\omega_{\rm FS}:=(\Psi^*)^{*}\omega^*$.  We put $\widetilde{\Phi}:=\Phi\circ \Psi^{-1}$ and $\widetilde{\Phi}^*:=\Phi^*\circ (\Psi^*)^{-1}$ (see the left hand side of the following diagrams).

\[
\xymatrix@C=10pt@R=10pt{
&\ \fp\ &\ \fp^*\ \ar[dd]^-{\widetilde{\Phi}^*} \ar[rd]^-{(\Psi^*)^{-1}}&\\
M\ar[ru]^-{\Phi} & & & (M^*)^o \ar[ld]^-{\Phi^*}\\
&\quad \ar[lu]^-{\Psi^{-1}}  D_{\tanh} \ar[uu]^-{\widetilde{\Phi}}\ &\   D^*_{\sin} &
}
\quad\quad
\xymatrix@C=10pt@R=10pt{
&\ \fp\ \ar[ld]_-{\Phi^{-1}} \ar[dd]_-{\widetilde{\Psi}}&\ \fp^*\ &\\
M \ar[rd]_-{\Psi}& & &  \ar[lu]_-{\Psi^*} (M^*)^o \\
&\quad   D_{\tanh} \ &\   D^*_{\sin} \ar[ru]_-{(\Phi^*)^{-1}}  \ar[uu]^-{\widetilde{\Psi}^*}&
}
\]

Then, both $\widetilde{\Phi}: (D_{\tanh}, \omega_{\rm hyp})\to (\fp, \omega_o)$ and $\widetilde{\Phi}^*: (\fp^*, \omega_{\rm FS})\to (D_{\sin}, \omega_o)$ are symplectomorphisms. Moreover,  by using the identification $\fp\simeq \fp^*$,  \eqref{CDpp} leads that a non trivial relation $\widetilde{\Phi}^*=\widetilde{\Phi}^{-1}$. In particular, 
\[
\begin{cases}
(\widetilde{\Phi})^{*}\omega_o=\omega_{\rm hyp}\\
(\widetilde{\Phi}^*)^{*}\omega_o=\omega_{\rm FS}
\end{cases}
\quad \textup{is equivalent to}\quad
\begin{cases}
(\widetilde{\Phi})^{*}\omega_o=\omega_{\rm hyp}\\
(\widetilde{\Phi})^{*}\omega_{FS}=\omega_o
\end{cases}
\]
under the identification $\fp\simeq \fp^*$.
The latter property of $\widetilde{\Phi}$ is called the {\it symplectic duality} in \cite[Theorem 1.1. (S)]{DL}, namely, $\widetilde{\Phi}$ is regarded as a simultaneous symplectomorphism between different symplectic structures.  In fact, $\widetilde{\Phi}$ coincides with the symplectomorphism constructed by Di Scala-Loi in \cite{DL} (See Appendix \ref{A2}).

Similarly, we can identify $(M, \omega, J)$ and $((M^*)^o, \omega^*, J^*)$ with $(\fp, \omega_o, J_{\rm hyp})$ and $(D_{\sin}^*, \omega_o, J_{\rm FS})$ by $\Phi$ and $\Phi^*$, respectively, where  $J_{\rm hyp}$ (resp. $J_{\rm FS}$) is the induced complex structure of $M$ (resp. $(M^*)^o$) via the symplectomorphism $\Phi$ (resp. $\Phi^*$).
Then, the map $\widetilde{\Psi}:=\Psi\circ \Phi^{-1}: (\fp, J_{\rm hyp})\to (D_{\tanh}, J_o)$ and $\widetilde{\Psi}^*:=\Psi^*\circ (\Phi^*)^{-1}: (D_{\sin}^*, J_{\rm FS})\to (\fp^*, J_o)$ are holomorphic diffeomorphisms, and \eqref{CDpp} implies that  $\widetilde{\Psi}^*=\widetilde{\Psi}^{-1}$. In particular,
\[
\begin{cases}
J_o\circ (\widetilde{\Psi})_*=(\widetilde{\Psi})_*\circ J_{\rm hyp}\\
J_o\circ (\widetilde{\Psi}^*)_*=(\widetilde{\Psi}^*)_*\circ J_{\rm FS}
\end{cases}
\quad \textup{is equivalent to}\quad
\begin{cases}
J_o\circ (\widetilde{\Psi})_*=(\widetilde{\Psi})_*\circ J_{\rm hyp}\\
J_{\rm FS}\circ (\widetilde{\Psi})_*=(\widetilde{\Psi})_*\circ J_o
\end{cases}
\]
under the identification $\fp\simeq \fp^*$. Namely, the map $\widetilde{\Psi}$ is regarded as a simultaneous holomorphic diffeomorphism between different complex structures.
}
\end{remark}

\section{Realization of totally geodesic submanifold}\label{Section:property} 

In this section, we consider realizations of totally geodesic submanifolds in an HSSNT $M$ by strongly diagonal realizations. The main result in this section is Theorem \ref{mainthm3b} (see also Theorem \ref{mainthm3}). Throughout this section, we assume $\Omega_{\eta}: M\to \fp$ is a strongly diagonal realization of $M$ associated with an injective odd function $\eta:\R\to\R$.

 \subsection{Complexification of abelian subspace as LTS}  

Let $M=G/K$ be a HSSNT of rank $r$ and $\fg=\fk\oplus \fp$ the associated Cartan decomposition. We denote the canonical complex structure on $\fp$ by $J_o$.  In this subsection, we introduce the following notion:

\begin{definition}\label{CLTS}
Let $\fa'$ be an abelian subspace in $\fp$. We say {\rm $\fa'$ has a complexification as LTS in $\fp$} if $\fa'\oplus J_o\fa'$ becomes a complex Lie triple system, i.e. $(\fa')^\C=\fa'\oplus J_o\fa'$ is a complex subspace in $\fp$ satisfying $[[(\fa')^\C, (\fa')^\C],(\fa')^\C]\subset (\fa')^\C$. 
 \end{definition}
 
 Note that any maximal abelian subspace $\fa$ in $\fp$ has a complexification as LTS (Proposition \ref{key1}). 
The aim of this subsection is to give some characterizations of the abelian subspace which admits a complexification as LTS in $\fp$ (Proposition \ref{keyprop1} below).   In order to describe the result simply, we introduce a canonical basis of abelian subspace in $\fp$ as follows:

 Let $\fa'$ be any abelian subspace in $\fp$ of ${\rm dim}_{\R}\fa'=r'$ $(1\leq r'\leq r)$. Then, by Zorn's lemma, there exists a maximal abelian subspace $\fa$ of $\fp$ containing $\fa'$. We take a basis $\{\wH_i\}_{i=1}^r$ of $\fa$ to be the same as given in Section \ref{Prelim}.  
Then,  there exists a subset of indices  $I'=\{i_1,\ldots, i_{r'}:i_1< \cdots < i_{r'}\}\subset I=\{1,\ldots, r\}$ such that set of vectors $\{V_m\}_{m=1}^{r'}$ of the form
\begin{align}\label{canform1}
V_{m}=\wH_{i_m}+\sum_{i\in I\setminus I'} a_{m}^i \wH_i\quad {\rm for}\quad m=1,\ldots, r'
\end{align}
becomes a basis of $\fa'$, where $a_{m}^i\in \R$. Moreover, such a basis is uniquely determined as long as we fix the basis  $\{\wH_i\}_{i=1}^r$ of $\fa$.  Indeed,  one modifies arbitrary basis of $\fa'$ to the basis satisfying \eqref{canform1} by an elementary linear algebra.  We call the basis $\{V_m\}_{m=1}^{r'}$ a {\it canonical basis} of $\fa'$ with respect to $\{\wH_i\}_{i=1}^r$.

\begin{proposition}\label{keyprop1}
Let $\fa'$ be an abelian subspace in $\fp$ of ${\rm dim} \fa'=r'\ (1\leq r'\leq r)$. Then, the following three are equivalent:
\begin{enumerate}
\item  $\fa'$ has a complexification as LTS in $\fp$.
\item $(A')^\C:=\Exp_o(\fa'\oplus J_o\fa')$ is a totally geodesic complex submanifold and splits into an $r'$-times Riemannian product of totally geodesic $\C H^1$ in $M$. More precisely, there exists an orthogonal decomposition $\fa'\oplus J_o\fa'=\bigoplus_{m=1}^{r'}\fw_{m}'$  such that 
$\Exp_o\fw_{m}'$ is a totally geodesic complex submanifold in $M$ which is holomorphically  isometric to $\C H^1(-C_{m})$ for each $m$, and the natural splitting 
\begin{align*}
F_{\fa'}&:(A')^\C\xrightarrow{\sim}\Exp_o\fw_{1}'\times \cdots \times \Exp_o\fw_{r'}'\simeq \C H^1(-C_1)\times \cdots \times \C H^1(-C_{r'})\\
&\quad \Exp_o(v_1'+\cdots +v_{r'}')\mapsto ( \Exp_ov_1', \ldots,  \Exp_ov_{r'}'), \nonumber
\end{align*}
gives a holomorphic isometry.
 \item The canonical basis  $\{V_{m}\}_{m=1}^{r'}$ of $\fa'$ satisfies the following two properties: 
 \begin{itemize}
\item[(a)] For each $i\in I\setminus I'$, ${\bf a}^i=(a_1^i,\ldots, a_{r'}^i)$ contains at most one non-zero element.
\item[(b)] Each coefficient $a_m^i$ is equal to either $0$, $1$ or $-1$. 
\end{itemize}
 \end{enumerate}
\end{proposition}

\begin{proof}
First, we show the equivalence of (i) and (iii).   Suppose $\fa'$ has a complexification as LTS, i.e.  $\fa'\oplus J_o\fa'$ is a LTS. We take the canonical basis $\{V_m\}_{m=1}^{r'}$ of $\fa'$ given by \eqref{canform1}.
Recall that $\{\wH_i\}_{i=1}^r$ satisfies
\[
J_o\wH_i=2Z_i^{\fp},\quad [Z_i^{\fk}, Z_i^{\fp}]=\frac{1}{2}\wH_i,\quad [\wH_i, Z_i^{\fk}]=2Z_i^{\fp}, \quad [\wH_i, Z_i^{\fp}]=2Z_i^{\fk} 
\]
(see \eqref{brawh}). Also, we have $[\fa_i, \fk_j]=[\fa_i, \fp_j]=[\fa_i, \fa_j]=\{0\}$ if $i\neq j$ by Lemma \ref{brau}. Thus, we obtain
\begin{align}
\label{lts1} [[V_m, J_oV_m], J_oV_m]&=4\wH_{i_{m}}+\sum_{i\in I\setminus I'} 4(a_{m}^i)^3\wH_i
\end{align}
for any $m=1,\ldots, r'$. Since $\fa'\oplus J_o\fa'$ is a LTS,  we have $[[V_m, J_oV_m], J_oV_m]\in \fa'$,  and  \eqref{canform1} and \eqref{lts1} imply we must have $[[V_m, J_oV_m], J_oV_m]=4V_m$. Then, we see $(a_{m}^i)^3=a_{m}^i$, namely, $a_{m}^i$ is equal to either $0$, $1$ or $-1$ for any $m=1,\ldots, r'$ and $i=1,\ldots, r$.
 On the other hand, we have 
\begin{align}
\label{lts2} [[V_m, J_oV_m], J_oV_{n}]=\sum_{i\in I\setminus I'}4(a_{m}^i)^2a_{n}^i\wH_i
\end{align}
if $m\neq n$.  Since $[[V_m, J_oV_m], J_oV_{n}]\in \fa'$, \eqref{canform1} and \eqref{lts2} imply  we must have $[[V_m, J_oV_m], J_oV_{n}]=0$, and hence, $(a_{m}^i)^2a_{n}^i=0$ for any $i\in I\setminus I'$ if $m\neq n$. This means ${\bf a}^i=(a_1^i,\ldots, a_{r'}^i)$ contains at most one non-zero element.   This proves (i) $\Rightarrow$ (iii).
 
 Conversely, if (iii) holds, it is straightforward to check (i).  Indeed, we have $[V_m, V_{n}]=0$ for any $m,n=1,\ldots, r'$ since $\fa'$ is abelian. Furthermore,   we see
 \begin{align}\label{abra}
 \begin{cases}
 [[V_m, J_oV_m], J_oV_m]
 =  J_o[[V_m, J_oV_m], V_m]
 =4V_m, & \textup{for any}\ m=1,\ldots, r'\\
 [[V_m, J_oV_{n}], J_oV_{l}]=[[V_m, J_oV_{n}], V_{l}]=0 & \textup{otherwise}
 \end{cases}.
  \end{align}

Next, we show (iii) implies (ii). If (iii) holds, the canonical basis $\{V_m\}_{m=1}^{r'}$ is  an orthogonal basis of $\fa'$.  Moreover,  $\eqref{abra}$ shows $\fu_{m}:=\R[V_m, J_oV_m]\oplus \R V_m\oplus \R J_oV_m$ is a Lie subalgebra of $\fg$ and isomorphic to $\mathfrak{su}(1,1)$ for any $m=1,\ldots, r'$, and putting $\fw_m':=\R V_m\oplus J_oV_m$, we see $\Exp_o\fw_m'$ is a totally geodesic submanifold which is holomorphically isometric to $\C H^1(-C_m)$ by the same arguments given in Lemmas \ref{Lemma:iso} and \ref{Lemma:ch1}. Here the holomorphic sectional curvature $-C_m$ is determined by $\|V_m\|^2=C_m$.  Therefore, by the same argument given in Proposition \ref{key1}, we see $(A')^{\C}:={\rm Exp}_o(\fa'\oplus J_o\fa')$ is a totally geodesic K\"ahler submanifold which is holomorphically isometric to the $r'$-times product of totally geodesic $\C H^1$ in $M$.  

Finally, if (ii) holds, then $\fa'\oplus J_o\fa'$ is a LTS because $(A')^{\C}$ is totally geodesic and $T_o(A')^{\C}=\fa'\oplus J_o\fa'$. This proves (ii) $\Rightarrow$ (i), and we complete the proof.
\end{proof}

\begin{remark} \label{Rem:clts}
{\rm  Proposition \ref{keyprop1}  (iii) shows that $\fa'$ has a complexification as LTS if and only if the canonical basis $\{V_m\}_{m=1}^{r'}$ of $\fa'$ consists of vectors placed in special positions in $\fa$. In particular,  if $\fa'$ has a complexification as LTS,  then each vector $V_m$ is placed in a ``vertex" of cube  
$
\square_{\fa}:={\{}\sum_{i=1}^r x_i\wH_i\in \fa: |x_i|\leq 1{\}}. 
$
Moreover, it turns out that there exist only a {\it finite} number of abelian subspaces in $\fa$ which admits a complexification as LTS in $\fp$.
}
\end{remark}

Also, we remark that, if $\fa'$ has a complexification as LTS,  we have the following characterization of the canonical basis without using root vectors: 

\begin{proposition}\label{Prop:canbasis}
Suppose $\fa'$ has a complexification as LTS.  A basis $\{\widetilde{V}_m\}_{m=1}^{r'}$ of $\fa'$ coincides with the  set of canonical basis of $\fa'$ up to sign of each vector if and only if $\{\widetilde{V}_m\}_{m=1}^{r'}$ satisfies the bracket relations \eqref{abra}.
\end{proposition}
 \begin{proof}
 The ``only if" part follows from the proof of Proposition \ref{keyprop1}. We shall show the converse. Suppose $\{\widetilde{V}_m\}_{m=1}^{r'}$ satisfies \eqref{abra}. We put  $\widetilde{V}_m=\sum_{i=1}^r \epsilon_m^i \wH_i$. Then, the relation $[[\widetilde{V}_m, J_o\widetilde{V}_m], J_o\widetilde{V}_m]=4\widetilde{V}_m$ implies that  $(\epsilon_m^i)^3=\epsilon_m^i$, and hence, $\epsilon_m^i=0$ or $\pm 1$. On the other hand, the relation $[[\widetilde{V}_m, J_o\widetilde{V}_m], J_o\widetilde{V}_n]=0$ ($m\neq n$) implies  
 $(\epsilon_m^i)^2\epsilon_n^i=0$, and hence, $(\epsilon_1^i,\ldots,\epsilon_{r'}^i )$ contains at most one non-zero element for each $i=1,\ldots, r$.  Thus, we see $\{\widetilde{V}_m\}_{m=1}^{r'}$  coincides with the set of canonical basis of $\fa'$, up to sign of each vector.
 \end{proof}
 
 For example, the set of normalized root vectors $\{\wH_i\}_{i=1}^r$ i.e. the canonical basis of the maximal abelian subspace $\fa$ of $\fp$ is characterized (up to sign of each vector) by the basis of $\fa$ satisfying the relations \eqref{abra}.

\subsection{Realization of totally geodesic submanifolds}\label{subsec:reali}
  It is well-known that any connected complete totally geodesic submanifold $N$ in a symmetric space $M=G/K$ corresponds to a Lie triple system (LTS for short) $\fp_N$, i.e. a linear subspace $\fp_N$ in $\fp$ such that $[[\fp_N, \fp_N], \fp_N]\subset \fp_N$, where $\fp$ is the $-1$-eigenspace of the Cartan decomposition $\fg=\fk\oplus\fp$. More precisely, if $\fp_N$ is a LTS, we obtain a Lie subalgebra $\fg_N:=[\fp_N,\fp_N]\oplus \fp_N$  of $\fg$ and the corresponding connected Lie subgroup $G_N$ of $G$. Moreover, the orbit through the origin $N=G_N\cdot o$ is a connected complete totally geodesic submanifold satisfying $T_oN=\fp_N$. Conversely, any connected, complete totally geodesic submanifold through the origin is obtained in this way.  The reader is referred to \cite[Section 11.1]{BCO} for details.
   In the following, we always assume $N$ is a totally geodesic submanifold obtained by an orbit through the origin (possibly, $N=M$), and  we denote the isotropy subgroup of $G_N$ at the origin by $K_N$.  The exponential map of $N$ at the origin coincides with the restricted map $\Exp_o|_{\fp_N}$, and hence, we use the same notation $\Exp_o$.
  
   It turns out that $N\simeq G_N/K_N$ is also a Riemannian symmetric space with respect to the induced metric.  Indeed, the Lie algebra $\fg_N$ is $\theta$-invariant, where $\theta$ is the Cartan involution of the ambient symmetric space $M$, and $[\fp_N,\fp_N]$ and $\fp_N$ coincide with the  $+1$-eigenspace and $-1$-eigenspace of $\theta|_{\fg_{N}}$, respectively.  We take a maximal abelian subspace  $\fa_N$ of $\fp_N=T_oN$. Then, $\fa_N$ is regarded as an abelian subspace in $\fp$ whose dimension coincides with the rank of $N$. We denote the rank of $N$ by $r_N$.

  Now, we suppose $M=G/K$ is a HSSNT of rank $r$.  
  We fix an abelian subspace $\fa$ of $\fp$  containing $\fa_N$, and we denote the canonical basis of $\fa_N$ with respect to $\{\wH_i\}_{i=1}^r$ by $\{V_m\}_{m=1}^{r_N}$.

Suppose furthermore $\fa_N$ has a complexification as LTS in $\fp$.  Then, by Proposition \ref{keyprop1},  we have a holomorphic isometry 
  \[
  F_{\fa_N}: A_N^\C \to  \C H^1(-C_1)\times\cdots  \times\C H^1(-C_{r_N}),
  \]
  where we put $A_N^\C=\Exp_o(\fa_N\oplus J_o\fa_N)$.  Note that $A_N^\C$ is not necessarily contained in $N$, although $A_N^\C$ is a complex submanifold in $M$.

  Recall that any injective odd function $\eta:\R\to \R$ defined on $\R$ yields a $K$-equivariant embedding $\Omega_{\eta}:M\to \fp$ of $M$ (Subsection \ref{subsec:SDR}).  We shall show a $K_N$-equivariant map $\Omega_{\eta, N}: N\to \fp_N$ is defined in the same manner:  We take the radial map $\Omega_{\eta, m}: \C H^1(-C_m)\to \widehat{\fp}$ defined by \eqref{Def:ometa0} for each $m=1,\ldots, r_N$, and extend  these to a map on $A_N^\C$ by the direct product:
  \[
  \Omega_{\eta, A_N^\C}:=\Omega_{\eta, 1}\times \cdots \times \Omega_{\eta, r_N}:  A_N^\C\to \fa_N^\C.
  \]
 We put $A_N:=\Exp_o\fa_N$. Then, the restriction $\Omega_{\eta, A_N}:=  \Omega_{\eta, A_N^\C}|_{A_N}$ is a map from $A_N$ into $\fa_N$ which is expressed by 
 \[
  \Omega_{\eta, A_N}\Big(\Exp_o\Big(\sum_{m=1}^{r_N}x_mV_m\Big)\Big)=\sum_{m=1}^{r_N}\eta(x_m)V_m
 \]
 by using the canonical basis $\{V_m\}_{m=1}^{r_N}$ of $\fa_N$.  Since  $N=K_N\cdot A_N$ and $\fp_N={\rm Ad}(K_N)\fa_N$, we define a $K_N$-equivariant map $\Omega_{\eta, N}: N\to \fp_N$ by
  \[
  \Omega_{\eta, N}(k_N\Exp_ov_N):={\rm Ad}(k_N)\circ \Omega_{\eta, A_N}(\Exp_ov_N)={\rm Ad}(k_N)\sum_{m=1}^{r_N}\eta(x_m)V_m
  \]
  for $k_N\in K_N$ and $v_N=\sum_{m=1}^{r_N}x_mV_m\in \fa_N$.
 Obviously, $\Omega_{\eta, N}=\Omega_{\eta}$ if $N=M$.
  
  \begin{proposition}
 Suppose $\fa_N$ has a complexification as LTS in $\fp$. Then, $\Omega_{\eta, N}$ is well-defined for any injective odd function $\eta: \R\to \R$.  Moreover, $\Omega_{\eta, N}$ depends only on $\eta$ and $N$.
  \end{proposition}

One may check this proposition directly, however, we shall prove much stronger result (Theorem \ref{mainthm3b} below). Thus we omit the proof. 

For any totally geodesic submanifold $N$ in $M$ through the origin and any injective odd function $\eta:\R\to\R$, we put
\[
\square_{\eta, \fa_N}:=\Big{\{}\sum_{m=1}^{r'} x_m V_m\in \fa_N:\ |x_m|<s_{\eta}\Big{\}}\subset \fa_N,\quad D_{\eta, N}:={\rm Ad}(K_N)(\square_{\eta, \fa_N})\subset \fp_N,
\]
where $s_{\eta}:={\rm sup}\{\eta(x): x\in \R\}$. It follows from a usual argument that $D_{\eta, N}$ is independent of the choice of maximal abelian subspace $\fa_N$ of $N$, i.e. it depends only on $\eta$ and $N$. If $\fa_N$ has a complexification as LTS, it is easy to see that
 \[
\Omega_{\eta, N}(A_N)=\square_{\eta, \fa_N},\quad \Omega_{\eta, N}(N)=D_{\eta, N}.
\]

Now, we state the main result of this section. The following theorem is a generalization of \cite[Theorem 1.1 (H)]{DL} (see Example \ref{ex:CLTS} (3) below).  
  
\begin{theorem}\label{mainthm3b}
Let $M=G/K$ be an irreducible HSSNT and $N$  a complete totally geodesic submanifold  in $M$ through the origin.  If $\fa_N$ has a complexification as LTS in $\fp$,  it holds that $\Omega_{\eta}|_N=\Omega_{\eta, N}$ for any injective odd function $\eta:\R\to \R$. 
In particular, the pair $(N, A_N)$ is mapped onto $(D_{\eta, N}, \square_{\eta, \fa_N})$ by the strongly diagonal realization $\Omega_{\eta}$ for any injective odd function $\eta:\R\to \R$.

Conversely, if $(N, A_N)$ is mapped onto $(D_{\eta, N}, \square_{\eta, \fa_N})$ by $\Omega_{\eta}$  for any injective odd function $\eta: \R\to \R$, then $\fa_N$ has a complexification as LTS in $\fp$.
\end{theorem}

\begin{remark}
{\rm It also follows from the proof of Theorem \ref{mainthm3b} that when $\Omega_{\eta}$ is  the holomorphic embedding $\Omega_{\tanh}=\Psi: M\to D_{\tanh}$ (or the symplectomorphism $\Omega_{\sinh}=\Phi: M\to \fp$) given in Theorem \ref{mainthm1},  then the pair $(N, A_N)$ is mapped onto $(D_{\eta, N}, \square_{\eta, \fa_N})$ by $\Psi$ (or $\Phi$) if and only if $\fa_N$ has a complexification as LTS in $\fp$.
}
\end{remark}

We exhibit several examples of totally geodesic submanifold $N$ whose $\fa_N$ has a complexification as LTS. Theorem \ref{mainthm3b} shows that such $N$ is always realized as either a linear subspace (i.e. $\fp_N$) or a $K_N$-invariant bounded domain of it (i.e. $D_{\eta, N}$) in $\fp$ by the map $\Omega_{\eta}$.

\begin{example}\label{ex:CLTS}
{\rm We always assume $N$ is a complete totally geodesic submanifold in $M$ through the origin.

(1)  If $M$ is rank 1 (i.e. $M=\C H^n$), then, $\fa_N$ has a complexification as LTS for any totally geodesic submanifold $N$ of ${\rm dim} N\geq 1$. Indeed,  we may assume $\fa_N=\fa$ since ${\rm dim}_{\R}\fa_N={\rm dim}_{\R}\fa=1$, and hence, it has a complexification as LTS.

More generally, if $N$ is a totally geodesic submanifold with {\it maximal rank} in a Hermitian symmetric space $M$, that is, ${\rm dim}_{\R} \fa_N={\rm dim}_{\R} \fa$, then $\fa_N$ has a complexification as LTS.

(2) If $N$ is a {geodesic} of initial direction $V_N$, then $\fa_N=\R V_N=T_oN$. We take a maximal abelian subspace $\fa$ of $\fp$ containing $\fa_N$ and let $\{\widetilde{H}_i\}_{i=1}^r$ be the set of normalized root vectors with respect to $\fa$. Then, by Proposition \ref{keyprop1}, $\fa_N$ has a complexification as LTS if and only if the initial direction $V_N$ is parallel to the position vector of a vertex of the cube $
\square_{\fa}:={\{}\sum_{i=1}^r x_i\wH_i\in \fa: |x_i|\leq 1{\}}
$, i.e. $V_N$ coincides up to constant multiple with $\sum_{i=1}^r \epsilon_i\wH_i$ where $\epsilon_i=0,1$ or $-1$ for each $i$.

More generally, if $N=\Exp_o\fa_N$ is a {\it flat} totally geodesic submanifold of ${\rm dim} N=r_N$, then $T_oN=\fa_N$ is an abelian subspace in $\fp$.   In this case, $\fa_N$ has a complexification as LTS if and only if the canonical basis of $\fa_N$ satisfies the condition (iii) in Proposition \ref{keyprop1}.  Recall that there exist only a finite number of such abelian subspaces in $\fa$ (see Remark \ref{Rem:clts}).

(3)  If $N$ is a totally geodesic {\it complex} submanifold, then $\fa_N$ has a complexification as LTS in $\fp$. This fact is proved as follows: Fix a maximal abelian subspace $\fa$ in $\fp$ containing $\fa_N$. We put $\fa_N^\C:=\fa_N\oplus J_o\fa_N$. Since $\fp_N$ is a complex subspace in $\fp$, we have $\fa_N^\C\subset \fp_N$. In particular, it follows that $[[\fa_N^\C, \fa_N^\C], \fa_N^\C]\subset \fp_N$ since $\fp_N\simeq T_oN$ is a LTS.
On the other hand, a computation by using the canonical basis of $\fa_N$ shows that (see the proof of Proposition \ref{keyprop1}), we have $[[\fa_N^\C, \fa_N^\C], \fa_N^\C]\subset \fa$. Therefore, we obtain $[[\fa_N^\C, \fa_N^\C], \fa_N^\C]\subset \fa\cap \fp_N$. Since $\fa_N$ is a maximal abelian subspace in $\fp_N$ and $\fa_N\subset \fa\cap \fp_N$, we have $\fa\cap \fp_N=\fa_N$, and hence $[[\fa_N^\C, \fa_N^\C], \fa_N^\C]\subset \fa_N\subset \fa_N^\C$.  This proves that $\fa_N^\C$ is a LTS, as required.

In fact, if $N$ is a totally geodesic complex submanifold, one verifies that $N$ is also a HSSNT with respect to the induced K\"ahler structure. In particular, we have a holomorphic embedding $\Psi_N: N\to D_N$ and a symplectomorphism $\Phi_N: N\to \fp_N$ for $N$ by Theorem \ref{mainthm1b}, and Theorem \ref{mainthm3b} shows that we have $\Psi|_{N}=\Psi_N$ and $\Phi|_{N}=\Phi_N$, where $\Psi: M\to D$ is  the holomorphic embedding of $M$, and $\Phi:M\to \fp$ is the symplectomorphism of $M$, respectively. This recovers the result by Di Scala-Loi \cite[Theorem 1.1 (H)]{DL} for the map $\Phi$.

(4) A submanifold $N$ in $M$ is called {\it real form} if $N$ coincides with a connected component of the fixed point set of an anti-holomorphic involution $\tau: M\to M$. See \cite{Leung} for examples and the classification of real forms in an irreducible Hermitian symmetric space. It is known that any real form becomes a totally geodesic Lagrangian submanifold in $M$.  We shall show, for any real form $N$, $\fa_N$ has a complexification as LTS.  

We follow the argument of \cite[Proposition 3.4]{TT} due to Tanaka-Tasaki.
Let $\tau$ be an involutive anti-holomorphic isometry of $M$ such that $N$ coincides with the fixed point set of $\tau$. We assume $\fa_N\subseteq \fa$ for some maximal abelian subspace $\fa$.
By a similar argument given in \cite[Lemma 3.1]{TT},  it turns out that  $\tau(A)=A$. In particular, $A_N$ coincides with the fixed point set of $\tau|_A$. Recall that $A$ is isometric to $A_1\times \cdots \times A_r$, where $A_i=\Exp_o\fa_i\simeq \R H^1$.
Since $\tau(A)=A$, the image of $A_i$ under $\tau$ is either the same $A_i$ or another $A_j$. If necessary, we change the order as follows: 

If there exist some distinct pairs $(i,j)$ such that $\tau(A_i)=A_j$, we change the order of the basis $\{ \widetilde{H}_i \}_{i=1}^r$ so that $\tau(A_{2i-1})=A_{2i}$ for $1 \leq i \leq p$. In this case, we have $d\tau_o(\wH_{2i-1})=\wH_{2i}$ or $-\wH_{2i}$ since $\|\wH_{2i-1}\|=\|\wH_{2i}\|$. We change the order again so that 
\[
\begin{cases}
d\tau_o(\wH_{2i-1})=\wH_{2i} & \textup{ for $1\leq i\leq p'$,}\\
d\tau_o(\wH_{2i-1})=-\wH_{2i} & \textup{ for $p'+1\leq i\leq p$.}
\end{cases}
\]
For $2p+1\leq i\leq r$, we have $\tau(A_i)=A_i$. In this case,  we see $d\tau_o(\widetilde{H}_i)=\wH_i$ or $-\wH_i$. We change the order of the basis so that 
\[
\begin{cases}
d \tau_o(\widetilde{H}_i)=\widetilde{H}_i& \textup{ for $2p+1 \leq i \leq q$,}\\
d \tau_o(\widetilde{H}_i)=-\widetilde{H}_i & \textup{ for $q+1 \leq i \leq r$.}
\end{cases}
\]
Then,   we see that
\begin{align*}
\fa_N&=\{v\in \fa:\ d\tau_o(v)=v\}\\
&=\left\{\sum_{i=1}^r x_i\wH_i\in \fa:
\begin{array}{l}
x_1=x_2,\ldots,x_{2p'-1}=x_{2p'},\quad  x_{2p'+1}=-x_{2p'+2},\ldots, x_{2p-1}=-x_{2p} \\
 x_{q+1}= \cdots = x_r=0
\end{array} \right\} 
\end{align*}
since $A_N$ coincides with the fixed point set of $\tau|_{A}$.
Thus, $\fa_N$ is spanned by
$\wH_{2i-1}+\wH_{2i}$ $(i=1, \ldots, p')$, $\wH_{2i-1}-\wH_{2i}$ $(i=p'+1, \ldots, p)$ and   
$\wH_{i}$ $(i=2p+1, \ldots, q)$, and they compose the canonical basis of $\fa_N$ satisfying the condition (iii) in Proposition \ref{keyprop1}. Therefore, $\fa_N$ has a complexification as LTS.
}
\end{example}

\subsection{Proof of Theorem \ref{mainthm3b}}\label{proof:main2}

Let $N$ be a complete totally geodesic submanifold in $M$ through the origin. As described above, we identify $N$ with a Riemannian symmetric space $G_N/K_N$. We denote the Cartan decomposition by $\fg_N=\fk_N\oplus \fp_N$ and take a maximal abelian subspace $\fa_N$ of $\fp_N$. We may assume $\fa_N$ is contained in a maximal abelian subspace $\fa$ of $\fp$.
The first assertion of Theorem \ref{mainthm3b} is an immediate consequence of Proposition \ref{keyprop1}:

\begin{proposition}\label{Pr:mainthm3_1}
If $\fa_N$ has a complexification as LTS, then $\Omega_{\eta}|_{N}=\Omega_{\eta, N}$ for any injective odd function $\eta$.
\end{proposition}

\begin{proof}
Suppose $\fa_N$ has a complexification as LTS. Then, by Proposition \ref{keyprop1},  the canonical basis $\{V_m\}_{m=1}^{r_N}$ of $\fa_N$ is expressed by 
$
V_m=\sum_{i=1}^r \epsilon_m^i \wH_i,
$
where $\epsilon_m^i$ is either $0$, $1$ or $-1$, and $(\epsilon_1^i,\ldots, \epsilon_{r_N}^i)$ contains at most one non-zero element for each $i=1,\ldots, r$. Then, any element $v_N\in \fa_N$ is expressed by 
\[
v_N=\sum_{m=1}^{r_N}x_{m} V_m=\sum_{i=1}^r\Big(\sum_{m=1}^{r_N}x_{m}\epsilon_{m}^i\Big) \wH_{i}.
\]
Then, we see
\begin{align*}
\Omega_{\eta}(\Exp_ov_N)&=\sum_{i=1}^r\eta\Big(\sum_{m=1}^{r_N}x_{m}\epsilon_{m}^i\Big) \wH_{i}
=\sum_{i=1}^r\sum_{m=1}^{r_N}\eta(x_{m}\epsilon_{m}^i)\wH_{i}\\
&=\sum_{i=1}^r\sum_{m=1}^{r_N}\eta(x_{m})\epsilon_{m}^i\wH_{i}=\sum_{m=1}^{r_N}\eta(x_m)V_m=\Omega_{\eta,N}(\Exp_ov_N),
\end{align*}
where the second equality is due to the fact that $(\epsilon_1^i,\ldots, \epsilon_{r_N}^i)$ contains at most one non-zero element for each $i$, and the third equality follows from the facts that $\epsilon_{m}^i=0, 1$ or $-1$ and $\eta$ is an odd function. Thus, we obtain $\Omega_{\eta}|_{A_N}=\Omega_{\eta,N}|_{A_N}$, and this implies $\Omega_{\eta}|_{N}=\Omega_{\eta,N}$ since $N=K_N\cdot A_N$ and both $\Omega_{\eta}$ and $\Omega_{\eta, N}$ are $K_N$-equivariant. 
\end{proof}

To prove the converse, we consider the following two properties of totally geodesic submanifold $N$ in $M$:
\begin{enumerate}
\item[${\rm (iii)}$] The canonical basis $\{V_{m}\}_{m=1}^{r'}$ of $\fa_N$ satisfies the same property given in Proposition \ref{keyprop1} {\rm (iii)} (Equivalently, $\fa_N$ has a complexification as LTS).
\item[(iv)]  $\Omega_{\eta}(A_N)=\square_{\eta, \fa_N}\subseteq\fa_N $.
\end{enumerate}
 The proof of Proposition \ref{Pr:mainthm3_1} shows that ${\rm (iii)}$ $\Rightarrow$ (iv). Our claim here is to show  the converse for some $\eta$. Note that, the converse (iv) $\Rightarrow$ (iii) does not hold for general $\eta$. For example, if $\eta$ is a linear odd function, i.e. $\eta(x)=cx$ ($c\neq 0$), then we have $\Omega_{\eta}=c\Omega_{id}=c\Log_o$, and hence, $\Omega_{\eta}(A_N)=\fa_N$ for any totally geodesic submanifold $N$, namely, condition (iv) holds for any $N$. On the other hand, we shall show the following:

\begin{proposition}\label{keyprop2}
Let  $\eta:\R\to\R$ be an injective $C^0$-odd function such that $\eta$ is strictly convex or strictly concave on $[0,\infty)$.  Then, ${\rm (iii)}$ and {\rm (iv)} are equivalent. In particular,  $\Omega_{\eta}$ maps $(N, A_N)$ onto $(D_{\eta, N}, \square_{\eta, N})$ if and only if $\fa_N$ has a complexification as LTS.
\end{proposition}

 For example,  we may assume $\eta(x)=\tanh x$ (or $\sinh x$), i.e. $\Omega_{\eta}=\Psi$ (or $\Phi$) given in Theorem \ref{mainthm1}.

\begin{proof}
The rest is to show (iv) $\Rightarrow {\rm (iii)}$.
Let $\{V_{m}\}_{m=1}^{r'}$ be the canonical basis of $\fa_N$.  
Our claim is that  if ${\rm (iv)}$ holds, then the set of coefficients $\{a_m^i\}_{m=1,\ldots, r', i\in I\setminus I'}$ satisfies
 (a) ${\bf a}^i=(a_1^i,\ldots, a_{r'}^i)$ contains at most one non-zero element for each $i\in I\setminus I'$ and (b) $a_m^i$ is equal to either $0$, $1$ or $-1$.

For any distinct $V_m$ and $V_n$, we have
\begin{align*}
\Omega_{\eta}(\Exp_o (xV_{m}+yV_{n}))=\eta(x) \wH_{i_{m}}+\eta(y)\wH_{i_{n}}+\sum_{i\in I\setminus I'} \eta(xa_{m}^i+ya_{n}^i)\wH_i
\end{align*}
for any $x,y\in \R$. Since $\Omega_{\eta}(\Exp_o(xV_{m}+yV_{n}))\in \fa_N$ by ${\rm (iv)}$,  it must hold that $\Omega_{\eta}(\Exp_o(xV_{m}+yV_{n}))=\eta(x)V_{m}+\eta(y)V_{n}$ due to the definition of canonical basis. Thus, comparing the coefficients of $\wH_i$ in both sides, we obtain the identity 
\begin{align}\label{eq:linear}
\eta(xa_{m}^i+ya_{n}^i)=\eta(x)a_{m}^i+\eta(y)a_{n}^i
\end{align}
for any $x,y\in \R$ and $i\in I\setminus I'$.  Note that,  by putting $y=0$ or $x=0$ in \eqref{eq:linear}, we have 
\begin{align}\label{eq:homog}
\eta(xa_m^i)=\eta(x)a_m^i,\quad \eta(ya_n^i)=\eta(y)a_n^i
\end{align}
for any $x,y\in R$ since $\eta(0)=0$. Thus, \eqref{eq:linear} is equivalent to the identity $\eta(xa_{m}^i+ya_{n}^i)=\eta(xa_{m}^i)+\eta(ya_{n}^i)$.  If both $a_m^i$ and $a_n^i$ were not equal to zero, this identity is equivalent to well-known Cauchy's functional equation $\eta(x+y)=\eta(x)+\eta(y)$, and it turns out that the continuous function $\eta$ is a linear function. This contradicts to the assumption of $\eta$. Thus, we conclude that ${\bf a}^i=(a_1^i,\ldots, a_{r'}^i)$ contains at most one non-zero element for each $i\in I\setminus I'$. 

On the other hand, by \eqref{eq:homog}, $a_m^i$ satisfies $\eta(a_m^i)=\eta(1)a_m^i$.  Since $\eta$ is an injective odd function such that it is strictly convex or concave on $[0,\infty)$,  the equation $\eta(a_m^i)=\eta(1)a_m^i$ is satisfied if only if $a_m^i$ is equal to either $0$, $1$ or $-1$.  This follows from an elementary argument for convex or concave function, and hence, we omit the proof.
\end{proof}

\begin{remark}\label{rem:eta}
{\rm It is observed from this proof that, if $\eta$ is  not a linear odd function,  a totally geodesic submanifold $N$ satisfies (iv) if and only if the canonical basis $\{V_m\}_{m=1}^r$ of $\fa_N$ satisfies (a) ${\bf a}^i=(a_1^i,\ldots, a_{r'}^i)$ contains at most one non-zero element for each $i\in I\setminus I'$ and ${\rm (b)}'$ any coefficient $a_m^i$ satisfies $\eta(xa_m^i)=\eta(x)a_m^i$ for any $x\in \R$.  We remark that ${\rm (b)}'$ does not imply (b) $a_m^i=0$ or $\pm1$ in general even if $\eta$ is not linear. 
}
\end{remark}

\appendix
\section{Relation to known results}

In this appendix, we show relations between our results and some known results.  We confirm that, under appropriate identifications of spaces, $\Psi$ and $\Phi$ constructed in Theorem \ref{mainthm1}   coincide with  the Harish-Chandra realization \cite{HC}  and  the   symplectomorphism   given by Di Scala-Loi \cite{DL} and Roos \cite{Roos}, respectively.  Moreover, we mention the relation to the result by Loi-Mossa \cite{LM}. Throughout this appendix, we assume $M=G/K$ is a HSSNT of rank $r$ and denote the Cartan decomposition by $\fg=\fk\oplus \fp$. We follow the notations given in Section \ref{Prelim}.  In the following, we denote the complexification of a real vector space $V$ by $V_{\C}=V\oplus \sqrt{-1}V$.

\subsection{Harish-Chandra realization}  \label{A1}
First, we briefly recall the Harish-Chandra realization of HSSNT. Throughout this subsection, we use several results proved in \cite[\S 7 of Ch. VIII]{Hel}.   See also \cite[Section 2]{BIW} and \cite{Koranyi} for a nice summary.

Let $\fg_{\C}$ be the complexification of $\fg$. Take a maximal abelian subalgebra $\ft$ of $\fk$.  Then, the center $\fc(\fk)$ of $\fk$ is contained in $\ft$ and $\ft$ is a maximal abelian subalgebra of $\fg$. Moreover, $\ft_{\C}$ is a Cartan subalgebra of $\fg_\C$ (see Proof of Theorem 7.1 in \cite[Ch. VIII]{Hel}), and we have a root space decomposition with respect to $\ft_\C$:
\begin{align*}
\fg_\C=\ft_\C\oplus \bigoplus_{\chi\in \Delta}\fg_{\chi},
\end{align*}
where $\Delta$ is the set of non-zero roots, and $\fg_{\chi}:=\{Z\in \fg_\C: ({\rm ad}T)Z=\chi(T)Z\ \forall T\in\ft_\C\}$ is the root space.  Note that ${\rm dim}_{\C}\fg_
{\chi}=1$ for any $\chi\in \Delta$ (see \cite[\S 4 of Ch. III]{Hel}). Since $B$ is non-degenerate on the Cartan subalgebra $\ft_\C$, there exists a unique vector $T_{\chi}$ which is a dual of $\chi$ with respect to $B$. Then, we have $[\fg_{\chi}, \fg_{-\chi}]=\C T_{\chi}$, and there exists a basis $E_{\chi}$ of $\fg_{\chi}$ so that 
\begin{align*}
E_{\chi}-E_{-\chi},\quad \sqrt{-1}(E_{\chi}+E_{-\chi})\in \fk\oplus\sqrt{-1}\fp,\quad[E_{\chi}, E_{-\chi}]=\frac{2T_{\chi}}{\chi(T_{\chi})}.\label{defx1}
\end{align*}

Since each $\chi\in \Delta$  is real valued on $\sqrt{-1}\ft$, we define an ordering of $\Delta$ associated with the ordered real vector space $(\sqrt{-1}\ft)^*$. Moreover, we may assume that the ordering is compatible with $\sqrt{-1}\fc(\fk)\subset \sqrt{-1}\ft$, that is, $\alpha|_{\sqrt{-1}\ft}$ is positive whenever $\alpha|_{\sqrt{-1}\fc(\fk)}$ is positive. We denote the set of positive roots by $\Delta_+$.

For each $\chi\in \Delta$, the root space $\fg_{\chi}$ is contained in either $\fk_{\C}$ or $\fp_\C$. We put $\Delta_{\fp}:=\{\chi\in \Delta:\ \fg_{\chi}\subset\fp_{\C}\}$, and the element in $\Delta_{\fp}$ is so called {\it noncompact root} (otherwise, we say {\it compact root}).   
We set $Q_+:=\Delta_+\cap \Delta_{\fp}$, and put
\[
\fp_+=\bigoplus_{\chi\in Q_+}\fg_{\chi},\quad \fp_-=\bigoplus_{-\chi\in Q_+}\fg_{\chi}.
\]
Then, it turns out that  $\fp_+$ and $\fp_-$ are abelian subalgebras of $\fg^{\C}$ and  we have
\begin{align}\label{splitp}
[\fk_\C, \fp_+]\subset \fp_+,\quad [\fk_\C, \fp_-]\subset \fp_-,\quad \fp_\C=\fp_+\oplus \fp_{-}
\end{align}
(\cite[Proposition 7.2 in Ch. VIII]{Hel}).

Let ${\bf G}$ be a simply connected Lie group with Lie algebra $\fg_{\C}$ as a real Lie algebra. We denote the connected Lie subgroups of ${\bf G}$ with Lie algebras $\fk_{\C}$, $\fp_{+}$ and $\fp_{-}$ by ${\bf K}, P_+$ and $P_{-}$, respectively.  It turns out that ${\bf K}P_{+}$ is a subgroup of ${\bf G}$ by \eqref{splitp}, and furthermore, it is closed in ${\bf G}$.  Since ${\bf G}$, ${\bf K}$ and $P_+$ are complex Lie groups,  the coset space ${\bf G}/{\bf K}P_+$ inherits a ${\bf G}$-invariant complex structure. Moreover, a natural correspondence 
\[
f: M^*=G^*/K\to {\bf G}/{\bf K}P_+,\quad g^*K\mapsto g^*{\bf K}P_+
\]
gives rise to a holomorphic diffeomorphism from the compact dual $M^*$ of $M$, and it also follows that $M=G/K$ is holomorphically embedded into ${\bf G}/{\bf K}P_+\simeq M^*$ as a $G$-orbit through the origin, namely, the map 
\[
b: M=G/K\to {\bf G}/{\bf K}P_+,\quad gK\mapsto g{\bf K}P_+
\]
is a holomorphic embedding onto an open subset in ${\bf G}/{\bf K}P_+$(\cite[Proposition 7.14 in Ch.VIII]{Hel}).  The map $b$ is referred as {\it Borel embedding} of $M$. By using the identification $f: M^*\simeq{\bf G}/{\bf K}P_+ $, we say also the map $f^{-1}\circ b: M\to M^*$ the Borel embedding.

On the other hand,  the map 
\[
\xi:  \fp_{-}\to {\bf G}/{\bf K}P_+,\quad X\to (\exp X){\bf K}P_+
\]
gives a holomorphic diffeomorphism onto an open dense subset $U$  in ${\bf G}/{\bf K}P_+$ such that $b(M)\subset U=\xi(\fp_-)$  (\cite[Theorem 7.16 in Ch. VIII]{Hel}). By denoting  $f^{-1}(U)=(M^*)^o\subset M^*$, we obtain two holomorphic embeddings
\[
\psi:=\xi^{-1}\circ b: M\to \fp_-,\quad \psi^*:=\xi^{-1}\circ f|_{(M^*)^o}: (M^*)^o\to \fp_{-}.
\]
Then, the image $D_{-}:=\psi(M)$ is a bounded domain in $\fp_{-}$, and the map $\psi$ is so called the {\it Harish-Chandra realization} of $M$.

We shall give more precise description of $\psi$ and $D_{-}$.
It is known that there exists a subset $\{\mu_1,\ldots, \mu_r\}$ of $Q_+$, where $r$ is the rank of  $M$, such that they consist of {\it strongly orthogonal roots}, i.e., $\mu_i\pm \mu_j\notin \Delta$ for any $i,j=1,\ldots, r$. We simply denote the vectors $E_{\pm\mu_i}$  and $T_{\mu_i}$ by $E_{\pm i}$ and $T_i$, respectively. We put
\begin{align*}
\fa:=\bigoplus_{i=1}^r \mathbb{R}V_i,\quad {\rm where}\ V_i:=E_i+E_{-i}.
\end{align*}
Then, $\fa$ becomes a maximal abelian subspace of $\fp$ (\cite[Corollary 7.6 in Ch. VIII]{Hel}).  
Moreover, the complex Lie subalgebra $\fl_i:=\C E_{i}\oplus \C E_{-i}\oplus \C [E_i, E_{-i}]$ is isomorphic to $\mathfrak{sl}(2,\C)$, and we have a decomposition
\begin{align}\label{Eq:exp}
\exp(x_iV_i)=\exp((\tanh x_i)E_{-i})\exp((\log \cosh x_i)[E_i, E_{-i}])\exp( (\tanh x_i)E_{i})
\end{align}
for any $x_i\in \R$ and each $i=1,\ldots, r$ (see \cite[Lemma 7.11 in Ch. VIII]{Hel}). In particular, for any element $v=\sum_{i=1}^r x_i V_i\in \fa$,   $\exp v$ is expressed by $\exp v=\exp v_{-}\exp w\exp v_{+}$ with  $v_{-}\in \fp_{-}$, $w\in \fk^{\C}$ and $v_{+}\in \fp_{+}$, and $v_{-}$ is given by
\[
v_-=\sum_{i=1}^r (\tanh x_i) E_{-i}\in \fp_{-}.
\]
Thus, we have
$
\psi(\exp v \cdot o)=\xi^{-1}((\exp v_{-}) \cdot o^*)=v_{-}=\sum_{i=1}^r (\tanh x_i) E_{-i}.
$
Since $\psi$ is $K$-equivariant, the Harish-Chandra realization is expressed by
\begin{align*}
&\psi(k\exp v \cdot o)={\rm Ad}(k)\sum_{i=1}^r (\tanh x_i) E_{-i}
\end{align*}
for $k\in K$, $v=\sum_{i=1}^r a_i V_i\in \fa$. In particular, we have
\begin{align*}
&D_-={\rm Ad}(K)(\square_{-}),\ {\rm where}\ \square_{-}=\Big{\{}\sum_{i=1}^ry_i E_{-i}: |y_i|<1\Big{\}}.
\end{align*}

Now, we adapt the above description to our description given in the present paper. 
Recall that the complex structure $J_o$ on $\fp$ is defined by $J_o={\rm ad}(\zeta)|_{\fp}$ for some element $\zeta\in \mathfrak{c(k)}\subset \ft_\C$. We may assume that $\mu_i(\zeta)=-\sqrt{-1}$ so that 
\[
J_oE_i=-\sqrt{-1}E_i\quad {\rm and}\quad J_oE_{-i}=\sqrt{-1}E_{-i}
\]
(\cite[Corollary 7.13 in Ch.VIII]{Hel}).

We first show the following correspondence:

\begin{lemma}\label{Lemma:vh}
The basis $\{V_i\}_{i=1}^r$ coincides with  the set of normalized root vectors $\{\wH_i\}_{i=1}^r$ given in Section \ref{subsec:polydisk}, up to sign of each vector. 
\end{lemma}
\begin{proof}
By Proposition \ref{Prop:canbasis}, it suffices to show that $\{V_i\}_{i=1}^r$ satisfies the bracket relations \eqref{abra}. Indeed, since  $J_oV_i={\rm ad}(\zeta)(E_i+E_{-i})=-\sqrt{-1}(E_i-E_{-i})$, we see
\begin{align*}
[V_i, J_oV_i]=[E_i+E_{-i},-\sqrt{-1}(E_i-E_{-i})]=2\sqrt{-1}[E_i, E_{-i}]=\frac{4\sqrt{-1}}{\mu_i(T_i)}T_i.
\end{align*}
By using the fact $T_i\in \ft$, we see
\begin{align*}
[[V_i, J_oV_i], J_oV_i]&=\frac{4}{\mu_i(T_i)}[T_i,E_i-E_{-i}]=4V_i,\\
[[V_i, J_oV_i], V_i]&=\frac{4\sqrt{-1}}{\mu_i(T_i)}[T_i,E_i+E_{-i}]=-4J_oV_i
\end{align*}
for any $i=1,\ldots,r$.  Otherwise, we easily see that $[[V_i, J_oV_j], J_oV_k]=[[V_i, J_oV_j], V_k]=0$ by the strong orthogonality of roots. This proves $\{V_i\}_{i=1}^r$ satisfies \eqref{abra} as required.
\end{proof}

Next lemma gives an natural isomorphism between $\fp$ and $\fp_{-}$ (resp. $\fp^*=\sqrt{-1}\fp$ and $\fp_{-}$):

\begin{lemma}
The differential $d\psi: \fp\to \fp_-$ at $o\in M$ and $d\psi^*: \fp^*\to \fp_{-}$ at $o^*\in (M^*)^o$ are linear  isomorphisms, and they are explicitly given by
\[
d\psi(X)=\frac{1}{2}(X-\sqrt{-1}J_oX)\quad {\rm and}\quad d\psi^*(X^*)=\frac{1}{2}(X^*-\sqrt{-1}J_oX^*)
\] 
for $X\in \fp$, where $X^*:=\sqrt{-1}X$. In particular we have $d\psi^*(\sqrt{-1}X)=\sqrt{-1}d\psi(X)$.
\end{lemma}

See \cite[Corollary 7.13 in Ch.VIII]{Hel} for a proof. Note that the proof of the formula for $d\psi^*$  is given by a similar argument  for $d\psi$.

We also define a linear isomorphism between $\fp$ and $\fp^*$ by
\[
\iota:=(d\psi^*)^{-1}\circ d\psi: \fp\to \fp^*,\quad X\mapsto -J_oX^*.
\]

The following proposition shows that the holomorphic diffeomorphism $\Psi$ (resp. $\Psi^*$) coincides with $\psi$ (resp. $\psi^*$)  under the identification $d\psi: \fp \to \fp_{-}$ (resp. $d\psi^*: \fp^* \to \fp_{-}$).
\begin{proposition}\label{Prop:HC}
\begin{enumerate}
\item Let $\Psi:M\to \fp$ be the map constructed in Theorem \ref{mainthm1} and $\psi: M\to \fp_-$ the Harish-Chandra realization.  Then, we have $\Psi=d\psi^{-1}\circ \psi$.
\item  Let $\Psi^*: (M^*)^o\to \fp^*$ be the dual map of $\Psi$ given in Theorem \ref{mainthm2}. Then, we have 
$\Psi^*=(d\psi^*)^{-1}\circ \psi^*$. 
\end{enumerate}
In particular, the Borel embedding $f^{-1}\circ b: M\to M^*$ is given by 
\[
f^{-1}\circ b=(\psi^*)^{-1}\circ \psi=(\Psi^*)^{-1}\circ \iota\circ \Psi,
\]
namely,  we have the following commutative diagram:
\[
\xymatrix@C=25pt@R=25pt{
M\  \ar[r]_-b  \ar@/^16pt/[rr]_{\psi} \ar@/^30pt/[rrr]_{\Psi} \ar[d]_-{f^{-1}\circ b} &\ ({\bf G}/{\bf K}P_+)^o\  \ar[r]_-{\xi^{-1}} \ar[d]_-{id}&\  \fp_{-}\  \ar[r]_-{d\psi^{-1}} \ar[d]_-{id} & \ \fp \ar[d]_-{\iota}\\
(M^*)^o\  \ar[r]^-{f} \ar@/_16pt/[rr]^{\psi^*} \ar@/_30pt/[rrr]^{\Psi^*} &\ ({\bf G}/{\bf K}P_+)^o\ \ar[r]^-{\xi^{-1}}&\  \fp_{-}\ \ar[r]^-{(d\psi^*)^{-1}} & \ \fp^*
}
\]
\end{proposition}

\begin{proof}
(i) It is sufficient to show that $d\psi\circ \Psi|_{A}=\psi|_{A}$ for $A=\Exp_o\fa$. We consider the maximal abelian subspace $\fa$ spanned by $\{V_i\}_{i=1}^r$.  By Lemma \ref{Lemma:vh}, we may assume $V_i$ coincides with either $\wH_i$ or $-\wH_i$ for each $i$. We use the notation $V_i=\epsilon_i\wH_i$, where $\epsilon_i=1$ or $-1$.  Then, any element $v$ in $\fa$ is expressed by $v=\sum_{i=1}^r x_i V_i=\sum_{i=1}^r (\epsilon_i x_i )\wH_i$, and hence, by using \eqref{hdiff2}, we see
\begin{align*}
d\psi\circ \Psi(\Exp_ov)=d\psi\Big(\sum_{i=1}^r \tanh(\epsilon_i  x_i) \wH_i\Big)=d\psi\Big(\sum_{i=1}^r (\tanh x_i) V_i\Big)=\sum_{i=1}^r (\tanh x_i) E_{-i},
\end{align*}
where we used the fact $d\psi(V_i)=E_{-i}$. This proves (i).

(ii) We shall show $d\psi^*\circ\Psi^*|_{A^*}=\psi^*|_{A^*}$ for $A^*=\Exp_o^*\fa^*$, where $\fa^*=\sqrt{-1}\fa$. Since  $v^*\in \fa^*$ is expressed by $v^*=\sqrt{-1}\sum_{i=1}^r x_iV_i$,  \eqref{Eq:exp} implies that 
\[
\psi^*(\exp v^*\cdot o^*)=\sqrt{-1}\sum_{i=1}^r (\tan x_i)E_{-i}
\]
if $\exp v^*\cdot o^*\in (M^*)^o$. On the other hand, we have $v^*=\sum_{i=1}^r(\epsilon_i  x_i)\wH_i^*$ with $\wH_i^*=\sqrt{-1}\wH_i$, and hence,
\[
d\psi^*\circ \Psi^*(\Exp_{o^*}^* v^*)=d\psi^* \Big(\sum_{i=1}^r \tan (\epsilon_i  x_i) \wH_i^*\Big)=d\psi^* \Big(\sum_{i=1}^r (\tan x_i) V_i^*\Big)=\sqrt{-1}\sum_{i=1}^r (\tan x_i)E_{-i}
\]
since $d\psi^*(V_i^*)=\sqrt{-1}d\psi(V_i)=\sqrt{-1}E_{-i}$.
This proves (ii).

Then, one easily checks that the above commutative diagram holds.  This proves the proposition.
\end{proof}

\begin{remark}\label{Rem:HC}
{\rm 
In Remark \ref{Rem:holsymp}, we define a holomorphic embedding $h: M\to (M^*)^o$ by $h:=(\Psi^*)^{-1}\circ \sqrt{-1}\circ \Psi$.  It is easy to see that we have $f^{-1}\circ b=F\circ h$, where $F:=(\Psi^*)^{-1}\circ (-J_o)\circ \Psi^*: (M^*)^o\to (M^*)^o$ is a holomorphic diffeomorphism on $(M^*)^o$. Note that the difference between $h$ and $f^{-1}\circ b$ is caused by the difference of identifications between $\fp$ and $\fp^*$. 

}
\end{remark}

\subsection{Di Scala-Loi-Roos's formula} \label{A2}
In this subsection,  we consider the symplectomorphism given by  Di Scala-Loi and Roos. See \cite{DL, Roos} for details.  Let $M$ be a HSSNT.  It is known that $M$ is associated with a Hermitian positive Jordan triple system $(T_oM, \{, ,\})$, which is defined by
\[
\{u,v,w\}=-\frac{1}{2}(R_o(u,v)w+J_oR_o(u, J_ov)w)
\]
for $u,v,w\in T_oM$, where $R_o$ and $J_o$ are the curvature tensor and the complex structure of $M$ at the origin, respectively. The Bergman operator $B$ is defined by 
\[
B(u,v)={\rm Id}- D(u,v)+Q(u)Q(v),
\]
where  $D$ is an operator on $\fp$ defined by $D(u,v)(w)=\{u,v, w\}$ and $Q$ is the quadratic representation which is defined by $2Q(u)(v)=\{u,v,u\}$. 

In the following, we identify $T_oM$ with $\fp$, and we regard $\fp$ as a Jordan triple system.
 Moreover, we identify $M$ with a bounded domain $D=\Psi(M)\subset \fp$ by the map $\Psi: M\to \fp$. Note that $\Psi$ is regarded as the Harish-Chandra realization (Proposition \ref{Prop:HC}). In \cite{DL}, Di-Scala-Loi proved that the following map
\begin{align}\label{Def:DLmap}
\widetilde{\Phi}: D\to \fp,\quad \widetilde{\Phi}(z):=B(z,z)^{-\frac{1}{4}}z,
\end{align}
 becomes a $K$-equivariant symplectomorphism between $(D, \Psi^{-1}\omega)$ and $(\fp, \omega_o)$. Our purpose of this subsection is to show the following relation:

\begin{proposition}\label{Prop:DLmap}
Let $\Phi: M\to \fp$ be the symplectomorphism constructed in Theorem \ref{mainthm1} and $\widetilde{\Phi}: D\to \fp$ the map given by \eqref{Def:DLmap}. Then, we have $\Phi\circ \Psi^{-1}=\widetilde{\Phi}$.
\end{proposition}

 To prove this, we recall some basic notions in Jordan triple systems. An element $c\in \fp$ is said to be {\it tripotent} if $\{c, c, c\}=2c$, and we say two elements $c_1, c_2\in \fp$ are {\it orthogonal} if $D(c_1,c_2)=0$.   It is known that any element $z\in \fp$ has a  decomposition 
\[
z=\lambda_1c_1+\cdots +\lambda_pc_p
\]
so called the {\it spectral decomposition}, where $\lambda_1>\lambda_2>\cdots>\lambda_p>0$ and $\{c_1,\ldots, c_p\}$ is a set of mutually orthogonal tripotents.  In terms of the restricted root system, the set of orthogonal tripotents coincides with the set of strongly orthogonal root vectors (up to sign of each vector):

\begin{lemma}\label{Lem:trip}
Let $z$ be any element in $\fp$, $\fa$ a maximal abelian subspace in $\fp$ containing $z$, and $\{\wH_i\}_{i=1}^r$ the set of normalized root vectors with respect to $\fa$ given in Section \ref{subsec:polydisk}. Then, any subset of $\{\wH_i\}_{i=1}^r$ becomes a set of  mutually orthogonal tripotents.
\end{lemma}

\begin{proof}
First, we show $\wH_i$ is a tripotent for each $i$. Since the curvature tensor $R_o$ at the origin is given by $R_o(u,v)w=-[[u,v],w]$, we see by \eqref{brawh} that
\begin{align*}
\{\wH_i, \wH_i, \wH_i\}=\frac{1}{2}J_o[[\wH_i, J_o\wH_i], \wH_i]=J_o[[\wH_i, X_i^{\fp}], \wH_i]=2J_o[X_i^{\fk}, \wH_i]=-4J_oX_i^\fp=2\wH_i,
\end{align*}
as required. Next, for any $i\neq j$, we have $[\wH_i, \wH_j]=0$ and $[\wH_i, J_o\wH_j]=2[\wH_i, X_j^{\fp}]=2\gamma_j(\wH_i)X_i^{\fk}=0$ since $\fa$ is abelian and $\{\wH_i\}_{i=1}^r$ is an orthogonal basis. Therefore, we obtain  $D(\wH_i, \wH_j)=0$ for any $i\neq j$, and  $\{\wH_i\}_{i=1}^r$ is a set of mutually orthogonal tripotents.
\end{proof}

We give a proof of Proposition  \ref{Prop:DLmap}.

\begin{proof}[Proof of  Proposition  \ref{Prop:DLmap}]
Fix arbitrary $z\in \fp$. We may assume $z$ is contained in a maximal abelian subspace $\fa$. Let $\{\wH_i\}_{i=1}^r$ be the strongly orthogonal root vectors with respect to $\fa$. As described in \cite[Equation (26) in Section 6]{DL}, the symplectomorphism $\widetilde{\Phi}$ is expressed by
\[
\widetilde{\Phi}(z)=\sum_{i=1}^p\frac{\lambda_i}{\sqrt{1-\lambda_i^2}}c_i
\]
with respect to the spectral decomposition $z=\sum_{i=1}^p\lambda_ic_i$. By Lemma \ref{Lem:trip}, the spectral decomposition can be written by $z=\sum_{i=1}^p\epsilon_i\lambda_i\wH_i$, where $\epsilon_i$ is equal to either $1$ or $-1$ which is determined by the relation $c_i=\epsilon_i\wH_i$. By putting $x_i=\epsilon_i\lambda_i$,  we have $z=\sum_{i=1}^p x_i\wH_i$ and $\widetilde{\Phi}$ is expressed by
\[
\widetilde{\Phi}(z)=\sum_{i=1}^p\frac{x_i}{\sqrt{1-x_i^2}}\wH_i.
\]
Thus, by using the formulas in \eqref{hdiff2}, we see
\[
\Phi\circ \Psi^{-1}(z)=\sum_{i=1}^p \sinh(\tanh^{-1}x_i)\wH_i=\sum_{i=1}^p\frac{x_i}{\sqrt{1-x_i^2}}\wH_i=\widetilde{\Phi}(z),
\]
as required.
\end{proof}

\subsection{Loi-Mossa's diastatic exponential map} \label{A3}

Regarding $M$ as a bounded domain $D$ via $\Psi$ and $\fp$ as the associated Jordan triple system,  
  Loi-Mossa \cite{LM} defined the following map ${\rm DE}_o:\fp\to D$ (\cite[eq. (11)]{LM}):
\begin{align*}
{\rm DE}_o(z):=\sum_{i=1}^p \sqrt{1-e^{-\lambda_i^2}}\cdot c_i
\end{align*}
 with respect to the spectral decomposition $z=\sum_{i=1}^p \lambda_ic_i\in D$.  Note that ${\rm DE}_o: \fp\to D$ is a diffeomorphism, and Loi-Mossa proved that ${\rm DE}_o$ is the {\it diastatic exponential map} for $D\simeq M$ at the origin $o\in D$ (\cite[Theorem 1]{LM}). We denote the inverse map by ${\rm DL}_o: D\to \fp$, which is given by
\[
{\rm DL}_o(z)=\sum_{i=1}^p \sqrt{-\log(1-\lambda_i^2)}\cdot c_i
\]
for $z=\sum_{i=1}^p \lambda_ic_i\in D$.  Our aim here is to confirm that ${\rm DL}_o$ is also constructed by a strongly diagonal realization $\Omega_{\eta}$:

\begin{proposition}
Define an injective odd function $\eta:\R\to \R$ by 
\begin{align}\label{Def:etaLM}
\eta(x):=
\begin{cases}
\sqrt{\log \cosh^2x} & x\geq 0\\
-\sqrt{\log \cosh^2x} & x<0\\
\end{cases}.
\end{align}
 Then, we have 
$
\Omega_{\eta}\circ \Psi^{-1}={\rm DL}_o.
$
\end{proposition}

\begin{proof}
The spectral decomposition is written by $z=\sum_{i=1}^p \epsilon_i\lambda_i \wH_i$, where $\epsilon_i=1$ or $-1$ and is determined by the relation $c_i=\epsilon_i\wH_i$. We put $x_i=\epsilon_i\lambda_i$ so that $z=\sum_{i=1}^p x_i\wH_i$. Then we have
\[
{\rm DL}_o(z)=\sum_{i=1}^p \sqrt{-\log(1-x_i^2)}\cdot \frac{x_i}{|x_i|}\wH_i
\]
with ${\rm DL}_o(0)=0$.
On the other hand, we see
\[
\Omega_{\eta}\circ \Psi^{-1}(z)=\sum_{i=1}^p \eta(\tanh^{-1}x_i)\wH_i=\sum_{i=1}^p \sqrt{-\log(1-x_i^2)}\cdot \frac{x_i}{|x_i|}\wH_i
\]
since 
\[
\eta(x)=\sqrt{\log(\cosh^2x)}\cdot \frac{x}{|x|}=\sqrt{-\log(1-\tanh^2x)}\cdot \frac{x}{|x|}
\]
with $\eta(0)=0$.
This proves the lemma.
\end{proof}

We shall consider the dual map of $\Omega_{\eta}$.  The dual function $\eta^*: (-\pi/2, \pi/2)\to \R$ of \eqref{Def:etaLM} is given by
\begin{align*}
\eta^*(x):=
\begin{cases}
\sqrt{-\log \cos^2x} & x\in [0, \pi/2)\\
-\sqrt{-\log \cos^2x} & x\in (-\pi/2, 0)\\
\end{cases}
\end{align*}
with $\eta(0)=0$. This function is formally obtained by 
$
\eta^*(x)=-\sqrt{-1}\cdot \eta(\sqrt{-1}x).
$
Then, the dual map $\Omega_{\eta^*}^*$ is a map from $(M^*)^o=M^*\setminus {\rm Cut}_o(M^*)$ onto $\fp^*=\sqrt{-1}\fp$.

Let us identify $(M^*)^o$ with $\fp^*$ by the dual map $\Psi^*$ of $\Psi$.  We put ${\rm DL}_o^*:=\Omega_{\eta^*}^*\circ (\Psi^*)^{-1}: \fp^*\to \fp^*$. Then, for any $z=\sum_{i=1}^r x_i\wH_i^*\in \fp^*$, we obtain 
\[
{\rm DL}_o^*(z)=\sum_{i=1}^r \eta(\tan^{-1}x_i)\wH_i^*=\sum_{i=1}^r \sqrt{\log(1+x^2_i)}\cdot \frac{x_i}{|x_i|} \wH_i^*
\]
with $\eta^*(0)=0$ since 
\[
\eta^*(x)=\sqrt{-\log(\cos^2x)}\cdot \frac{x}{|x|}=\sqrt{\log(1+\tan^2x)}\cdot \frac{x}{|x|}.
\]

By using the spectral decomposition, we see
\[
{\rm DL}_o^*(z)=\sum_{i=1}^p \sqrt{\log(1+\lambda^2_i)}\cdot c_i,
\]
for $z=\sum_{i=1}^p \lambda_i c_i$,
and the inverse map is given by
\[
{\rm DE}_o^*(z)=\sum_{i=1}^p \sqrt{e^{\lambda_i^2}-1}\cdot c_i.
\]
This recovers the map defined in \cite[eq. (17)]{LM}.  Loi-Mossa proved that ${\rm DE}_o^*(z)$ is the diastatic exponential map for the compact dual $M^*$. See \cite[Theorem 2]{LM}.

\subsection*{Acknowledgements} 
The authors would like to thank Professor Hiroyuki Tasaki for suggesting the idea of construction and stimulating discussion. T.K. would like to thank Professor Antonio J. Di Scala for explaining the results on \cite{DL, DLR} to him and for helpful comments.
T.H. is supported by JSPS KAKENHI Grant Number JP16K17603. T.K. is supported by JSPS KAKENHI Grant Number JP18K13420.

\end{document}